\documentclass[11pt]{amsart}
\usepackage{amsmath, amsthm, amssymb}
\usepackage{amsmath,amscd}
\usepackage{mathabx}
\usepackage{hyperref}
\usepackage{mathrsfs}
\usepackage{xcolor, changepage}
\usepackage{graphicx}
\usepackage{titletoc}
\usepackage{enumerate}
\usepackage{accents}
\usepackage{romanbar}

\usepackage{tikz}
\usetikzlibrary{matrix,arrows}
\usetikzlibrary{shapes}
\usetikzlibrary{calc}
\usetikzlibrary{arrows}
\usetikzlibrary{decorations.pathreplacing,decorations.markings}
\usepackage[all]{xy}
\usepackage{tikz-cd}

\usepackage{caption}
\usepackage{subcaption}
\usepackage{geometry}
\usepackage{comment}

\theoremstyle{plain}
\newtheorem{theorem}{Theorem}
\newtheorem{proposition}[theorem]{Proposition}
\newtheorem{lemma}[theorem]{Lemma}

\numberwithin{theorem}{section}

\theoremstyle{remark}
\newtheorem{remark}[theorem]{Remark}

\theoremstyle{definition}
\newtheorem{definition}[theorem]{Definition}
\newtheorem{example}[theorem]{Example}
\newtheorem{construction}[theorem]{Construction}

\newcommand{\R}{\mathbb{R}}
\newcommand{\cH}{\mathcal{H}}
\newcommand{\fs}{\mathfrak{s}}
\newcommand{\cA}{\mathcal{A}}
\newcommand{\eps}{\varepsilon}

\DeclareMathOperator{\Lip}{Lip}
\DeclareMathOperator{\Isom}{Ismo}

\DeclareMathOperator{\Int}{Int}

\DeclareMathOperator{\Ric}{Ric}

\DeclareMathOperator{\Area}{Area}
\DeclareMathOperator{\sys}{sys}

\DeclareMathOperator{\id}{id}

\DeclareMathOperator{\cor}{cor}
\DeclareMathOperator{\fac}{fac}
\DeclareMathOperator{\Id}{Id}

\title[Capillary prisms, Coxeter gluing, and $\pi_2$-systole]{Capillary prisms, Coxeter gluing, and the $\pi_2$-systole of $3$-manifolds with positive scalar curvature}
\author{Shihang He}
\address{Key Laboratory of Pure and Applied Mathematics,
School of Mathematical Sciences, Peking University, Beijing, 100871, P. R. China
}
\email{hsh0119@pku.edu.cn, sh9035@nyu.edu}
\author{Chao Li}
\address{Courant Institute, New York University, 251 Mercer St, New York, NY
10012, USA
}
\email{chaoli@nyu.edu}
\begin{document}

\begin{abstract}
    We study the $\pi_2$-systole of positive scalar curvature $3$-manifolds via a local-to-global approach. The tessellated local building blocks are capillary prisms $M$:  Riemannian cylinders enclosed by mean convex surfaces meeting at prescribed dihedral angles. For a class of such prisms, called energy-essential prisms, we obtain angle-sensitive relative $\pi_2$-systole estimates. In particular, if the upper capillary angle is at most $\alpha \in (0,\tfrac{\pi}{2}]$, then
    \[\sys_2(M,\partial_0 M, g)\cdot \inf R_g \lesssim |\log \alpha|^{-1},\qquad \alpha\to 0.\]
    The proof relies on the monotonicity of a new quasi-local mass.

    The local-to-global step is achieved via Coxeter gluing of copies of capillary prisms. We prove a general Coxeter gluing and smoothing theorem for scalar curvature: under Coxeter compatibility of the corner strata and nonnegative mean curvature jump conditions along the glued facets, we prove that the resulting piecewise smooth manifold can be smoothened while preserving the scalar curvature lower bound. This theorem gives, to our knowledge, the first general rigorous formulation and proof of the Coxeter-polyhedral smoothing principle for scalar curvature lower bounds, a principle that has long been used heuristically in positive scalar curvature geometry.

    As an application, we construct smooth positive scalar curvature metrics with large $\pi_2$-systole on connected sums of copies of $S^2\times S^1$ and lens spaces that are not covered by $S^2\times \R$. The scale-invariant systolic constants obtained include $4\pi$ for connected sums of $S^2\times S^1$ and 
    \[
    \min\left\{
        4\pi,\,
        \min_j 8\pi\left(1-\cos\frac{\pi}{p_j}\right)
    \right\}
    \]
    for connected sums with lens-space $L(p_j, q_j)$. These examples give the first explicit metrics with quantitatively large $\pi_2$-systole on these topologies. We also discuss rigidity and non-rigidity phenomena for the local $\pi_2$-systolic estimates.
\end{abstract}
\maketitle
	\tableofcontents	

\section{Introduction}

\begin{comment}
    \begin{definition}
A Riemannian cylinder  $(M^n,g)$ is a compact Riemannian manifold with corner, whose boundary has a decomposition $\partial M = \partial_-M\cup \partial_0 M\cup\partial_+ M$, satisfying
\begin{itemize}
    \item $\partial_- M\cap\partial_+ M = \emptyset$;
    \item Both $\partial_{\pm}M$ meet $\partial_0 M$ in a smooth curve.
\end{itemize}
We denote the dihedral angle between $\partial_{\pm} M$ and $\partial_0 M$ by $\angle_{\pm0}(M,g)$, and the mean curvature of $\partial_0M$ with respect to the outward pointing normal of $M$ by $\bar{H}$.
\end{definition}
\end{comment}

Let $(M^3, g)$ be a closed oriented Riemannian $3$-manifold with nontrivial $\pi_2$. Its $\pi_2$-systole, denoted by $\sys_2(M,g)$, is defined as

\begin{definition}\label{defn: pi2 systole}
	\[\sys_2(M,g) = \inf\{\Area(S^2, f^*g): f: S^2\to M \text{ is homotopically nontrivial}\}.\]
\end{definition}

The $\pi_2$-systole is an important metric invariant of the $(M,g)$ that is intimately related to scalar curvature $R_g$,  especially when $(M,g)$ has positive scalar curvature (abbreviated as PSC in this article). Indeed, denoting by
\[\fs (M) = \sup\{\sys_2(M,g)\cdot \inf_M R_g: g\text{ is PSC}\},\]
a fundamental result by Bray-Brendle-Neves \cite{BBN10} proved that whenever $\pi_2(M)\ne 0$ and $R_g>0$, we have that
\begin{equation}\label{eq:BBN}
	\fs(M)\le 8\pi.
\end{equation}
The inequality \eqref{eq:BBN} has a rigidity phenomenon: equality holds if and only if $(M,g)$ is covered by a product manifold $S^2\times \R$. The systolic inequalities have been extensively studied in the scalar curvature context and various extensions of \eqref{eq:BBN} were obtained. For example, Bray-Brendle-Eichmair-Neves \cite{BBEN} studied the least area of nontrivial embedded $\R P^2$ and obtained a rigidity result modeled on the round $\R P^3$; Nunes \cite{Nunes} obtained sharp area bounds for locally area-minimizing high-genus surfaces under a negative lower bound of $R_g$, with rigidity modeled by $\Sigma\times \R$. A particularly intriguing result by Xu \cite{Xu25}, which largely motivated this work, reveals a topological gap phenomenon on $\sys_2$ for $3$-manifolds that are not covered by $S^2\times S^1$. In higher dimensions, Zhu \cite{Zhu2020rigidity,Zhu2023rigidity} proved sharp area upper bounds for nontrivial $S^2$ under a lower bound of $R$ with rigidity modeled on $S^2\times T^n$. 

In view of the rigidity theorem by Bray-Brendle-Neves \cite{BBN10} and Xu's topological gap theorem \cite{Xu25}, it is natural to ask how large $\fs(M)$ can be for individual PSC topologies. By the work of Schoen-Yau, Gromov-Lawson and Perelman, a closed orientable $3$-manifold $M$ admits a PSC Riemannian metric if and only if it decomposes into $S^2\times S^1$ and spherical space forms. The exact value of $\fs(M)$ is presently unknown except in the case when $M$ is covered by $S^2\times \R$. It is therefore natural to seek qualitative and quantitave relations between the topology of $M$ and the size of $\fs(M)$. For example, one expects that increasingly complicated spherical prime factors should force $\fs(M)$ to become small: if $\{M_j\}$ contains a prime factor $S^3/\Gamma_j$ with $|\Gamma_j|\to \infty$, then one should expect that $\fs(M_j)\to 0$. 

From the opposite direction, one would like to construct PSC metrics with a fixed scalar curvature lower bound and large $\pi_2$-systole. Prior to this work, no explicit constructions seem to have been available for metrics on such $3$-manifolds achieving quantitatively large $\pi_2$-systole.

We approach these questions on the connected sums  of copies of $S^2\times S^1$ and lens spaces via a local-to-global strategy \footnote{We also give potential future directions on how to handle other spherical space forms, see Section \ref{subsection: prism manifolds} and \ref{subsection: 3 special spaces}.}. We isolate a geometric configuration - a Riemannian capillary prism  with controlled dihedral angle - as the tessellated localization for these PSC $3$-manifolds. By appropriately prescribing different capillary angles, these prisms serve as Coxeter building blocks for $S^2\times S^1$ and lens spaces. We obtain relative $\pi_2$-systole estimates on all such capillary prisms, localizing and extending the results of Bray-Brendle-Neves \cite{BBN10} and Xu \cite{Xu25}. Conversely, by constructing capillary prisms with large relative $\pi_2$ systole, we obtain global metrics with large $\pi_2$-systole. We now lay out our local-to-global estimates and constructions in detail.

\subsection{Riemannian capillary prisms}
 
In the first part of this paper, we analyze Riemannian capillary prisms. They serve as fundamental building blocks for closed $3$-manifolds and will be the central geometric object of this paper. Let us focus on the local picture at this moment and defer the local-to-global construction till Section \ref{subsection: coxeter smoothing - intro}.

\begin{definition}
    A Riemannian cylinder  $(M^n,g)$ is a compact Riemannian manifold with corner admitting a diffeomorphism $\varphi: M\longrightarrow D^{n-1}\times [0,1]$. We denote the upper and bottom boundary by $\partial_+ M = \varphi^{-1}(D^{n-1}\times\{1\})$, $\partial_- M = \varphi^{-1}(D^{n-1}\times\{0\})$, and the lateral boundary by $\partial_0 M$. We denote the dihedral angle between $\partial_{\pm} M$ and $\partial_0 M$ by $\angle_{\pm0}(M,g)$, and the mean curvature of $\partial_0M$ with respect to the outward pointing normal of $M$ by $\bar{H}$.
\end{definition}

For a $3$-dimensional Riemannian cylinder $(M^3,g)$, we define the relative $\pi_2$-systole, $\sys_2(M,\partial_0 M, g)$, analogously as in Definition \ref{defn: pi2 systole}.

\begin{align*}
    \sys_2(M,\partial_0 M,g) = \inf\{|D^2|_{f^*g},\quad &f:(D^2,\partial D^2)\longrightarrow (M,\partial_0 M) \text{ is a smooth immersion},\\
    &\text{ and }f\text{ is not homotopically trivial}\}.
\end{align*}

The relative $\pi_2$-systole has been studied previously: assuming that $R_g>0$ and $\partial M$ is weakly mean convex, Ambrozio \cite{Ambrozio} proved that
\[\sys_2(M,\partial_0 M, g)\inf_M R_g \le 4\pi,\]
with rigidity modeled on $S^2_+\times \R$.

It turns out that the appropriate localization involves Riemannian cylinders that satisfies a capillary type boundary condition.

\begin{definition}\label{defn: capillary prism}
    Fix $\alpha\in (0,\tfrac{\pi}{2}]$. A capillary prism $(M^3,g)$ with angle at most $\alpha$ is a Riemannian cylinder satisfying
    \[\angle_{+0}(M,g)\le \alpha,  \quad \angle_{-0}(M,g) \le \frac{\pi}{2},\]
    and $\partial_\pm M, \partial_0 M$ are weakly mean convex with respect to the outward unit normal \footnote{In this paper, we adopt the convention that the unit sphere in $\R^3$ has mean curvature $=2$ with respect to the outward unit normal.}.
\end{definition}

\begin{example}\label{example: standard cylinder block - intro}
    Fix $\alpha\in (0,\tfrac{\pi}{2}]$. We introduce a capillary prisms with angle $=\alpha$ and $R_g = 2$, hereafter called a \textit{standard cylindrical block} $M_\alpha$. They are central examples throughout this paper. For different choices of $\alpha$, these will be the Coxeter building blocks for various $3$-manifolds. We delay the detailed construction to Section \ref{section: construction} and only describe their basic properties.

    $M_\alpha$ is a subset of the cylinder $S^2_+\times \R$, where $S^2_+$ stands for the upper hemisphere of radius $1$.  $\partial_\pm M_\alpha$ are contained in well-chosen slices $S^2_+\times \{t_\pm\}$, and hence are totally geodesic; $\partial_0 M_\alpha$ is a rotationally invariant minimal surface in $S^2_+\times \R$, meeting at constant angle $\alpha$ with $\partial_+ M_\alpha$  and orthogonally with $\partial_{-} M_\alpha$. The leaves $\{M_\alpha\cap (S^2_+\times \{t\})\}_{t\in [t_{-},t_+]}$ form a unit speed totally geodesic foliation of $M_\alpha$. See Figure \ref{figure: example 1}.

    The relative $\pi_2$-systole for $M_\alpha$ is achieved by $|\partial_{-}M_\alpha|$, and equals $2\pi(1-\cos\alpha)$.
\end{example}

As we will see in Section \ref{section: Coxeter gluing and smoothing} and \ref{section: construction}, a capillary prism with angle $\alpha\to 0$ should be thought of the localization of a spherical space form $S^3/\Gamma$ with $|\Gamma|\to \infty$. 

\subsection{Upper bounds for the relative $\pi_2$-systole}

In our first main result, we obtain an upper bound of the relative $\pi_2$-systole for a large class of capillary prisms $M$, modeled on Example \ref{example: standard cylinder block - intro}. To describe the assumption we require on $M$, we note that the maximum principle implies that for any $\beta\in (\alpha,\tfrac{\pi}{2})$, the minimizer to the capillary functional
\[\cA_\beta (\Omega) = \cH^2(\partial\Omega\cap \mathring M) - \cos \beta \cH^2(\partial \Omega\cap \partial_0 M) \]
exists. Moreover, any such minimizer $\Omega$ contains $\partial_{-}M$ and is disjoint from $\partial_+ M$. We are now ready to describe the class of prisms we consider.

\begin{definition}\label{defn: energy-energy essential}
    We say a capillary prism with angle $\alpha$ is \textit{energy-essential}, if for all $\beta\in [\alpha, \tfrac{\pi}{2}]$, any $\cA_\beta$-minimizer $\Omega$ is connected.
\end{definition}

In other words, the $\cA_\beta$-minimizer does not contain any relatively null-homologous open sets along $\partial_0 M$. Since $M$ is simply connected, this is equivalent to requiring that the capillary surface $\partial \Omega\cap\mathring M$ is connected. In Proposition \ref{prop: energy-essential}, we will give sufficient conditions for the energy-essential property:
\begin{itemize}
    \item $M$ admits a foliation by capillary surfaces of monotone angles from $\partial M_{-}$ to $\partial_+ M$.
    \item $\partial_0 M$ is area-minimizing.
    \item $M$ does not contain any stable capillary disk with angle $\beta\in (\alpha, \tfrac{\pi}{2})$ and boundary on $\partial_0 M$.
\end{itemize}
Moreover, we will prove in Proposition \ref{prop: energy-essential is open} that energy-essential is an open condition: it is preserved under small perturbations of $(M,g)$.

To describe our first result, we introduce an auxiliary function $\Psi(\alpha)$. For $\alpha\in (0,\frac\pi2]$, let $x = \Psi(\alpha)$ be the unique solution in $(0,1]$ to the following equation
\begin{align*}
    \frac{1}{x}+\log x = 1-\frac{1}{3}\log\sin\alpha.
\end{align*}
It is straightforward to check that $\Psi$ is an increasing function on $(0,\frac\pi2]$ with $\Psi(\frac{\pi}{2}) = 1$, and
\begin{align*}
    \Psi(\alpha) = O(|\log\alpha|^{-1}),\quad \alpha\to 0.
\end{align*}
In particular, $\lim_{\alpha\to 0}\Psi(\alpha) = 0$.

\begin{theorem}\label{thm: 2 systole estimate}
     Let $(M^3,g)$ be an energy-essential capillary prism with scalar curvature $R_g\ge R_0>0$ and angle at most $\alpha$, $\alpha\in (0,\frac{\pi}{2}]$. Then
    \begin{align*}
        \sys_2(M,\partial_0 M,g)\cdot R_0\le 4\pi\Psi(\alpha).
    \end{align*}
    In particular, the upper bound of $ \sys_2(M,\partial_0 M,g)$ tends to $0$ as $\alpha\to 0$.
\end{theorem}

When $\alpha = \tfrac{\pi}{2}$, our result recovers that of Ambrozio \cite{Ambrozio}. As $\alpha\to 0$, Theorem \ref{thm: 2 systole estimate} is the localization that the $\pi_2$-systole approaches $0$ on for a $3$-manifold containing $S^3/\Gamma$, as $|\Gamma|\to \infty$.

\begin{comment}
\begin{remark}
    We list approximate values of $\Psi(\alpha)$ for some special choices of $\alpha$:

\[
\begin{array}{c|c|c}
\alpha & \Psi(\alpha) & \text{Upper bound of } \sys_2\cdot R_0 \\ \hline
10^\circ & 0.400 & 1.600\pi \\
30^\circ & 0.543 & 2.172\pi \\
45^\circ & 0.641 & 2.563\pi \\
60^\circ & 0.745 & 2.980\pi
\end{array}
\]

\end{remark}
\end{comment}

Let us discuss the proof of Theorem \ref{thm: 2 systole estimate} as it is interesting in its own right. It is well-known that the relative systole is achieved by an embedded free boundary area-minimizing disk $\Sigma$. Plugging the constant test function into the stability inequality on $\Sigma$ immediately gives the upper bound $\sys_2\cdot R_0\le 4\pi$. However, this estimate is independent of the angle $\alpha$ and cannot be achieved unless $\alpha=\tfrac{\pi}{2}$.

We thus instead study a whole capillary free energy profile. For each angle $\beta\in (\alpha, \tfrac{\pi}{2})$, we minimize $\cA_\beta$ and obtain a minimizer $\Omega_\beta$. Denote by $\Sigma_\beta=\Omega_\beta\cap \mathring M$ be the capillary surface and  $S(\Sigma_\beta) = \partial_\beta \Omega\cap \partial_0 M$ the wetting region. For clarity of presentation, let us assume here that the surfaces $\{\Sigma_\beta\}_{\beta\in (\alpha,\frac{\pi}{2})}$ gives a foliation $\mathcal{F}$ of $M$. 

For each $s\in (0, |\partial_0 M|)$, we consider the free energy profile
\[I(s) = |\Sigma_s|,\]
where $\Sigma_s$ is the unique leaf in $\mathcal{F}$ such that $|S(\Sigma_s)|= s$. This profile was first considered (in a different context) by Eichmair-Koerber \cite{EK24}. The key property is that $I$ satisfies an ordinary differential inequality in the viscosity sense (see Proposition \ref{prop: differential inequality for I}):
\[I''(s) \le |\partial \Sigma_s|^{-2} (1-I'(s)^2) (2\pi - \frac12 R_0 I(s)). \]
The constant $2\pi$ is a consequence of the Gauss-Bonnet formula and is where we used the energy-essential assumption on $M$. We note here that the connectedness of minimizers has been a necessary assumption for many applications of geometric profile functions, most notably Bray's comparison results \cite{Bray97}. The work by Eichmair-Koerber \cite{EK24} managed to elegantly bypass this issue with an ingenious observation on the Euclidean isoperimetric inequality. Our case is similar to that of Bray \cite{Bray97} because of the PSC ambient geometry. See more discussions in Remark \ref{remark: connectedness is needed}.

We thus need an estimate of $|\partial \Sigma_s|$ in terms of $I(s)$, which reduces to typical relative isoperimetric inequality. We deduce such an estimate for $\Sigma_s$ as a consequence of the capillary stability:
\[|\partial \Sigma_s| \ge \frac{R_0}{\sqrt 6} |\Sigma_s|.\]
The combination of these two estimates leads to the monotonicity of a new quasi-local mass: defining
\[m(s) = \exp \left(-\frac{24\pi}{R_0 I(s)}\right)I(s)^{-6} (1-I'(s)^2),\]
we have that $m'(s)\ge 0$. From this we deduce the desired upper bound for $|\Sigma|$.

\subsection{A general Coxeter smoothing theorem}\label{subsection: coxeter smoothing - intro}

In the second part of this paper, we rigorously prove a general Coxeter gluing/smoothing theorem allowing polyhedral singularities, while maintaining scalar curvature lower bounds. The main result works for in all dimensions and general (i.e. not necessarily nonnegative) scalar curvature lower bounds.

To motivate the discussions, we explain how we use the standard cylindrical blocks as in Example \ref{example: standard cylinder block - intro} to build $(S^2\times S^1)\# (S^2\times S^1)$ with large $\pi_2$-systole. 

Fix $\alpha=\tfrac{\pi}{3}$ and take six copies of standard cylindrical blocks $M_i$, $i=1,\cdots, 6$, each isometric to $M_\alpha$. We identify facets of $M_i$ according to the relation
\[\partial_+ M_{2k} \sim \partial_+ M_{2k+1}, \quad \partial_0 M_{2k-1} \sim \partial_0 M_{2k}. \]

Here the subscripts are understood modulo $6$. The result of this gluing is a geometric realization of the `pant' $\mathcal{P}_3$ - a topological $3$-manifold that is diffeomorphic to $S^3$ with three disjoint $D^3$'s removed - equipped with a piecewise smooth Riemannian metric $g$. The singular set of $g$ is polyhedral: the codimension $1$ singular strata are minimal surfaces; at each interior point along the codimension $2$ strata of $(\mathcal{P}_3,g)$, the choice $\alpha = \tfrac{\pi}{3}$ guarantees a total cone angle of $2\pi$. Locally, each of these conditions individually is known to preserve the scalar curvature lower bounds, see, e.g. \cite{Miao2002,LM2019}. Moreover, the boundary of $(\mathcal{P}_3, g)$ are least area homotopically nontrivial $S^2$'s with area $=2\pi$. Finally, take a double of $(\mathcal{P}_3,g)$ along its boundary and we obtain a Lipschitz metric (which we still denote by $g$) on $(S^2\times S^1)\# (S^2\times S^1)$ with $\pi_2$-systole equal to $2\pi$. Our goal is thus to mollify $g$ to a smooth metric while preserving the same scalar curvature lower bound and leaving the $\pi_2$-systole almost unchanged. A similar construction with copies of $M_\alpha$ with $\alpha = \tfrac{\pi}{p}$ leads to a piecewise smooth, Lipschitz metric on $L(p,q)\setminus D^3$ with large $\pi_2$-systole, which we aim to smoothen. These constructions are described in detail in Section \ref{subsection: construction of S2xS1 and lens spaces}.

This type of local-to-global constructions involving the gluing/smoothing of polyhedral singularities have appeared multiple times in the PSC literature. For example, in his foundational paper \cite{Gromov2014Dirac}, Gromov proposed the well-known dihedral rigidity conjecture and sketched a proof for the case of cubes with an analogous gluing/smoothing idea (for a detailed discussion, see e.g. \cite[Introduction]{LM2019}). Similar strategies have also appeared in Gromov's more recent works, e.g. \cite{Gro18,Gro23,Gr2024}. However, such smoothing procedure has not been rigorously proved.

Previously, people have been able to carry out the smoothing for hypersurface singularities (see, e.g. \cite{Miao2002,BMN2011,BarHanke} and \cite{CarlottoLi2024}), as well as codimension two edge-cone type singularities (e.g. \cite{LM2019}). It is noteworthy that smoothing general singularities of codimensions at least three cannot always be carried out, as illustrated in the recent work of Cecchini-Frenck-Zeidler \cite{CFZ2025}.

The second main result of this paper is a general Coxeter gluing and smoothing theorem. We here give a quick and slightly informal statement of the theorem and delay the precise definitions and statements till Section \ref{section: Coxeter gluing and smoothing}. 

\begin{theorem}[Theorem \ref{thm: coxeter gluing smoothing}]\label{thm: coxeter gluing smoothing - intro}
    Let $n\ge 3$, $R_0\in\R$, and
\[
    X
    =
    \left(\bigsqcup_{\lambda\in\Lambda} M_\lambda\right)\big/\sim
\]
be a closed topological \(n\)-manifold, obtained
from finitely many compact Riemannian Coxeter building blocks
$(M_\lambda,g_\lambda)$ by identifying pairs of boundary facets. Assume:
\begin{enumerate}
    
    \item The gluing is Coxeter-compatible.
    \item The induced metrics are isometric on every pair of identified facets.  

    \item Each chamber has scalar curvature bounded below by $R_0$:
    \[
        R_{g_\lambda}\ge R_0
        \qquad\text{on }\Int M_\lambda .
    \]

    \item Along every identified facet \(F_{\lambda}\sim F_{\mu}\), the
    mean curvature jump is nonnegative:
    \[
        H_{\lambda}+H_{\mu}\ge 0.
    \]
\end{enumerate}

Then $X$ carries a smooth structure, compatible with the smooth structures
on the interiors of the chambers. The piecewise metrics  $\{g_\lambda\}$ patch together to a globally defined Lipschitz metric $g$ on $X$. For every $\varepsilon>0$, there exists a smooth Riemannian metric
    $g_\varepsilon$ on $X$ such that
    \[
        R_{g_\varepsilon}\ge R_0-\varepsilon,
    \]
    and
    \[
        g_\varepsilon\longrightarrow  g
    \]
    uniformly on $X$ and smoothly on compact subsets of the regular part of $g$.  In particular, \(g_\varepsilon\) may be chosen so that
    \[
        (1-\varepsilon) g\le g_\varepsilon\le (1+\varepsilon) g
    \]
    as quadratic forms.

\end{theorem}

Theorem~\ref{thm: coxeter gluing smoothing - intro} gives, to our
knowledge, the first general rigorous formulation and proof of the Coxeter-polyhedral smoothing principle for scalar curvature lower bounds.  This principle has long guided constructions in scalar curvature geometry: if the scalar curvature is bounded below on each chamber, the induced metrics match along the facets, the mean-curvature jumps across codimension-one strata are nonnegative, and the higher-codimensional strata are modeled on Coxeter
tessellations, then the resulting polyhedral metric should be smoothable with arbitrarily small loss of scalar curvature.

\subsection{PSC metrics with large systole on $3$-manifolds}

In the third part of this paper, we use the standard cylindrical blocks and the Coxeter gluing theorem to construct Riemannian cylinders with large systole. These gives the first examples of a large class of PSC $3$-manifolds with nontrivial lower bounds of the $\pi_2$-systole. 

\begin{theorem}\label{thm: examples of PSC 3 manifolds}
    Fix $R_0>0$.
    \begin{enumerate}
        \item Let $M^3$ be a $3$-manifold diffeomorphic to $S^3$ with three disjoint open $3$-balls removed. Then there exists a Riemannian metric $g$ on $M^3$ such that the boundary components are area-minimizing spheres, each of which realizes the $\pi_2$-systole of $(M,g)$, and
\begin{align*}
(\inf R_g)\cdot\sys_2(M,\partial M,g)>4\pi.
\end{align*}
        \item Let $M^3$ be a $3$-manifold  diffeomorphic to a lens space $L(p,q)$ with a solid $D^3$ removed $((p,q) = 1, p\ge 2)$. Then there exists a Riemannian metric $g$ on $M^3$ with area-minimizing boundary sphere, such that
        \begin{itemize}
            \item $(\inf R_g)\cdot\sys_2(M,\partial M,g) = 8\pi$, when $p=2$;
            \item $(\inf R_g)\cdot\sys_2(M,\partial M,g)>8\pi(1-\cos\frac\pi p)$, when $p\ge 3$.
        \end{itemize}
        \begin{align*}
            (\inf R_g)\cdot\sys_2(M,\partial M,g)>8\pi(1-\cos\frac\pi p).
        \end{align*}
        \item Let $M^3$ be a closed $3$-manifold such that
\begin{align*}
    M^3\cong \#_{i=1}^k (S^2\times S^1).
\end{align*}
Then 
\begin{itemize}
    \item $\fs(M)=8\pi$ when $k=1$;
    \item $\fs(M)>4\pi$ when $k\ge 2$.
\end{itemize}
        \item Let $M^3$ be a closed $3$-manifold such that
        \begin{align*}
            M\cong (\#_{i=1}^kS^2\times S^1)\#(\#_{j=1}^l L(p_j,q_j)) \qquad (l\ge 1, (p_j,q_j) = 1,p_j\ge 2).
        \end{align*}
         Assume that $\pi_2(M)\ne 0$, which holds if and only if $k\ge 1$ or $k=0, l\ge 2$. Then 
    \begin{itemize}
        \item $\fs(M)=8\pi$ when $M^3\cong RP^3\# RP^3$;
        \item Otherwise, $\fs(M)>\min \{4\pi, \min_j\{8\pi(1-\cos\frac{\pi}{p_j})\}\}$.
    \end{itemize}
    \end{enumerate}
\end{theorem}

The constant $4\pi$ in Theorem \ref{thm: examples of PSC 3 manifolds} could be compared with $8\pi$ by Bray-Brendle-Neves \cite{BBN10} and $24\pi \cdot\frac{2-\sqrt 2}{4-\sqrt 2}\approx 5.44 \pi$ by Xu \cite{Xu25}. We also note here that our construction gives much better $\pi_2$-systole bounds compared to the naive surgery construction (see Section \ref{section:comparison with surgery}). For example, on a lens space $L(p,q)\setminus D^3$, Theorem \ref{thm: examples of PSC 3 manifolds} gives a metric with $\pi_2$ systole $=O(p^{-2})$, while the naive Schoen-Yau/Gromov-Lawson surgery construction performed on the round metric only achieves the bound $O(p^{-6})$. We prove this bound in Section \ref{section:comparison with surgery} using the inverse mean curvature flow.

\subsection{Rigidity and non-rigidity phenomena}
It would be interesting to understand whether the standard cylindrical block $M_\alpha$ in Example \ref{example: standard cylinder block - intro} is rigid for the relative $\pi_2$-systole among all capillary prisms with angle at most $\alpha$, and consequently, whether the constants in Theorem \ref{thm: examples of PSC 3 manifolds} are optimal. As we will see in Section \ref{section: construction}, the answer to both questions are negative even among local deformations.

\begin{example}[Example \ref{example: perturbation counterexample}]
    There exists $\alpha_0 \approx 0.35 \pi$, such that for each $\alpha\in (0,\alpha_0)$, the standard cylindrical block $(M_\alpha,g)$ can be perturbed into $(M_\alpha, \tilde g)$ such that
    \[\sys_2(M_\alpha, \partial_0 M_\alpha,\tilde g)\cdot \inf_M R_{\tilde g}> 4\pi (1-\cos\alpha). \]
\end{example}

On the other hand, we have the following rigidity result for $M_\alpha$, under an additional curvature pinching assumption.

\begin{theorem}\label{thm: 2 systole estimate, rigidity}
     Let $(M^3,g)$ be an energy-essential capillary prism with scalar curvature $R_g\ge R_0>0$, sectional curvature $\sec\le \frac{R_0}{2}$, and angle at most $\alpha\in (0,\frac{\pi}{2}]$. Assume that $(M,g)$ contains no stable free boundary minimal surfaces, nor stable capillary minimal surfaces with capillary angle $\alpha$ whose boundary lies in $\partial_0 M$ in the interior. Then
    \begin{align*}
        \sys_2(M,\partial_0 M,g)\le \frac{4\pi(1-\cos\alpha)}{R_0}.
    \end{align*}
    Moreover, if the equality holds, then $(M,g)$ is isometric to a standard cylindrical block of upper angle $\alpha$.
\end{theorem}

It would be very interesting to investigate the optimal constant \[\fs_\alpha := \sup \{\sys_2(M_\alpha, \partial M_\alpha, g)\cdot \inf R_g : (M_\alpha,g) \text{ is a capillary prism with angle at most } \alpha\}\] 
in the $\pi_2$-systole estimate. Comparing the Theorem \ref{thm: 2 systole estimate} and Example \ref{example: standard cylinder block - intro}, we see that, assuming an almost optimal shape is energy-essential,
\[ \alpha^2\lesssim \fs_\alpha \lesssim |\log \alpha|^{-1}, \qquad \alpha \to 0. \]

We believe that the discrepancy between the proved bound $O(|\log \alpha|^{-1})$ and the construction with $O(\alpha^2)$ reflects the highly non-rigid behavior of PSC geometry: often times the optimal shape is not formed by a unit-speed foliation by totally geodesic leaves. Previously, such phenomena have been illustrated by the counterexample of the Min-Oo conjecture \cite{BMN2011}, and consequently the failure of the positive mass theorem modeled on the Schwarzschild de Sitter space.

\subsection{Building blocks for other spherical space forms}

A natural question following Theorem \ref{thm: examples of PSC 3 manifolds} is to construct PSC metrics with large $\pi_2$-systole on $3$-manifolds containing other spherical space forms. From \cite[Section 12.2]{Martelli2025GeometricTopology}, we know that every $3$-dimensional spherical space form is a finite cyclic quotient of a \textit{single-action}, or \textit{homogeneous}, spherical space form. Equivalently, every spherical space form is regularly cyclically covered by one of the homogeneous spherical space forms $S^3/G$, where $G$ is cyclic, binary dihedral ($D^*_{4m}$), binary tetrahedral ($T^*_{24}$), binary octahedral ($O^*_{48}$), or binary icosahedral ($I^*_{120}$). The corresponding spherical space forms are the homogeneous lens spaces, the prism manifolds, the binary tetrahedral space (also known as quaternionic space), the binary octahedral space, and the binary icosahedral space (also known as the dodecahedral space or the Poincar\'e homology sphere). 

In Section \ref{subsection: prism manifolds} and \ref{subsection: 3 special spaces}, we propose capillary Coxeter building blocks for $(S^3/G) \setminus D^3$, $G$ as above. A key difficulty here is that when $G = D^*_{4m}, T^*_{24}, O^*_{48}$ or $I^*_{120}$, the Coxeter rank of $G$ is $3$ (in comparison, the Coxeter rank of a finite cyclic group is $2$). This means that the Coxeter building blocks for the corresponding space forms must contain three faces meeting transversely at a common point. It is thus extremely challenging to find a simple and elegant candidate analogous to Example \ref{example: standard cylinder block - intro}. We have some preliminary numerical evidences for the existence of such Riemannian polyhedra. It is an interesting future question to explicitly construct them.

\subsection{Organization of the paper}

This paper contains three main parts. In part I of the paper (Section 2 and 3), we study the free capillary energy functional and the relative isoperimetric inequality to define a new monotone quasi-local mass. We then use it to prove Theorem \ref{thm: 2 systole estimate}. Part II of the paper (Section 4) contains a completely self-contained, rigorous proof of the general Coxeter gluing and smoothing theorem for scalar curvature. This part of the paper is independent from the others, and the proof can be skipped at a first reading. In Part III (Section 5, 6, 7) of the paper, we explicitly construct the standard cylindrical blocks, use them to build PSC $3$-manifolds with large $\pi_2$-systole, and study their rigidity and non-rigidity behaviors. 

\subsection{Acknowledgment}
The authors would like to thank Otis Chodosh for helpful conversations. S.H. would like to thank Yuguang Shi for constant encouragement. S.H. was partially supported by NSFC Grant No.125B2005. C.L. was partially supported by NSF grant DMS-2202343 and a Sloan Fellowship.

\section*{Part I. The relative $2$-systole estimates in capillary prisms}

\section{Capillary minimal surfaces and the free energy profile}
\subsection{Preliminaries}

$\quad$

We start by recalling some basics of capillary surfaces. Let $(M^n,g)$ be a Riemannian cylinder. We introduce the following class of admissible surfaces:

\begin{align*}
    \mathcal{F}(M) = \{\Sigma:\, &\text{There is a Caccioppoli set }\Omega\subset M, \partial_-M\subset \Omega, \partial_+ M\cap \Omega = \emptyset\\
    &\text{such that } \Sigma = \partial \Omega\cap \mathring{M}\}.
\end{align*}

Let us introduce several notations. For $\Sigma\in\mathcal{F}(M)$, denote $\Omega(\Sigma)$ to be the Caccioppoli set associated to $\Sigma$ in the above expression of $\mathcal{F}(M)$, and let $S = S(\Sigma) = \partial\Omega\cap \partial M_0$ be the wetting region on the lateral surface. Conversely, for any Caccioppoli set $\Omega$, let $\Sigma_\Omega$ and $S_\Omega$ denote the associated capillary surface and the wetting region, respectively. Let $N$ be the outer normal of $\Sigma$ respect to $\Omega$, $\bar{N}$ be the outer normal of $\partial_0 M$ respect to $M$; Let $\nu,\bar{\nu}$ be the co-normal of $\Sigma$ and $S$ pointing outward. Let $A$ be the second fundamental form of $\Sigma$ with respect to $N$ and $\mathrm{I\!I}$ the second fundamental form of $\partial_0 M$ with respect to $\bar{N}$. Let $H = tr A$ and $\bar{H} = tr \mathrm{I\!I}$ be the mean curvature of $\Sigma$ and $\partial_0 M$. Our convention is that the mean curvature of the unit $S^{n-1}$ in $\mathbf{R}^n$ is $n-1$ with respect to the outward pointing normal. We summarize the notations used here in Figure \ref{figure: the set up}.

Let $\Sigma\in \mathcal{F}(M)$ be a smooth embedded surface. An admissible deformation means a diffeomorphism $\Psi: (-\epsilon,\epsilon)\times \Sigma\longrightarrow M$ such that $\Psi(t,\cdot): \Sigma\longrightarrow M$ is a diffeomorphism when $t\in (-\epsilon,\epsilon)$, satisfying $\Psi(0,x) \equiv x$. Denote $\Sigma(t) = \{\Psi(t,x), x\in \Sigma\}$. Let $\Omega(t) = \Omega(\Sigma(t))$ and $S(t) = S(\Sigma(t))$. Denote $Y = \frac{\partial\Psi(t,\cdot)}{\partial t}\big|_{t=0}$ to be the vector field generating $\Psi$. 

For a fixed angle $\theta\in (0,\pi)$, we consider the functional
\begin{align}\label{eq: 1st variation for capillary functional}
    \mathcal{A}_{\theta}(\Sigma) = \mathcal{H}^{n-1}(\Sigma)-\cos\theta\cdot \mathcal{H}^{n-1}(S(\Sigma)).
\end{align}
Let $f = \langle Y,N\rangle$ be the normal component of the vector field $Y$. The first variation of $\mathcal{A}_\theta$ with respect to the variation vector field $Y$ is given by
\begin{align*}
    \frac{d}{dt}\bigg |_{t=0}\mathcal{A}_\theta(\Sigma(t)) = \int_\Sigma Hfd\mathcal{H}^{n-1}+\int_{\partial\Sigma}\langle Y, \nu-(\cos\theta) \bar\nu \rangle d\mathcal{H}^{n-2}.
\end{align*}
See for instance \cite{RosSouam97capillarystability}. $\Sigma$ is said to be a capillary minimal surface if $\frac{d}{dt}\big|_{t=0}\mathcal{A}_\theta(\Sigma(t)) = 0$ for any admissible deformation. From \eqref{eq: 1st variation for capillary functional}, $\Sigma$ is a  capillary minimal surface if and only if $H\equiv 0$ and the dihedral angle between $\Sigma$ and $\partial_0 M$, which equals the angle between $\nu$ and $\bar\nu$, is everywhere equal to $\theta$.

Now we assume $\Sigma$ is minimal capillary. The second variation formula is given by
\begin{align}\label{eq: 2nd variation for capillary functional}
    \frac{d^2}{dt^2}\bigg|_{t=0}\mathcal{A}_\theta(\Sigma(t)) = \int_\Sigma |\nabla_\Sigma f|^2-(\Ric(N,N)+|A|^2)f^2d\mathcal{H}^{n-1}-\int_{\partial\Sigma} Qf^2 d\mathcal{H}^{n-2},
\end{align}
where 
\begin{align*}
    Q = \frac{1}{\sin\theta}\mathrm{I\!I}(\bar\nu,\bar\nu)+(\cot\theta) A(\nu,\nu).
\end{align*}

Next, we recall the definitions for the $\pi_2$-systole.
\begin{definition}
    Let $(M^n,g)$ be a Riemannian manifold, possibly with boundary. If $M$ has nonvanishing second homotopy group, then we define the $\pi_2$-systole for $(M,g)$ to be
    \begin{align*}
        \sys_2(M,g) = \inf\{|S^2|_{f^*g}: f: S^2\longrightarrow M\text{ is an immersion with }[f]\ne 0\text{ in }\pi_2(M)\}.
        \end{align*}
\end{definition}

We also need the following relative version of the $\pi_2$-systole.

\begin{definition}
    Let $(M^n,g)$ be a Riemannian manifold with boundary, and let $A$ be a subset of $\partial M$. Suppose that
    \begin{align*}
        [(D^2,\partial D^2),(M,A)]\ne \{0\} 
    \end{align*}
        where $[(D^2,\partial D^2),(M,A)]$ denotes the set of homotopy classes of maps of pairs from $(D^2,\partial D^2)$ to $(M,A)$, and $0$ denotes the class of constant maps.
    We define the relative $\pi_2$-systole of $(M,A,g)$ by
\begin{align*}
    \sys_2(M,A,g)
    =
    \inf\left\{
    |D^2|_{f^*g}:
    \begin{aligned}
        &f:(D^2,\partial D^2)\longrightarrow (M,A)
        \text{ is an immersion}\\
        &\text{and } [f]\ne 0
        \text{ in } [(D^2,\partial D^2),(M,A)].
    \end{aligned}
    \right\}.
\end{align*}
\end{definition}

\begin{figure}
\centering
\includegraphics[width=0.6\textwidth]{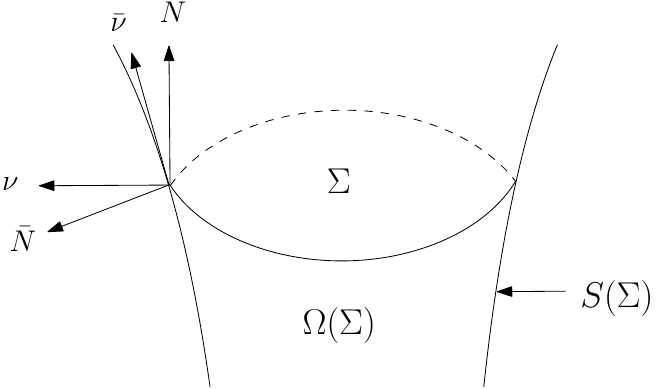}
\caption{Capillary surfaces}
\label{figure: the set up}
\end{figure}

\subsection{The energy-essential condition}
We proceed to discuss the energy-essential condition from Definition \ref{defn: energy-energy essential}. 

\begin{proposition}\label{prop: energy-essential}
    Let $(M,g)$ be a capillary prism with angle $\alpha$. Then it is energy-essential if any of the following holds.
    \begin{enumerate}
        \item $M$ admits a foliation by capillary surfaces of monotone angles from $\partial_{-}M$ to $\partial_+ M$;
        \item $\partial_0 M$ is area-minimizing.
        \item $M$ does not contain any stable capillary surface with angle $\beta\in (\alpha, \tfrac{\pi}{2})$ and boundary on $\partial_0 M$.
    \end{enumerate}
\end{proposition}

\begin{proof}
    Fix $\beta\in (\alpha, \tfrac{\pi}{2})$ and let $\Omega$ be a minimizer of $\cA_\beta$. By the standard regularity theory of capillary surfaces, we can write $\Omega$ as finite disjoint union $\sqcup_{j=0}^L \Omega^j$, such that each $\Sigma^j = \partial \Omega^j\cap \mathring M$ is a capillary surface meeting $\partial M$ at angle $\beta$. There has to be a connected component, which we denote by $\Omega^0$, that contains $\partial_{-}M$. Since $M$ is simply connected, $\Sigma^0$ is connected. Thus, the energy-essential condition holds if there is no other component $\Omega^j$, $j>0$. We now rule this out under one of the conditions in Proposition \ref{prop: energy-essential}.

    Suppose (1) holds. We have the standard calibration argument as follows. Let $u: M\to [-1,1]$ defines the foliation, namely for each $s\in [-1,1]$, $u^{-1}(t)$ is a leaf of the foliation with $u^{-1}(\pm 1) = \partial_\pm M$. Denote by $\xi=\tfrac{\nabla u}{|\nabla u|}$. Then along $\partial_0 M$, $\langle \xi, \bar N \rangle$ is decreasing. We thus obtain that the calibration form $\omega = \iota_\xi dV_g$ is closed, $|\omega|\le 1$, and $\omega|_{\partial_0 M} = \langle \xi, \bar N\rangle dA_{\partial_0 M} = -q$ ($q$ denotes the boundary flux). Denote by $\Sigma_\beta$ a leaf meeting with capillary angle $=\beta$ and $S_\beta$ the corresponding wetting region. Then for any Caccioppoli set $\Omega$,
    \[\int_{\Sigma_\Omega} \omega - \int_{S_\Omega} q = \int_{\Sigma_\beta} \omega - \int_{S_\beta}q = \cH^2(\Sigma_\beta) - \int_{S_\beta}q. \]
    This follows directly from that $d\omega=0$. Hence
    \begin{align*}
        \cA_\beta(\Sigma_\Omega) &= \cH^2(\Sigma_\Omega) - \cos\beta \cH^2(S_\Omega) \\
        &\ge \cH^2(\Sigma_\beta) - \int_{S_\beta} q + \int_{S_\Omega} q - \cos\beta \cH^2(S_\Omega)\\
        & = \cA_\beta(\Sigma_\beta) + \int_{S_\Omega\setminus S_\beta} (q-\cos\beta) - \int_{S_\beta \setminus S_\Omega} (q-\cos\beta).
    \end{align*}
    By assumption, $q\le \cos\beta$ in $S_\beta$ and $q\ge \cos\beta$ on $\partial_0 M\setminus S_\beta$, so we conclude that $\cA_\beta(\Sigma_\Omega) \ge \cA_\beta(\Sigma_\beta)$. This proves that $\Sigma_\beta$ is a minimizer of $\cA_\beta$, which is connected. It is straightforward to check from the equality case that any minimizer of $\cA_\beta$ is calibrated by $\omega$, and hence is a leaf in the foliation.

    Suppose (2) holds. Then for any null-homologous open set $U$ (relative to $\partial_0 M$), we have that
    \begin{align*}
        \cA_\beta(U) &= \cH^2(\partial U\cap \mathring M) - \cos\beta  \cH^2 (\partial U\cap \partial M)\\
        &\ge (1-\cos\beta) \cH^2(\partial U\cap \partial M)> 0,
    \end{align*}
    where the first inequality follows from that $\partial_0 M$ is area-minimizing. Thus, 
    \[\cA_\beta(\Omega) = \sum_{j=0}^l \cA_\beta(\Omega^j) > \cA_\beta(\Omega^0),\]
    unless $l=0$. So the capillary minimizer is just $\Omega^0$, which is connected.

    Suppose (3) holds. Note that any capillary surface arises from $\Omega^j$, $j\ge 1$, is relatively null-homologous and locally minimizing (and hence stable), which cannot exist by the assumption.

\end{proof}

Next, we establish that the energy-essential condition is preserved under small perturbations. We start with the following lemma.

\begin{lemma}\label{lem: no small parasitic regions}
    Suppose $(M,g)$ is a Riemannian cylinder. There exists a constant $\eps_0$ depending smoothly on $M$ and $\alpha$ such that the following holds. Given Caccioppoli set $\Omega$ in $M$ with $\cH^2(S_\Omega)<\eps_0$, for any $\beta\in [\alpha, \tfrac{\pi}{2}]$,
    \[\cA_\beta(\Omega)>0.\]
\end{lemma}

\begin{proof}
    Extend the vector field $\bar N$ on $\partial_0 M$ to a smooth vector field $X$ on $M$ such that $|X|\le 1$. Set $\omega= \iota_X dV_g$. Then $\omega|_{\partial_0 M} = dA_{\partial_0 M}$. Denote by 
    \[\Lambda = \|d\omega\|_{L^\infty (M)}.\]
    Fix $\Omega, \beta$ and denote by $\Sigma = \Sigma_\Omega$, $S=S_\Omega$ and $c=\cos\beta$. We prove that for sufficiently small $\eps_0$ depending on $M, \Lambda, \alpha$, 
    \[\cH^2(\Sigma) - c \cH^2(S) \ge \frac{1-c}{2}\cH^2(S), \]
    provided that $\cH^2(S)\le \eps_0$.

    Suppose the contrary. From
    \[\cH^2(\Sigma) - c \cH^2(S) < \frac{1-c}{2} \cH^2(S), \]
    we deduce that $\cH^2(\Sigma) < \cH^2(S)$. Since $S-\Sigma = \partial \Omega$ as currents, Almgren's small-area isoperimetric inequality \cite[Proposition 1.12]{Almgren1962} implies that
    \[\cH^3(\Omega) = M(\Omega) \le C_I M(S-T)^{3/2} \le C_I (2\cH^2(S))^{3/2}, \]
    provided that $\cH^2(S)<\eps_0$ is sufficiently small. Stokes' theorem gives
    \begin{align*}
        \cH^2(S) &= \int_S \omega =\int_\Sigma \omega + \int_\Omega d\omega\\
            &\le \cH^2(\Sigma) + \Lambda \cH^3(\Omega)\\
            &\le \cH^2(\Sigma) + C_I 2^{3/2} (\cH^2(S))^{3/2}.
    \end{align*}    
    
    Therefore, assuming that $\cH^2(S)$ is sufficiently small depending on $C_I$, the above gives 
    \[\cH^2(\Sigma) - c \cH^2(S) \ge \frac{1-c}{2}\cH^2(S),\]
    contradiction. Finally, since $\Lambda$, the isoperimetric constant $C_I$ as well as the required Almgren's smallness of area smoothly depend on $g$, so does $\eps_0$.

\end{proof}

We are ready to prove that the energy-essential condition is preserved under small perturbations.

\begin{proposition}\label{prop: energy-essential is open}
    Suppose a capillary prism $(M,g)$ with angle $\alpha$ is energy-essential. Then there is $\eps_1>0$ such that any capillary prism $(M, g')$ with angle $\alpha$ that is within the $\eps_1$-neighborhood in the smooth topology is energy-essential.
\end{proposition}

\begin{proof}
    Suppose the contrary. Then there exists a sequence of capillary prisms $(M, g_j)$ with $g_j \to g$, and $\Omega_j \subset M$ a disconnected minimizer of $\cA_{\beta_j, g_j}$. Write $\Omega_j = \sqcup_{i=1}^{l_j} \Omega_j^i$ with $l_j>0$ and $\Omega_j^0$ the connected component containing $\partial_{-}M$. Since $g_j$ converges to $g$, Lemma \ref{lem: no small parasitic regions} guarantees a uniform $\eps_0$ so that $\cH^2(S_{\Omega_j^i})\ge \eps_0$ for all $j$ and $1\le i\le l_j$. In particular, $\{l_j\}$ is uniformly bounded. Moreover, since each $\Sigma_{\Omega_j^i}$ is a stable capillary surface, the standard curvature estimates implies that a subsequence of $(\beta_j, \Omega_j)$ (which we do not relabel) converges smoothly graphically to $(\beta, \Omega)$, and $\Omega$ is a $\cA_{\beta, g}$-minimizer. The uniform area bound $\cH^2(S_{\Omega_j^i})\ge \eps_0$ guarantees that the limit $\Omega$ is disconnected, contradicting that $(M,g)$ is energy-essential.
\end{proof}

\subsection{Monotonicity of the free energy profile}\label{subsection: monotonicity of free energy profile}

$\quad$

In this subsection, we define the free energy profile in Riemannian manifolds and derive a differential inequality of it. We note that this profile was first considered for Euclidean domains by Eichmair-Koerber \cite{EK24} in their work of the extrinsic Penrose inequality. Most of their fundamental structural results for the profile function extends to the Riemannian setting in the current work. Let $(M^3,g)$ be an energy-essential capillary prism. We make the following assumption throughout this section: $\partial_+ M$ (resp. $\partial_{-}M$) is the outermost minimal surface meeting the lateral boundary $\partial_0 M$ at constant angle $\alpha$ (resp. $\tfrac{\pi}{2}$):

\textbf{Assumption A}
\begin{itemize}
    \item $\partial_+ M$ is a capillary minimal surface with capillary angle $\alpha$, and there is no other area-minimizing capillary minimal surface with capillary angle $\alpha$ inside $M$;
    \item $\partial_- M$ is a free boundary minimal surface, and there is no other free boundary minimizing surface inside $M$;
\end{itemize}
We remark here that requiring Assumption A does not lose any generality: if either $\partial_\pm$ is not the outermost minimal surface, we may minimize the corresponding functional (either the capillary functional for the area functional) to obtain another stable minimal disk, cut along it and obtain another capillary prism, which we then work on. The energy-essential property is preserved and the relative $\pi_2$-systole is non-decreasing in this procedure.
$\quad$

Let $\Phi(t) = \tanh t\,(t\in [0,\infty))$ be an auxiliary function. The only properties of $\Phi$ that we need are $\Phi(0) = 0$, $\lim_{t\to\infty}\Phi(t) = 1$ and $\Phi'(t)>0$. For $t\in (0,\Phi^{-1}(\cos\alpha))$ and $\Sigma\in\mathcal{F}(M)$, consider the functional
\begin{align*}
    J_t(\Sigma) = \mathcal{H}^2(\Sigma) - \Phi(t)\mathcal{H}^2(S(\Sigma)).
\end{align*}
Let $\mathcal{E}_t$ be the collection of surfaces $\Sigma\in\mathcal{F}(M)$ that minimizes $J_t$.

By assumption A, the dihedral angle between $\partial_+M$ and $\partial_0 M$ is $\alpha$, and the dihedral angle between $\partial_- M$ and $\partial_0 M$ equals $\frac\pi2$. Hence, from the maximum principle, the minimizer for $\mathcal{A}_\theta$ always exists for any $\theta\in (\alpha,\frac\pi2)$. In other words, we have:
\begin{lemma}
    For $t\in (0,\Phi^{-1}(\cos\alpha))$, $\mathcal{E}_t\ne\emptyset$.
\end{lemma}

The next two lemmas concern fundamental properties for elements in $\mathcal{E}_t$.
\begin{lemma}[{\cite[Lemma 21]{EK24}}]\label{lem: convergence of capillary surfaces}
   Let $\{t_k\}_{k=1}^\infty$ be a sequence of numbers with $t_k\in (0,\Phi^{-1}(\cos\alpha))$. Assume $\Sigma_{t_k}\in\mathcal{E}_{t_k}$ for each $k$. Then there exists $\Sigma_t\in\mathcal{E}_t$, such that by passing to a subsequence, $\Sigma_{t_k}\to\Sigma_t$ smoothly.
\end{lemma}

Given  $t\in (0,\Phi^{-1}(\cos\alpha))$, let
\begin{itemize}
    \item $s_t^- = \inf \{|S(\Sigma_t)|: \Sigma_t\in\mathcal{E}_t\}$;
    \item $s_t^+ = \sup \{|S(\Sigma_t)|: \Sigma_t\in\mathcal{E}_t\}$.
\end{itemize}

\begin{lemma}\label{lem: properties of Sigma t}

$\quad$

    \begin{enumerate}
        \item Let $t\in (0,\Phi^{-1}(\cos\alpha))$. There exist $\Sigma_t^-,\Sigma_t^+\in\mathcal{E}_t$ such that $|S(\Sigma_t^-)| = s_t^-$ and $|S(\Sigma_t^+)| = s_t^+$. Moreover,
        \begin{align*}
            |\Sigma_t^+| = |\Sigma_t^-|+\Phi(t)(s_t^+-s_t^-)
        \end{align*}
        \item Let $t_1,t_2\in  (0,\Phi^{-1}(\cos\alpha))$ with $t_2>t_1$. Then $s_{t_2}^->s_{t_1}^+$.
        \item Let $t\in  (0,\Phi^{-1}(\cos\alpha))$ and $\{t_k\}_{k=1}^\infty$ be a sequence of numbers. 
        \begin{itemize}
            \item Assume $t_k\searrow t$, then $s_{t_k}^-\searrow s_t^+$ and $s_{t_k}^+\searrow s_t^+$;
            \item Assume $t_k \nearrow t$, then $s_{t_k}^-\nearrow s_t^-$ and $s_{t_k}^+\nearrow s_t^-$.
        \end{itemize}
    \end{enumerate}
\end{lemma}
\begin{proof}
    The proof of (1)(2) is the same as \cite[Proposition 23, Lemma 24]{EK24}. (3) follows directly from Lemma \ref{lem: convergence of capillary surfaces} and (2).
\end{proof}

Next, we define the free energy profile $I: (0, |\partial_0M|)\longrightarrow (0,\infty)$ following the idea of \cite{EK24}. We give self-contained detailed arguments here. First, we need to define a function $\sigma$ that associates the lateral area $s$ to the time parameter $t$. From Lemma \ref{lem: properties of Sigma t} (2), $s_{t_1}^+<s_{t_2}^-$ whenever $0<t_1<t_2<\Phi^{-1}(\cos\alpha)$, so $[s_{t_1}^-,s_{t_1}^+]\cap [s_{t_2}^-,s_{t_2}^+] = \emptyset$. Moreover,  Lemma \ref{lem: convergence of capillary surfaces} and Lemma \ref{lem: properties of Sigma t}(3) imply $\bigcup_{t\in (0,\Phi^{-1}(\cos\alpha))}[s_t^-,s_t^+]$ is an open sub-interval of $(0,|\partial_0M|)$. Since $\partial_- M$ is an outermost free boundary minimizing surface, and $\partial_+ M$ is an innermost capillary minimizing surface of capillary angle $\alpha$, we conclude 
$$\bigcup_{t\in (0,\Phi^{-1}(\cos\alpha))}[s_t^-,s_t^+] = (0,|\partial_0 M|).$$
Consequently, we can define a continuous and nondecreasing function $\sigma: (0, |\partial_0M|)\longrightarrow (0,\Phi^{-1}(\cos\alpha))$ given by $\sigma(s) = t$ when $t\in (0,\Phi^{-1}(\cos\alpha))$ is the unique number satisfying $s\in [s_t^-,s_t^+]$.

With $\sigma$ at hand we proceed to define the free energy profile $$I: (0, |\partial_0M|)\longrightarrow (0,\infty),$$ which is given by
\begin{align*}
    I(s) = |\Sigma_{\sigma(s)}^-|+\Phi(\sigma(s))(s-s_{\sigma(s)}^-).
\end{align*}
Using Lemma \ref{lem: properties of Sigma t} and that $\sigma$ is continuous, we see that $I$ is also continuous. We have the following natural upper bound of $I$.

\begin{lemma}[{\cite[Lemma 28]{EK24}}]\label{lem: properties of I}
 Let $\Sigma\in \mathcal{F}(M)$ be a surface, then
    \begin{align*}
        I(|S(\Sigma)|)\le |\Sigma|.
    \end{align*}
\end{lemma}

Now we are in a position to derive the main differential inequality of this section
\begin{proposition}\label{prop: differential inequality for I}
Given $\delta>0$ sufficiently small. For every $s_0\in (\delta, |\partial_0M|-\delta)$, there exists $\tilde{\delta}\in (0,\delta)$ and $\psi\in C^\infty(s_0-\tilde{\delta},s_0+\tilde{\delta})$ such that
\begin{enumerate}
    \item\label{sec2 item1} $\psi(s_0) = I(s_0)$;
    \item\label{sec2 item2} $\psi(s)\ge I(s)$ whenever $s\in (s_0-\tilde{\delta},s_0+\tilde{\delta})$;
    \item\label{sec2 item3} $\psi'(s_0) = \Phi(\sigma(s_0))$;
    \item\label{sec2 item4} 
        If $s_0\in (s_{\sigma(s_0)}^-,s_{\sigma(s_0)}^+)$, then $\psi''(s_0) = 0$;\\
        If $s_0 = s_{\sigma(s_0)}^{\pm}$, let $\Sigma_{\sigma(s_0)}\in\mathcal{E}_{\sigma(s_0)}$ such that $|\Sigma_{\sigma(s_0)}| = s_0$, then
    \begin{multline}\label{eq: 5}
    \psi''(s_0) \le |\partial\Sigma_{\sigma(s_0)}|^{-2}(1-\psi'(s_0)^2)\\
    \cdot\left(2\pi\chi(\Sigma_{\sigma(s_0)}) -\int_{\Sigma_{\sigma(s_0)}}  \frac{1}{2}(R_g+|A|^2)-\frac{1}{\sqrt{1-\psi'(s_0)^2}}\int_{\partial\Sigma_{\sigma(s_0)}}\bar{H}\right).
\end{multline}  
\end{enumerate}
\end{proposition}

\begin{remark}\label{remark: connectedness is needed}
    At a technical level, the Euler characteristic term in \eqref{eq: 5} is the reason that we require the energy-essential condition. Without it, we may get unbounded number of connected components of $\Sigma$ and lose control of this term. In  \cite{EK24}, Eichmair-Koerber worked through this issue because of the validity of the Euclidean isoperimetric inequality $|\partial \Sigma|^2 \ge 4\pi |\Sigma|$. This scaling enables them to choose different weights in the stability of components of $\Sigma$. 

    For surfaces $\Sigma$ in a positive curvature background geometry, $|\partial \Sigma|$ and $|\Sigma|$ usually scale the same way. Because of this, there seems no elegant way around the multiple connectedness issue. This case is similar to that of Bray's comparison theorem \cite{Bray97}, where connectedness has to be assumed.
\end{remark}

\begin{proof}
    Let $\delta>0$, and $s_0\in (\delta, |\partial_0 M|-\delta)$.
    
    We first consider the case $s_0 = s_{\sigma(s_0)}^{\pm}$. Let $\Sigma_{\sigma(s_0)}\in\mathcal{E}_{\sigma(s_0)}$ such that $|S(\Sigma_{\sigma(s_0)})| = s_0$. Let $\Omega(\tau) \,(\tau\in (-\epsilon,\epsilon))$ be a variation of $\Omega(\Sigma_{\sigma(s_0)})$ of with an admissible variation vector field $X = fN+X_0$, where $X_0$ is the projection of $X$ to $T\Sigma_{\sigma(s_0)}$. Thus, we have $\Omega(0) = \Omega(\Sigma_{\sigma(s_0)})$, and $\Sigma(\tau) = \partial\Omega(\tau)\cap\mathring{M}\in\mathcal{F}(M)$ forms an admissible deformation of $\Sigma_{\sigma(s_0)}$ for $\tau\in (-\epsilon,\epsilon)$. Let $S(\tau) = S(\Sigma(\tau))$. Define $\psi(\tau) = |\Sigma(\tau)|$ and $\eta(\tau) = |S(\tau)|$.

    Let $\theta = \arccos(\Phi(\sigma(s_0)))$. Since $\Sigma_{\sigma(s_0)}$ minimizes the functional
    \begin{align*}
        J_{\sigma(s_0)}(\Sigma) = |\Sigma|-\cos\theta |S(\Sigma)|
    \end{align*}
    among $\mathcal{F}(M)$, $\Sigma_{\sigma(s_0)}$ is a connected area-minimizing capillary surface with capillary angle $\theta$. The first variation formula yields
    
\begin{equation}\label{eq: 1}
    \begin{split}
            &\psi'(0) = \frac{d}{d\tau}|_{\tau = 0}|\Sigma(\tau)| = \int_\Sigma Hf+\int_{\partial\Sigma}\langle X,\nu\rangle = \int_{\partial\Sigma}f\cot\theta,\\
    &\eta'(0) = \frac{d}{d\tau}|_{\tau = 0}|S(\tau)| = \int_{\partial\Sigma}\langle X,\bar{\nu}\rangle = \int_{\partial\Sigma}\frac{1}{\sin\theta}f.
    \end{split}
\end{equation}
Hence we deduce that
\begin{align}\label{eq: 4}
    \psi'(0) = \cos\theta\cdot \eta'(0)
\end{align}

Next, using the second variation formula for capillary minimal surfaces, we find that
\begin{equation}\label{eq: 2nd variation for capillary surface 1}
    \begin{split}
        \psi''(0) - \cos\theta\cdot &\eta''(0) = \frac{d^2}{d\tau^2}|_{\tau = 0}\mathcal{A}_\theta(\Omega_t)\\
    & = \int_\Sigma |\nabla f|^2-(\Ric(N,N)+|A|^2)f^2 - \int_{\partial\Sigma} (\frac{1}{\sin\theta}\mathrm{I\!I}(\bar{\nu},\bar{\nu})+(\cot\theta) A(\nu,\nu))f^2\\
    & = \int_\Sigma |\nabla f|^2-\frac{1}{2}(R_g+|A|^2)f^2+Kf^2 - \int_{\partial\Sigma} (\frac{1}{\sin\theta}\bar{H}-k_{\partial\Sigma})f^2\ge 0
    \end{split}
\end{equation}
where in the last equality  we have used \cite[(3.8)]{Li20}
\begin{align}\label{eq: 14}
    \mathrm{I\!I}(\bar{\nu},\bar{\nu})+\cos\theta A(\nu,\nu)+\sin\theta k_g = \bar{H}.
\end{align}

In the following let us assume $f\equiv 1$. Combinning \eqref{eq: 2nd variation for capillary surface 1} with the Gauss-Bonnet formula, we obtain
\begin{align}\label{eq: 2}
    \psi''(0) - \cos\theta\cdot\eta''(0) = 2\pi\chi(\Sigma) -\int_\Sigma  \frac{1}{2}(R_g+|A|^2)-\frac{1}{\sin\theta}\int_{\partial\Sigma}\bar{H}.
\end{align}

Let $s = |S(\tau)| = \eta(\tau)$ and note that $s_0 = |S(0)| = \eta(0)$. From \eqref{eq: 1},
\begin{align}\label{eq: eta derivative in t}
    \eta'(0) = \frac{1}{\sin\theta}|\partial\Sigma|>0,
\end{align}
so $s$ is strictly increasing in $\tau$ when $\tau$ is in a sufficiently small interval containing $0$. Thus, $\psi$ can be viewed as a function of $s$ near $s_0$. We have $\psi(s_0) = |\Sigma_{\sigma(s_0)}| = I(s_0)$, which verifies \eqref{sec2 item1}. \eqref{sec2 item2} follows directly from Lemma \ref{lem: properties of I}(1) and the definition of $\psi$.

The remaining task is to calculate the derivative of $\psi$ with respect to $s$. For simplicity, we use notations of the form $\psi', \eta'$ to represent the derivative with respect to $\tau$. To begin with, we note \eqref{eq: 4} implies that
\begin{align}\label{eq: 3}
    \frac{d\psi}{ds}\bigg |_{s=s_0} =\frac{\psi'}{\eta'}\bigg |_{\tau=0} = \cos\theta = \Phi(\sigma(s_0))>0,
\end{align}
which verifies \eqref{sec2 item3}. Thus
\begin{align*}
    \frac{d^2\psi}{ds^2}\bigg |_{s=s_0} = \frac{\psi''\eta' - \psi'\eta''}{(\eta')^3}\bigg |_{\tau=0} = \frac{1}{(\eta')^2}(\psi''-(\frac{\psi'}{\eta'})\eta'') = \frac{1}{(\eta')^2}(\psi''-\cos\theta\cdot\eta'').
\end{align*}
Combined with \eqref{eq: 2} and \eqref{eq: eta derivative in t}, we have
\begin{align*}
    \frac{d^2\psi}{ds^2} \bigg|_{s=s_0} = |\partial\Sigma|^{-2}\sin^2\theta\left(2\pi\chi(\Sigma) -\int_\Sigma  \frac{1}{2}(R_g+|A|^2)-\frac{1}{\sin\theta}\int_{\partial\Sigma}\bar{H}\right).
\end{align*}
Plugging \eqref{eq: 3} into the previous expression, we obtain \eqref{sec2 item4}. This completes the proof when $s_0 = s_{\sigma(s_0)}^{\pm}$.

If $s_0\in (s_{\sigma(s_0)}^-,s_{\sigma(s_0)}^+)$, using that $I$ is linear near $s_0$, we see that the function
    \begin{align*}
        \psi(s) = I(s_0)+\Phi(\sigma(s_0))(s-s_0)
    \end{align*}
    coincides with $I(s)$ near $s_0$. Thus in this case, $\psi''(s_0) = 0$ and $\psi'(s_0) = \Phi(\sigma(s_0))\in (0,1)$. Combined with \eqref{eq: 2nd variation for capillary surface 1}, we see that  the right hand side of \eqref{eq: 5} is nonnegative. This completes the proof of the proposition.
\end{proof}

\begin{lemma}\label{lem: semiconcavity of I}
Assume $R_g\geq 0$. Then 
\begin{enumerate}
    \item the left derivative $\underline{I}'$ and the right derivative $\bar{I}'$ exist at every point of $(0,|\partial_0M|)$;
    \item $\underline{I}'$ is left-continuous, and $\bar{I}'$ is right-continuous;
    \item $I$ is differentiable except at countably many points.
\end{enumerate}

\end{lemma}

\begin{proof}
For any $\Sigma\in\mathcal{F}(M)$, the boundary $\partial\Sigma$ represents a nontrivial homology class in $H_1(\partial_0M)$. Therefore, by standard results in geometric measure theory,
\begin{align*}
    c_0:=\inf\{|\partial\Sigma|:\Sigma\in\mathcal{F}(M)\}>0.
\end{align*}
Under the assumption $R_g\ge 0$, we obtain
\begin{align*}
    \psi''(s_0)
    \le
    2\pi\chi(\Sigma_{\sigma(s_0)})c_0^{-2}
    +\frac{1}{\sin\alpha}c_0^{-1}
    \sup_{\partial_0M}|\bar H|.
\end{align*}
Hence, $J(s) = I(s)-\frac{C}{2}s^2$ is concave in the viscosity sense for some constant $C$. By the standard equivalence between viscosity and classical concavity, it follows that $J$ is concave on $(0,|\partial_0M|)$. All of the properties listed in the lemma follows from the concavity of $J$.

%In particular, the left derivative $\underline{I}'$ and the right derivative $\bar{I}'$ exist at every point of $(0,|\partial_0M|)$. Since $\underline{J}'(s) = \underline{I}'(s)-Cs$ and $\bar{J}'(s) = \bar{I}'(s)-Cs$ are monotone, $I$ is differentiable except at countably many points.
\end{proof}

We end this section by recording the following lemma.

\begin{lemma}[{\cite[Lemma 29]{EK24}}]\label{lem: I is increasing}
    $I$ is strictly increasing on $(0,|\partial_0 M|)$. Furthermore,
    \begin{align*}
        \bar{I}'(s)\le \Phi(\sigma(s))\le \underline{I}'(s)
    \end{align*}
    for every $s\in (0,|\partial_0 M|)$.
\end{lemma}

\section{The angle to systole inequality in 3 dimensional PSC capillary prisms}

\subsection{Isoperimetric inequalities on compact surfaces with boundary}

$\quad$
In this section we continue the proof of Theorem \ref{thm: 2 systole estimate}. Given Proposition \ref{prop: differential inequality for I}, the remaining piece is to bound $|\partial \Sigma_s|$ in terms of $|\Sigma_s|$. We thus proceed with an isoperimetric inequality for surfaces with boundary. 

Let $\Sigma$ be a compact oriented surface with boundary. The key assumption we have is that
    \begin{align}\label{eq: 7}
        \int_\Sigma (|\nabla\varphi|^2+K\varphi ^2) d\mathcal{H}^2 +\int_{\partial\Sigma} k\varphi^2 d\mathcal{H}^1\ge \lambda\int_\Sigma \varphi^2d\mathcal{H}^2
    \end{align}
    holds for all $\varphi\in C^\infty(\Sigma)$, where $K$ and $k$ denote the Gaussian curvature of $\Sigma$ and the geodesic curvature of $\partial\Sigma$, $\lambda >0$. \eqref{eq: 7} asserts that $\Sigma$ has Gauss curvature bounded by $\lambda$ and convex boundary in the weak spectral sense. It is a consequence of the capillary stability of $\Sigma$.
    
    Assume $(\Sigma,g)$ satisfies \eqref{eq: 7}. We will adopt an argument by \cite{GL83}. First, we introduce the following notations:
\begin{align*}
\gamma &= \partial \Sigma, 
& d(x) &= d(x,\gamma), 
& \gamma(s) &= \{ x \in \Sigma : d(x) = s \}, \\
L(s) &= |\gamma(s)|, 
& K(s) &= \int_{\gamma(s)} K \, dl, 
& G(s) &= \int_0^s K(t)\, dt, \\
\rho_+ &= \sup_{x \in \Sigma} d(x), 
& \Omega(s) &= \{ x \in \Sigma : 0 \le d(x) < s \}, 
& \chi(s) &= \chi(\Omega(s)).
\end{align*}
Note that $\rho^+$ is exactly the relative filling radius $\text{Fill Rad}(\partial\Sigma\subset \Sigma)$ in the language of \cite{Gro83}. Finally, we define
\begin{align}\label{eq: 9}
    \Gamma(s) = 2\pi\chi(s)-G(s)-\int_{\gamma}k.
\end{align}

When $\gamma(s)$ is smooth, the Gauss-Bonnet theorem applied to $\Omega(s)$ implies
\begin{align*}
    \int_{\Omega(s)}K+\int_{\gamma}k+\int_{\gamma(s)}k = 2\pi\chi(s).
\end{align*}
In particular, we have
\begin{align*}
    \Gamma(s) = \int_{\gamma(s)}k
\end{align*}
whenever $\gamma(s)$ is smooth. Here for boundary components $\gamma, \gamma(s)\subset \partial\Omega(s)$, the geodesic curvature is taken with respect to the outward unit normal of $\Omega(s)$. Also, $L'(s) = \Gamma(s)$ when $s$ is small. However, $L(s)$ is generally not continuous because of cut locus, so we need the following partial regularity result.

\begin{lemma}(\cite{Hartman1964},\cite[Lemma 2.1]{Xu22})
    $\gamma(s)$ defined above is piecewise smooth for almost every $s$. Moreover, $L(s)$ is almost everywhere differentiable and $L'(s)\le \Gamma(s)$ for $s>0$.
\end{lemma}

The following fact is proved as a claim in \cite{Xu22}. We include a proof for completeness.
\begin{lemma}\label{lem: weak differential inequality}
    Let $f\in C^\infty([0,\rho^+])$ be a non-positive function, then
    \begin{align*}
        \int_0^{\rho^+}\Gamma f\le -\int_0^{\rho^+}Lf'-L(0)f(0).
    \end{align*}
\end{lemma}
\begin{proof}
    Since $L'(s)=\Gamma(s)$ is smooth in a neighborhood of $s=0$, it is enough to verify the identity for functions $f$ supported away from $0$. 

Fix parameters $\varepsilon \ll \delta \ll 1$ such that $L$ is differentiable at $\rho_+ - \delta$. Let $\psi$ be a cutoff function satisfying
\[
\psi \equiv 0 \ \text{on } [0, \rho_+ - \delta - \varepsilon], 
\quad 
\psi \equiv 1 \ \text{on } [\rho_+ - \delta + \varepsilon, \rho_+],
\]
and $|\psi'| \leq C \varepsilon^{-1}$ for some constant $C$. Decompose
\[
f = f\psi + f(1-\psi).
\]
Then one easily checks that
\[
\int_0^{\rho_+} \Gamma f \psi \, ds \to 0 \quad \text{as } \delta \to 0.
\]

Using the weak inequality $L' \leq \Gamma$, we obtain
\begin{align*}
\int_0^{\rho_+} \Gamma f(1-\psi)\, ds
&\leq - \int_0^{\rho_+} L \bigl(f(1-\psi)\bigr)' \, ds \\
&= - \int_0^{\rho_+} L f' \, ds + o_\delta(1) + \int_0^{\rho_+} L f \psi' \, ds.
\end{align*}

Finally, as $\varepsilon \to 0$, we have
\[
\int_0^{\rho_+} L f \psi' \, ds \to L(\rho_+ - \delta)\, f(\rho_+ - \delta),
\]
which is non-positive. This completes the proof.
\end{proof}

Now we  are able to derive a useful differential inequality from \eqref{eq: 7}.

\begin{lemma}
    Let $\Sigma$ be a compact surface with boundary satisfying \eqref{eq: 7}. Then for any  function $\phi\in C^\infty([0,\rho^+])$ with $\phi\ge 0, \phi'\le 0$ when $s\in (0,\rho^+)$, it holds
\begin{align}\label{eq: 11}
    \lambda\int_0^{\rho^+}L\phi^2 + \int_0^{\rho^+}2L\phi\phi''+\int_0^{\rho^+}L(\phi')^2+ 2L(\gamma)\phi(0)\phi'(0) \le 2\pi\chi(\Sigma)\phi(\rho^+)^2.
\end{align}
\end{lemma}

\begin{proof}
We test \eqref{eq: 7} by the function $\varphi(x) = \phi(d(x))$. The coarea formula implies
\begin{align}\label{eq: 8}
    \lambda \int_0^{\rho^+}L(s)\phi(s)^2ds\le \int_0^{\rho^+}[L(s)\phi'(s)^2+K(s)\phi(s)^2]ds+(\int_\gamma k)\phi(0)^2.
\end{align}
Since $G'(s) = K(s)$, an integration by parts together with \eqref{eq: 9} yields
\begin{align*}
    \int_0^{\rho^+}K(s)\phi(s)^2 ds &= \int_0^{\rho^+} G'(s)\phi(s)^2ds\\
    &=G(s)\phi(s)^2\Big|_0^{\rho^+}-\int_0^{\rho^+}2G\phi\phi'ds\\
    &= [2\pi\chi(\Sigma)-\int_\gamma k]\phi(\rho^+)^2 - \int_0^{\rho^+}[2\pi\chi(s)-\Gamma(s)-\int_\gamma k]2\phi\phi'ds\\
    &= [2\pi\chi(\Sigma)-\int_\gamma k]\phi(\rho^+)^2-\int_0^{\rho^+}4\pi\chi(s)\phi\phi'\\
    &+\int_0^{\rho^+}2\Gamma(s)\phi\phi'+(\int_\gamma k)[\phi^2(\rho^+)-\phi^2(0)]\\
    &= 2\pi\chi(\Sigma)\phi(\rho^+)^2-\int_0^{\rho^+}4\pi\chi(s)\phi\phi'+\int_0^{\rho^+}2\Gamma(s)\phi\phi'-(\int_\gamma k)\phi^2(0).
\end{align*}
Since $\Omega(s)$ is connected and has more than one boundary components, we have $\chi(s)\le 0$. 
Recalling that $\phi \ge 0$ and $\phi' \le 0$, the above inequality becomes
\begin{align*}
   \int_0^{\rho^+}K(s)\phi(s)^2 ds + (\int_\gamma k)\phi(0)^2\le 2\pi\chi(\Sigma)\phi(\rho^+)^2+\int_0^{\rho^+}2\Gamma(s)\phi\phi' .
\end{align*}
Applying Lemma \ref{lem: weak differential inequality} with $f = 2\phi\phi'$, we obtain
\begin{align*}
    \int_0^{\rho^+}2\Gamma(s)\phi\phi' &\le \int_0^{\rho^+}2L'\phi\phi' \\
    & = 2L\phi\phi'\Big|_0^{\rho^+}-\int_0^{\rho^+}2L\phi\phi''-\int_0^{\rho^+}2L(\phi')^2\\
    &= -2L(\gamma)\phi(0)\phi'(0) - \int_0^{\rho^+}2L\phi\phi''-\int_0^{\rho^+}2L(\phi')^2
\end{align*}
Thus
\begin{equation}\label{eq: 10}
    \begin{split}
        &\int_0^{\rho^+}K(s)\phi(s)^2 ds + (\int_\gamma k)\phi(0)^2+2L(\gamma)\phi(0)\phi'(0) + \int_0^{\rho^+}2L\phi\phi''+\int_0^{\rho^+}2L(\phi')^2\\
    &\le 2\pi\chi(\Sigma)\phi(\rho^+)^2.
    \end{split}
\end{equation}
Combined with \eqref{eq: 8} we get the desired result.
\end{proof}

The following lemma gives the bound of $\rho^+$ under the spectral curvature condition \eqref{eq: 7}.
\begin{lemma}\label{lem: bound for rho+}
Let $\Sigma$ be a compact surface with boundary satisfying \eqref{eq: 7}, then
    \begin{align*}
        \rho^+ = \mathrm{Fill Rad}(\partial\Sigma\subset \Sigma)\le \frac{\pi}{\sqrt{3\lambda}}
    \end{align*}
\end{lemma}
\begin{proof}
    Let $u>0$ be the first eigenfunction of $-\Delta+(K-\lambda)$ subject to a Robin boundary condition:
    \begin{align*}
        \begin{cases}
-\Delta u+(K-\lambda)u = \mu_1 u & \mbox{in  } \Sigma,\\
 \frac{\partial u}{\partial\nu}+ku = 0 &\mbox{on }\partial\Sigma.
\end{cases}
    \end{align*}
Assume the lemma is false, then we can find a point $p\in\Sigma$ such that $d(p,\partial\Sigma) = \rho^+>\tfrac{\pi}{\sqrt{3\lambda}}$. Pick $0<\epsilon<\frac{\sqrt{3\lambda}\rho^+}{\pi}-1$, and consider a smooth function $\rho_1: \Sigma\backslash B_\epsilon(p)$ satisfying $\rho_1|_{\partial\Sigma} = 0$, $\rho_1|\partial B_\epsilon(p) = \rho^+$ and $\Lip\rho_1<1+\epsilon$. Consider the following functional acting on all open sets $\Omega\subset\Sigma$ with $\partial\Sigma\subset\Omega$ and $\Omega\cap B_\epsilon(p) = \emptyset$:
\begin{align*}
    E_u(\Omega) = \int_{\partial\Omega} u \, d\mathcal{H}^1 
    - \int_\Omega h u \, d\mathcal{H}^2
\end{align*}
Let
\begin{align*}
    h(x) = -\frac{2\pi(1+\epsilon)}{3\rho^+}\tan\left( \frac{\pi}{2\rho^+}\rho_1(x)\right).
\end{align*}
It is direct to verify
\begin{equation}\label{eq: property of the mu function}
    \begin{split}
        h|_{\partial\Sigma = 0}, \quad h|_{\partial B_\epsilon(p)} = -\infty,\quad |\nabla h|<\frac{3}{4}h^2+\lambda.
    \end{split}
\end{equation}
By \cite[Proposition 15]{CL2024}, the minimizer $\gamma_1$ of $E_u$ exists. By selecting the test function in the second variation formula as in \cite[Appendix A]{Xu22}, we obtain
\begin{align*}
    0\le \int_{\gamma_1}\Big[\frac{3}{4}h^2-\lambda+|\nabla h|\Big]u^{-1}d\mathcal{H}^1,
\end{align*}
which contradicts \eqref{eq: property of the mu function}.
\end{proof}

Now we are ready to prove the main estimate of this subsection.
\begin{lemma}\label{lem: isoperimetric type inequality on Sigma}
    Let  $\Sigma$ be a compact surface with boundary satisfying \eqref{eq: 7}, then
    \begin{align*}
         |\partial\Sigma|\ge \sqrt{\frac{\lambda}{3}}|\Sigma|.
    \end{align*}
\end{lemma}

\begin{proof}
Let $\phi(s) = -\sin(\sqrt{\frac{\lambda}{3}}(s-\rho_+))$. It is direct to check
\begin{align*}
    \lambda\phi^2+2\phi''\phi+(\phi')^2 = \frac{\lambda}{3},\quad2\phi(0)\phi'(0) = -\sqrt{\frac{\lambda}{3}}\sin\left(2\sqrt{\frac{\lambda}{3}}\rho_+\right).
\end{align*}
By Lemma \ref{lem: bound for rho+}, $$\rho^+\le \frac{\pi}{\sqrt{3\lambda}}<\frac{\sqrt{3}\pi}{2\sqrt{\lambda}} = \frac{\frac{\pi}{2}}{\sqrt{\frac{\lambda}{3}}}.$$ Therefore, $\phi(s)$ is decreasing on $(0,\rho^+)$. Since $\phi(\rho^+) = 0$, it follows from \eqref{eq: 10} that
\begin{align*}
    \frac{\lambda}{3}|\Sigma|+2\phi(0)\phi'(0)|\partial\Sigma| = \frac{\lambda}{3}|\Sigma|-\sqrt{\frac{\lambda}{3}}\sin\left(2\sqrt{\frac{\lambda}{3}}\rho_+\right)|\partial\Sigma|\le 0,
\end{align*}
which implies the desired estimate.
\end{proof}

\subsection{Proof of Theorem \ref{thm: 2 systole estimate}}

$\quad$

In this subsection, we prove Theorem \ref{thm: 2 systole estimate}. Let $(M^3,g)$ be as in Theorem \ref{thm: 2 systole estimate}. We record the following fundamental consequence of the second variation formula.
\begin{lemma}\label{lem: spectral positivity on surfaces}
    Let $\Sigma\in\mathcal{F}(M)$ be an area-minimizing capillary surface of capillary angle $\theta$. Then 
    \begin{enumerate}
        \item  for any $\varphi\in C^\infty(\Sigma)$, there holds
    \begin{align}\label{eq: variational inequality for capillary surfaces}
        \int_\Sigma (|\nabla\varphi|^2+K\varphi ^2) d\mathcal{H}^2 +\int_{\partial\Sigma} k\varphi^2 d\mathcal{H}^1\ge \frac{R_0}{2}\int_\Sigma \varphi^2d\mathcal{H}^2,
    \end{align}
    where $K$ denotes the Gauss curvature of $\Sigma$ and $k$ the geodesic curvature of $\partial\Sigma$. 

    \item $\Sigma$ is a topological disk, and
    \begin{align}\label{eq: 33}
        |\Sigma|\le \frac{4\pi}{R_0}.
    \end{align}

    \item \begin{align}\label{eq: 34}
    |\partial \Sigma|\ge C_{\mathrm{iso}}\sqrt{R_0}|\Sigma|
\end{align}
for $C_{\mathrm{iso}} = \frac{1}{\sqrt{6}}$.
    \end{enumerate}
 
\end{lemma}
\begin{proof}
    From the second variation formula for capillary surfaces we deduce that (See for instance the reasoning of \eqref{eq: 2nd variation for capillary surface 1} in the proof of Proposition \ref{prop: differential inequality for I})
    \begin{align}\label{eq: 76}
        \int_\Sigma |\nabla \varphi|^2- \frac{1}{2}(R_g+|A|^2)\varphi^2+K\varphi^2 - \int_{\partial\Sigma} (\frac{1}{\sin\theta}\bar{H}-k_{\partial\Sigma})\varphi^2\ge 0
    \end{align}
    holds for all $\varphi\in C^\infty(\Sigma)$. The inequality \eqref{eq: variational inequality for capillary surfaces} then follows from the assumption that $R_g\ge R_0$ in $M$, as well as $\bar{H}\ge 0$ on $\partial_0 M$. Plugging $\varphi\equiv 1$ into \eqref{eq: variational inequality for capillary surfaces} and using the Gauss-Bonnet formula, we obtain $4\pi\chi(\Sigma)\ge R_0|\Sigma|>0$. Thus $\chi(\Sigma) = 1$, $\Sigma$ is a topological disk and $|\Sigma|\le \frac{4\pi}{R_0}$. The last statement follows from \eqref{eq: variational inequality for capillary surfaces} and Lemma \ref{lem: isoperimetric type inequality on Sigma}.
\end{proof}

Before proceeding, we make a simple observation: $(M^3,g)$ can be assumed to satisfy Assumption~A in Section~2. 

Indeed, by applying the curvature estimates for capillary surfaces (see, e.g. \cite{CEL24}), we can construct an outermost area-minimizing free boundary surface $\Sigma_- \in \mathcal{F}(M)$, as well as an innermost area-minimizing capillary surface $\Sigma_+ \in \mathcal{F}(M)$ with capillary angle $\alpha$, lying between $\partial_+ M$ and $\Sigma_-$. 

By Lemma~\ref{lem: spectral positivity on surfaces}, both $\Sigma_{\pm}$ are disks. Elementary topology then implies that the regions between $\partial_{\pm} M$ and $\Sigma_{\pm}$ are topological $3$-balls. Hence, the region bounded between $\Sigma_-$ and $\Sigma_+$ is a Riemannian cylinder. Note that replacing $(M^3,g)$ by the Riemannian cylinder bounded by $\Sigma_-$ and $\Sigma_+$ does not increase the free boundary $2$-systole. Thus, we may always assume $(M^3,g)$ satisfies Assumption A in the following discussion.

We record the following refined differential inequality for $I(s)$ under the curvature assumptions of Theorem \ref{thm: 2 systole estimate}. The notations are as in Section 2.
\begin{lemma}
    Let $(M^3,g)$ be a Riemannian cylinder satisfying Assumption A and the assumptions of Theorem \ref{thm: 2 systole estimate}. Given $\delta>0$ sufficiently small. For every $s_0\in (\delta, |\partial_0M|-\delta)$, there exists $\tilde{\delta}\in (0,\delta)$ and $\psi\in C^\infty(s_0-\tilde{\delta},s_0+\tilde{\delta})$ such that
\begin{enumerate}
    \item\label{sec3 item1} $\psi(s_0) = I(s_0)$;
    \item\label{sec3 item2} $\psi(s)\ge I(s)$ whenever $s\in (s_0-\tilde{\delta},s_0+\tilde{\delta})$;
    \item\label{sec3 item3} $\psi'(s_0) = \Phi(\sigma(s_0))$;
    \item\label{sec3 item4} \begin{align}\label{eq: 15}
    \psi''(s_0)\le \frac{1}{C_{\mathrm{iso}}^2R_0\psi^2}(1-\psi'(s_0)^2)(2\pi -\frac{R_0}{2}\psi(s_0))
\end{align}
For $C_{\mathrm{iso}} = \frac{1}{\sqrt{6}}$.
\end{enumerate}
\end{lemma}

\begin{proof}
Let $\psi$ be constructed as in Proposition \ref{prop: differential inequality for I}. 

If $s_0= s_{\sigma(s_0)}^{\pm}$, the inequality \eqref{eq: 15} follows as a consequence of  \eqref{eq: 5}, the facts $R_g\ge R_0$ in $M$ and $\bar{H}\ge 0$ on $\partial_0 M$, and an application of \eqref{eq: 34} to $\Sigma_{\sigma(s_0)}$. 

If $s_0\in (s_{\sigma(s_0)}^-,s_{\sigma(s_0)}^+)$, the strict monotonicity of the function $I$, \eqref{eq: 33} and the fact that $\psi'(s_0) = \Phi(\sigma(s_0))\in (0,1)$ implies that the right hand side of \eqref{eq: 15} is nonnegative at $s_0$. On the other hand, the first item of Proposition \ref{prop: differential inequality for I} says $\psi''(s_0) = 0$, which implies \eqref{eq: 15} remains true in this case.
\end{proof}

For $s\in (0,|\partial_0 M|)$, we define two mass functions as follows
\begin{equation}\label{eq: 16}
\begin{split}
    &m_0(s) = \exp\left({-\frac{4\pi}{C_{\mathrm{iso}}^2R_0I(s)}}\right)\cdot I(s)^{-\frac{1}{C_{\mathrm{iso}}^2}},\\
    &m(s) = m_0(s)\cdot(1-\bar{I}'(s)^2).
\end{split}
\end{equation}
From now on, we fix the isoperimetric constant $C_{\mathrm{iso}} = \frac{1}{\sqrt{6}}$.

\begin{lemma}\label{lem: The mass function is nondecreasing}
 $m(s)$ is nondecreasing.
\end{lemma}

\begin{proof}
We first treat the case that $I$ is smooth. Direct calculation yields
\begin{align*}
    m' = mI'\left[
\frac{4\pi}{C_{\mathrm{iso}}^2R_0I^2}
-\frac1{C_{\mathrm{iso}}^2I}
-\frac{2I''}{1-I'^2}
\right].
\end{align*}
Using \eqref{eq: 15}, we estimate
\[
\frac{2I''}{1-I'^2}
\le
\frac{2}{C_{\mathrm{iso}}^2R_0I^2}
\left(2\pi-\frac{R_0}{2}I\right).
\]
Substituting this bound into the previous expression and using $I$ is strictly increasing proves $m'(s)\ge 0$. In general case, the same reasoning of \cite[Lemma 3]{Bray97} shows $m(s)$ is nondecreasing in the distributional sense. From Lemma \ref{lem: semiconcavity of I}, $m(s)$ is right-continuous. This implies $m(s)$ is nondecreasing in usual sense.

\end{proof}

We define a function
\begin{align}\label{eq: 17}
    f(x) = \exp\left({-\frac{1}{C_{\mathrm{iso}}^2x}}\right)\cdot x^{-\frac{1}{C_{\mathrm{iso}}^2}}.
\end{align}
Taking the derivative, we have
\[
f'(x)
= \frac{1}{C_{\mathrm{iso}}^2}\exp\left({-\frac{1}{C_{\mathrm{iso}}^2 x}}\right)\cdot
x^{-\frac{1}{C_{\mathrm{iso}}^2}-1}
\left(
\frac{1}{x}-1
\right).
\]
Hence, $f(x)$ is increasing in $(0,1]$. Thus, for any $\alpha\in (0,\frac{\pi}{2}]$, there exists a unique value $\Psi(\alpha)\in (0,1]$ such that $f(\Psi(\alpha)) = f(1)\sin ^2\alpha$. Note this is equivalent to
\begin{align}\label{eq: 24}
    \frac{1}{x}+\log x = 1-2C_{\mathrm{iso}}^2\log\sin\alpha
\end{align}

\begin{lemma}
\begin{align*}
    \lim_{\alpha\to 0}\Psi(\alpha) = 0.
\end{align*}
More precisely, we have
    \begin{align*}
    \Psi(\alpha)\sim \frac{3}{|\log \alpha|}, \quad(\alpha\to 0).
    \end{align*}
\end{lemma}
\begin{proof}
The first statement follows from $\lim_{x\to 0}f(x) = 0$. For the second statement, we note that when $x$ is sufficiently small, $\exp({-\frac{1}{C_{\mathrm{iso}}^2 x}})\ll x^{-\frac{1}{C_{\mathrm{iso}}^2}}$. The desired conclusion then follows immediately.    
\end{proof}

\begin{lemma}\label{lem: comparison of mass function for different angles}
    For $\alpha<\theta_2<\theta_1<\frac\pi2$, $t_i = \Phi^{-1}(\cos\theta_i)\quad(i=1,2)$, and let $\Sigma_{t_i}\in\mathcal{E}_{t_i}$ be an area-minimizing surface with capillary angle $\theta_i$ in $(M,g)$. Then
    \begin{align}
        \exp\left({-\frac{4\pi}{C_{\mathrm{iso}}^2R_0|\Sigma_{t_1}|}}\right)\cdot |\Sigma_{t_1}|^{-\frac{1}{C_{\mathrm{iso}}^2}}\sin^2\theta_1\le  \exp\left({-\frac{4\pi}{C_{\mathrm{iso}}^2R_0|\Sigma_{t_2}|}}\right)\cdot |\Sigma_{t_2}|^{-\frac{1}{C_{\mathrm{iso}}^2}}\sin^2\theta_2.
    \end{align}
\end{lemma}
\begin{proof}
    Let $t_i = \Phi^{-1}(\cos\theta_i)$. Choose two sequences $s_{1,k}\searrow s_{t_1}^+$ and $s_{2,k}\nearrow s_{t_2}^-$, such that $I$ is differentiable at $s_{1,k}$ and $s_{2,k}$. By Lemma \ref{lem: I is increasing} and Lemma \ref{lem: The mass function is nondecreasing}, we have
    \begin{align*}
        &m_0(s_{1,k})(1-\Phi(\sigma(s_{1,k}))^2)= m_0(s_{1,k})(1-I'(s_{1,k})^2)\\
        \le& m_0(s_{2,k})(1-I'(s_{2,k})^2)= m_0(s_{2,k})(1-\Phi(\sigma(s_{2,k}))^2).
    \end{align*}
    Letting $k\to \infty$, and using the continuity of $I$ and $\sigma$, we have
    \begin{align}\label{eq: 77}
        m_0(s_{t_1}^+)\sin^2\theta_1 = m_0(s_{t_1}^+)(1-\Phi(\sigma(s_{t_1}^+))^2)\le m_0(s_{t_2}^-)(1-\Phi(\sigma(s_{t_2}^-))^2) = m_0(s_{t_2}^-)\sin^2\theta_2.
    \end{align}
    Let $\Sigma_{t_1}^+\in \mathcal{E}_{t_1}$ satisfy $|S(\Sigma_{t_1})| = s_{t_1}^+$, and $\Sigma_{t_2}^-\in \mathcal{E}_{t_2}$ satisfy $|S(\Sigma_{t_2})| = s_{t_2}^-$,  then $I(s_{t_1}^+) = |\Sigma_{t_1}^+|$ and $I(s_{t_2}^-) = |\Sigma_{t_2}^-|$. \eqref{eq: 77} yields
    \begin{align}\label{eq: 78}
        f\left(\frac{R_0|\Sigma_{t_1}^+|}{4\pi}\right)\sin^2\theta_1\le f\left(\frac{R_0|\Sigma_{t_2}^-|}{4\pi}\right)\sin^2\theta_2
    \end{align}
    Because $\Sigma_{t_1},\Sigma_{t_1}^+\in\mathcal{E}_{t_1}$ and $|S(\Sigma_{t_1}^+)|\ge |S(\Sigma_{t_1})|$ and $\Sigma_{t_1}$ is an area-minimizing capillary surface, we know that $|\Sigma_{t_1}|\le |\Sigma_{t_1}^+|$. Recall from \eqref{eq: 33} that $\frac{R_0|\Sigma_{t_1}^+|}{4\pi}\le 1$, we can use the monotonicity of $f$ over $(0,1]$ to deduce that
    \begin{align}\label{eq: 79}
        f\left(\frac{R_0|\Sigma_{t_1}|}{4\pi}\right)\le f\left(\frac{R_0|\Sigma_{t_1}^+|}{4\pi}\right).
    \end{align}
    Similarly,
    \begin{align}\label{eq: 80}
        f\left(\frac{R_0|\Sigma_{t_2}^-|}{4\pi}\right)\le f\left(\frac{R_0|\Sigma_{t_2}|}{4\pi}\right).
    \end{align}
    Combining \eqref{eq: 78}\eqref{eq: 79}\eqref{eq: 80}, we obtain the desired result.

\end{proof}

We are now in a position to prove Theorem \ref{thm: 2 systole estimate}.

\begin{proof}[Proof of Theorem \ref{thm: 2 systole estimate}]
     For $\alpha<\theta_2<\theta_1<\frac\pi2$, let $t_i = \Phi^{-1}(\cos\theta_i)$ and let $\Sigma_{t_i}\in\mathcal{E}_{t_i}$ be an area-minimizing surface with capillary angle $\theta_i$ in $(M,g)$ \quad $i=1,2$. By Lemma \ref{lem: comparison of mass function for different angles} we have
     \begin{align*}
         f\left(\frac{R_0|\Sigma_{t_1}|}{4\pi}\right)\sin^2\theta_1 \le f\left(\frac{R_0|\Sigma_{t_2}|}{4\pi}\right)\sin^2\theta_2.
     \end{align*}
     Recall from \eqref{eq: 33} we have $\frac{4\pi}{R_0|\Sigma_{\theta_i}|}\ge 1$  $(i = 1,2)$. Letting $\theta_2\to \alpha$, we obtain
     \begin{align*}
         f\left(\frac{R_0|\Sigma_{\theta_1}|}{4\pi}\right)\sin^2\theta_1\le f(1)\sin^2\alpha.
     \end{align*}
     When $\theta_1\to\frac{\pi}{2}$, $t_1\to 0$, and $\Sigma_{t_1}$ converges smoothly to a free boundary area-minimizing surface $\Sigma_0$ by possibly passing to a subsequence. 
     Since $f$ is increasing in $(0,1]$, we obtain
     \begin{align*}
         f\left(\frac{R_0|\Sigma_{0}|}{4\pi}\right)\le f(1)\sin^2\alpha.
     \end{align*}
     Therefore,
     \begin{align*}
         \sys_2(M,\partial_0M,g)\le |\Sigma_0|\le \frac{4\pi\Psi(\alpha)}{R_0}.
     \end{align*}
\end{proof}

\section*{Part II. The Coxeter smoothing theorem}
\section{A general Coxeter gluing and smoothing result}\label{section: Coxeter gluing and smoothing}

In this section, we describe a general mollification process that produces metrics with scalar curvature lower bound on closed manifolds by patching building blocks together. The result of this section is not restricted to $3$ dimensions or positive scalar curvature. We begin with the definition of manifolds with corners.

\begin{definition}[Riemannian manifold with corners]
     Let $M$ be a Hausdorff, second countable topological space. We say that $(M,g)$ is an $n$-dimensional smooth Riemannian manifold with corners, if for every point $x\in M$, there exist a closed $n$-dimensional Riemannian manifold $(\hat M^n,\hat g)$ without boundary, smooth domains-with-boundary
    \begin{align*}
        M_1,\dots,M_k\subset \hat M,
        \qquad 0\le k\le n,
    \end{align*}
    whose boundary hypersurfaces meet transversely, and an open set $U\subset \hat{M}$, such that $(M,g)$ is locally isometric to
    \begin{align*}
        (M_1\cap M_2\cap\dots\cap M_k\cap U, \hat{g}|_{M_1\cap M_2\cap\dots\cap M_k\cap U})).
    \end{align*}
\end{definition}

The next definition divides $\partial M$ into its facet part $\partial_{\fac} M$, and the corner part $\partial_{\cor} M$.

\begin{definition}\label{defn: building block}
   Let $(M,g)$ be a compact Riemannian manifold with boundary, possibly with corners, whose boundary is decomposed into a finite union of faces
\begin{align}\label{eq: 56}
    \partial M=\bigcup_{b\in \mathscr{B}} F_b .
\end{align}
Moreover, any two facets $F_{b_1}$ and $F_{b_2}$ $(b_1,b_2\in\mathscr{B})$ meet at a constant angle lying in $(0,\frac\pi2]$. We denote
\begin{align*}
    \partial_{\mathrm{fac}}M = \bigcup_{b\in \mathscr{B}} \Int F_b,
\end{align*}
to be the union of open facets, and $\partial_{\mathrm{cor}}M = \partial M\setminus \partial_{\mathrm{fac}}M$ to be the collection of the corner points.
\end{definition}

We stratify the corner locus according to codimension. For
$2\le k\le n$, let $\mathcal S^k$ denote the set of points
$x\in \partial_{\mathrm{cor}}M$ which have a neighborhood modeled on
\[
    [0,\infty)^k\times \mathbf R^{n-k},
\]
with $x$ corresponding to the origin in the first $k$ coordinates.
Equivalently, $\mathcal S^k$ consists of points where exactly $k$
local faces meet. Thus $\mathcal S^k$ is the codimension-$k$ strata
of the corner locus. %We write
%\[
%    \mathcal S^{\ge k}=\bigcup_{j=k}^n \mathcal S^j .
%\]

Next, we recall the definition of the Coxeter chamber and the Coxeter tessellation.

\begin{definition}[Coxeter chamber and Coxeter tessellation]\label{defn: Coxeter chamber and Coxeter tessellation}
Let $C\subset \mathbf R^n$ be a closed polyhedral cone with vertex at the
origin $\mathbf{0}$, and assume that $C$ has nonempty interior. Let
$\hat{F}_1,\ldots,\hat{F}_m$ be its facets. For each $i$, let $r_i$ denote the reflection across the
hyperplane containing $\hat{F}_i$.

We say that $C$ is a Coxeter chamber if the reflections $r_1,\ldots,r_m$
generate a discrete reflection group $G$ such that $C$ is a fundamental
chamber for the action of $G$ on $\mathbf{R}^n$. In particular, the dihedral angle of two adjacent facets $\hat{F}_i$ and $\hat{F}_j$ equals $\frac{\pi}{m_{ij}}$ for some integer $m_{ij}\geq 2$, and the corresponding reflections satisfy
\begin{align*}
    r_i^2 =1, (r_i r_j)^{m_{ij}} =1.
\end{align*}

Equivalently, the images of $C$ under $G$ have pairwise disjoint interiors and
cover $\mathbf{R}^n$:
\begin{align*}
    \mathbf R^n=\bigcup_{w\in G} w(C).
\end{align*}
The resulting decomposition
\begin{align*}
    \mathcal T_C=\{w(C):w\in G\}
\end{align*}
is called the Coxeter tessellation generated by $C$.
\end{definition}

\begin{example}\label{example: Coxeter chambers on the plane}
        For an integer $p\ge 2$, we denote by $C(p)\subset\mathbf{R}^2$
    the planar Coxeter chamber of angle $\pi/p$, defined by
    \begin{align*}
        C(p)
        =
        \left\{(r\cos\theta,r\sin\theta)\in\mathbf{R}^2
        \;\middle|\;
        r\ge 0,\ 0\le \theta\le \frac{\pi}{p}
        \right\}.
    \end{align*}
    The tessellation $\mathcal{T}_{C(p)}$ has $2p$ chambers, and the corresponding Coxeter group is the dihedral group $D_{2p}$\footnote{In general, we use $D_{2p}$ to denote the dihedral group, i.e. the symmetry group of a regular $2p$-gon, with presentation
\begin{align*}
D_{2p}=\langle a,b\mid a^{2p}=b^2=1,\ bab^{-1}=a^{-1}\rangle .
\end{align*}
}.
\end{example}

We also give the definition of the Coxeter building block.

\begin{definition}
    Let $(M^n,g)$ be a Riemannian manifold with boundary, possibly with corners. We say $(M^n,g)$ is a Coxeter building block, if 
    \begin{enumerate}
        \item any two facets of $(M^n,g)$ meet at a constant angle $\in (0,\frac\pi2]$;
        \item for any $x\in \partial_{\cor}M$, the rescaled pointed manifolds with corners
\begin{align*}
    (M,\lambda^2 g,x)
\end{align*}
converge, as $\lambda\to\infty$, in the pointed Cheeger--Gromov sense to a Euclidean Coxeter chamber $C_x$. We call $C_x$ the Coxeter chamber associated with $x$.
    \end{enumerate}
\end{definition}

We now state the main theorem of this section.

\begin{theorem}[Coxeter gluing and smoothing]\label{thm: coxeter gluing smoothing}
Let $n\ge 3$, $R_0\in\R$, and let
\[
    X
    =
    \left(\bigsqcup_{\lambda\in\Lambda} M_\lambda\right)\big/\sim
\]
be a closed topological \(n\)-manifold obtained
from finitely many compact Riemannian manifolds with corners
\((M_\lambda,g_\lambda)\) by identifying pairs of boundary facets.  Denote by \(\mathcal{S}\subset X\) the image of the glued facets, by \(\mathcal{S}_{\mathrm{fac}}\) the union of the open glued facets, and by \(\mathcal{S}_{\mathrm{cor}}\) the corresponding corner locus. Assume:
\begin{enumerate}
    \item For each \(\lambda\in\Lambda\), \((M_\lambda,g_\lambda)\) is a
    Coxeter building block in the sense of Definition
    \ref{defn: Coxeter chamber and Coxeter tessellation}. 

    \item The gluing is Coxeter-compatible: for every
    \(x\in\mathcal{S}_{\mathrm{cor}}\), there is a neighborhood \(\mathcal{U}\) of \(x\)
    in \(X\), a compact manifold \(N^{n-k}\), and a Euclidean Coxeter
    tessellation of a neighborhood of the origin in \(\mathbf R^k\), such that
    the chamber decomposition of \(\mathcal{U}\) is modeled on
    \[
        N^{n-k}\times \mathcal T_C ,
    \]
which means
    \begin{itemize}
        \item the tangent cones of all chambers incident to \(x\) fit
    together as the chambers of the Coxeter tessellation generated by the
    tangent Coxeter chamber;
    \item $G$ acts isometrically and transitively on chambers of $\mathcal{U}$.
    \end{itemize}

    \item The induced metrics are isometric on every pair of identified facets.  

    \item Each chamber has scalar curvature bounded below by \(R_0\):
    \[
        R_{g_\lambda}\ge R_0
        \qquad\text{in }\Int M_\lambda .
    \]

    \item Along every interior facet \(F_{\lambda}\sim F_{\mu}\), the
    mean-curvature jump is nonnegative:
    \[
        H_{\lambda}+H_{\mu}\ge 0,
    \]
    where \(H_{\lambda}\) and \(H_{\mu}\) are computed with respect to the
    outward unit normals of \(M_\lambda\) and \(M_\mu\), respectively.
\end{enumerate}

Then \(X\) carries a smooth structure, compatible with the smooth structures
on the interiors of the chambers, with the following properties.

\begin{enumerate}
    \item The piecewise metric
    \[
        \hat g|_{\Int M_\lambda}=g_\lambda
    \]
    descends to a globally defined Lipschitz metric on \(X\), smooth on
    \(X\setminus\mathcal{S}\).

    \item The scalar curvature of \(\hat g\) satisfies
    \[
        R_{\hat g}\ge R_0
    \]
    in the distributional sense of Lee--LeFloch \cite{LL15}.

    \item For every \(\varepsilon>0\), there exists a smooth Riemannian metric
    \(g_\varepsilon\) on \(X\) such that
    \[
        R_{g_\varepsilon}\ge R_0-\varepsilon,
    \]
    and
    \[
        g_\varepsilon\longrightarrow \hat g
    \]
    uniformly on \(X\) and smoothly on compact subsets of \(X\setminus
    \mathcal{S}\).  Moreover, \(g_\varepsilon\) may be chosen so that
    \[
        (1-\varepsilon)\hat g\le g_\varepsilon\le (1+\varepsilon)\hat g
    \]
    as quadratic forms.
\end{enumerate}

\end{theorem}

\subsection{The normal compatible coordinate map}

$\quad$

In \cite{Miao2002}, Miao glued two Riemannian manifolds with boundary to obtain a closed manifold, and produced a PSC metric on this new manifold. Using Fermi coordinates near the common boundary, \cite{Miao2002} specified a smooth structure on the resulting manifold and showed that the newly constructed metric is Lipschitz with respect to this smooth structure. To carry out this construction for manifolds with corners, we need an analogue of Fermi coordinates near corner points. This is the content of Theorem \ref{thm: fermi type coordinate for manifolds with corner} below. The goal of this subsection is to prove this result.

Let $(M^n,g)$ be a Riemannian manifold with corners. Recall that $\mathcal{S}^k\subset \partial_{\cor}M$ is an $n-k$ dimensional submanifold. It might have multiple connected components, some of which might be closed, and others are be open. The latter could happen only when $\mathcal{S}^{k+1}\ne \emptyset$. Let $N^{n-k}\subset\mathcal{S}^k$ be a compact submanifold, possibly with boundary. Note that $N^{n-k}$ and $\mathcal{S}^k$ are of the same dimension. It is possible that $N^{n-k}$ coincides with a closed component of $\mathcal{S}^k$, while it may also be a subdomain of an open component of $\mathcal{S}^k$. By a neighborhood $U$ of $N^{n-k}$ in $M$, we mean a subset of $M$ containing $N^{n-k}$ which is diffeomorphic to $N^{n-k}\times [0,\infty)^k$.

\begin{definition}
    A simplicial polyhedral cone $C\subset\mathbf{R}^k$ is a polyhedral cone with vertex $\mathbf{0}\in \mathbf{R}^k$ and $k$ flat facets $G_1,G_2,\dots,G_k$.
\end{definition}

We will work with coordinate maps in the following sense throughout this section.

\begin{definition}[The coordinate map modeled on $N^{n-k}$]

\quad

    Let $(M^n,g)$ be a Riemannian manifold with corners. Let $N^{n-k}\subset \mathcal{S}^k$ be a compact submanifold, possibly with boundary. A coordinate map modeled on $N^{n-k}$ is a smooth map
    \begin{align*}
        \Psi: U\longrightarrow N^{n-k}\times C
    \end{align*}
    which maps a neighborhood $U$ of $N^{n-k}$ in $M$ diffeomorphically onto an open subset of $N^{n-k}\times C$, where $C$ is a simplicial polyhedral cone in $\mathbf{R}^k$. Moreover, we require
    \begin{align}\label{eq: 96}
        \Psi(z) = (z,\mathbf{0})\text{ for all }z\in N^{n-k}.
    \end{align}
\end{definition}

To begin with, we introduce the notion of the normal compatibility condition.
\begin{definition}[The normal compatibility condition]

\quad

\begin{enumerate}
    \item Let $C\subset \mathbf{R}^k$ be a polyhedral cone with flat facets $G_1,\dots,G_k$, and let $(N^{n-k},\omega_N)$ be a compact Riemannian manifold, possibly with boundary. On $N^{n-k}\times (C\cap B_1(\mathbf{0}))$, let $g$ be a Riemannian metric and let $h=dy_1^2+\cdots+dy_k^2+\omega_N$ be the product metric. For $\hat F_i=N\times G_i$, denote by $n_i$ and $\hat n_i$ the outward unit normals to $\hat F_i$ with respect to $g$ and $h$, respectively. We say that $g$ satisfies the normal compatibility condition if $n_i=\hat n_i$ on $\hat F_i\cap B_1(\mathbf{0})$ for every $i=1,\dots,k$.
    \item Let $(M^n,g)$ be a Riemannian manifold with corner, and let $N^{n-k}\subset \mathcal{S}^k$ be an $n-k$ dimensional compact connected manifold, possibly with boundary. Assume $N^{n-k}$ is contained in $k$ facets $F_1,F_2,\dots, F_k$. Let $U$ be a neighborhood of $N$ in $M$, and let $\Psi: U\longrightarrow N\times (C\cap B_1(\mathbf{0}))$ be a smooth coordinate map. We say $g$ satisfies the normal compatibility condition with respect to the coordinate map $\Psi$, if $(\Psi^{-1})^*g$ satisfies the normal compatibility condition.
\end{enumerate}
\end{definition}

%The following theorem shows the existence of a fermi type coordinate around corner points. With this we can describe the smooth structure of the patched manifold in a clear way.

Now we can state the main theorem in this subsection.
\begin{theorem}[Existence of normal compatible coordinate maps] \label{thm: fermi type coordinate for manifolds with corner}

\quad

    Let $(M^n,g)$ be a compact Riemannian manifold with corners, such that any two facets of $(M^n,g)$ meet at a constant angle in $(0,\frac\pi2]$. Then there exists a diffeomorphism $\Phi: M\longrightarrow M$, satisfying 
\begin{itemize}
    \item $\Phi\in C^\infty(M\backslash \partial_{\cor}M)\cup C^{1,1}(M)$;
    \item $\Phi(x) = x$ for all $x\in\partial_{\cor}M$;
    \item $d\Phi_x = \Id$  for all $x\in\partial_{\cor}M$,
\end{itemize}
such that $g' = \Phi^*g$ has the following properties: For any $p\in \partial_{\cor} M$, there exists 
    \begin{itemize}
        \item a compact connected manifold (possibly with boundary) $N^{n-k}$ satisfying $p\in \Int N^{n-k}\subset \mathcal{S}^k$ which lies in the intersection of  $k$ facets $F_1,F_2,\dots, F_k$;
        \item a neighborhood $U$ of $N$ in $M$;
        \item a simplicial polyhedral cone $C\subset \mathbf{R}^k$ with vertex $\mathbf{0}$ and $k$ facets $G_1,G_2,\dots, G_k$;
        \item a smooth coordinate map $\Psi: U\longrightarrow N^{n-k}\times (C\cap B_1(\mathbf{0}))\subset N^{n-k}\times \mathbf{R}^k$;
    \end{itemize}
such that $g'$ satisfies the normal compatibility condition with respect to the coordinate map $\Psi$. In particular, we have
    \begin{enumerate}

        \item $(\Psi^{-1})^*g' = dy_1^2+dy_2^2+\dots +dy_k^2+\omega_{N}$ along $N$;

        \item \begin{align}
            \angle_{g}(F_i,F_j) = \angle_{g'}(F_i,F_j) = \angle_{\mathbf{R}^k}(G_i,G_j), \quad i,j\in \{1,2,\dots,k\};
        \end{align}

        \item $g'$ is Lipschitz on $M$.
    \end{enumerate}
\end{theorem}

$\quad$

The remaining part of the subsection is devoted to the proof of Theorem \ref{thm: fermi type coordinate for manifolds with corner}. We begin with the following elementary lemma.

\begin{lemma}\label{lem: linear algebra}
    Let $C\subset\mathbf{R}^k$ be a polyhedral cone with $k$ flat facets $G_1,G_2,\dots, G_k$ passing through the origin $\mathbf{0}\in \mathbf{R}^k$. Let $N^{n-k}$ be a compact manifold possibly with boundary, and denote $\hat{F}_i = N\times G_i\subset N\times \mathbf{R}^k (i=1,2,\dots,k)$. Let $\mathcal{C}$ be another polyhedral domain in $N\times \mathbf{R}^k$ with $k$ smooth facets $F_1', F_2', \dots, F_k'$ intersecting along $N = N\times \{\mathbf{0}\}$, satisfying $F_i'$ is tangent to $\hat{F}_i$ along $N$ $(i = 1,2,\dots, k)$. Then there is a smooth map $\zeta$ defined on a neighborhood of $N\subset N\times \mathbf{R}^k$ satisfying
    \begin{itemize}
        \item $\zeta(F_i')\subset \hat{F}_i, (i=1,2,\dots,k)$;
        \item $\zeta(x) = x$ for $x\in N$;
        \item $D\zeta \equiv\Id$ on $N$.
    \end{itemize}
    In particular, $\zeta$ is a diffeomorphism in a small neighborhood of $N\subset N\times \mathbf{R}^k$.
\end{lemma}

\begin{proof}
    We use $z$ to denotes points in $N$ and $(z,y_1,y_2,\dots,y_k)$ to denote points in $N\times \mathbf{R}^k$. For simplicity of exposition, we use the notation $\mathbf{y} = (y_1,y_2,\dots,y_k)$. Let
    \begin{align*}
        C_0 = \{\mathbf{y}\in \mathbf{R}^k, y_i\ge 0, i = 1,2,\dots,k\}.
    \end{align*}
    Note that there exists a linear isomorphism $T:\mathbf{R}^n\longrightarrow \mathbf{R}^n$ which maps $C$ to $C_0$. It suffices to prove the lemma for $C = C_0$, since we can construct $\zeta = T^{-1}\zeta_0 T$ once we obtain the map $\zeta_0$ for the cone $C_0$. Now we have $\hat{F}_i = \{(z,\mathbf{y}), y_i= 0\}\ (i = 1,2,\dots,k)$. Since $F_i'$ is tangent to $\hat{F}_i$ along ${N}$, it can be written as a graph over $\hat{F}_i$:
    \begin{align*}
        y_i = f_i(z,y_1,\dots,y_{i-1},y_{i+1},\dots,y_k)
    \end{align*}
    satisfying $Df_i(z,\mathbf{0})\equiv 0$ for all $z\in N$. Let
    \begin{align*}
        \rho_i(z,\mathbf{y}) = y_i-f_i(z,y_1,\dots,y_{i-1},y_{i+1},\dots,y_k)
    \end{align*}
    and define
    \begin{align*}
        \zeta(z,\mathbf{y}) = (z,\rho_1(z,\mathbf{y}),\dots,\rho_k(z,\mathbf{y})).
    \end{align*}
    Then $D\zeta(z,\mathbf{0}) \equiv \id$. This completes the construction of $\zeta$.
\end{proof}

We also need the following basic property derived from the normal compatibility condition.

\begin{lemma}\label{lem: nice form of the metric on the edge assuming angle matching condition}
    Let $(M^n,g)$ be a compact Riemannian manifold with corners. Let $N^{n-k}\subset \mathcal{S}^k$ be a compact connected submanifold. Assume $N^{n-k}$ is contained in $k$ facets $F_1,F_2,\dots, F_k$. Let $U$ be a tubular neighborhood of $N^{n-k}$ and let $\Psi: U\longrightarrow N^{n-k}\times C$ be a smooth coordinate map, where $C\subset \mathbf{R}^k$ is a simplicial polyhedral cone. Denote the facets of $N^{n-k}\times C$ by $\hat{F}_i\quad (i=1,2,\dots,k)$.
    
    Let $h = dy_1^2+dy_2^2+\dots +dy_k^2+\omega_N$ be the standard product metric on $N\times \mathbf{R}^k$, and let $n_i,\hat{n}_i$ be outer unit normal of $F_i,\hat{F}_i$ with respect to $g$ and $h$. Assume that 
    \begin{itemize}
        \item $g$ satisfies the normal compatibility condition with respect to the coordinate map $\Psi$;
        \item $\angle_g(F_i,F_j) = \angle_{h}(\hat{F}_i,\hat{F}_j)$.
    \end{itemize}
    Then
    \begin{align*}
        (\Psi^{-1})^*g = h
    \end{align*}
    along $N$.
\end{lemma}
\begin{proof}
By the normal compatibility condition we know that $d\Psi(n_i) = \hat{n}_i$. Since $\angle_g(F_i,F_j) = \angle_{h}(\hat{F}_i,\hat{F}_j)$, we have
    \begin{align}\label{eq: 97}
        (\Psi^{-1})^*g(\hat{n}_i,\hat{n}_j) = g(n_i,n_j) = h(\hat{n}_i,\hat{n}_j).
    \end{align}
    By \eqref{eq: 96}, we have $\Psi_*X =X$ for all $X\in TN$. We also have
    \begin{align}\label{eq: 98}
        (\Psi^{-1})^*g(X,\hat{n}_i) = g(X,n_i) = 0 = h(X,\hat{n}_i)
    \end{align}
    for all $X\in TN$. Combining \eqref{eq: 97}\eqref{eq: 98} and using $T_z(N\times \mathbf{R}^k) = T_z N\oplus \mathrm{Span}\{\hat{n}_1,\hat{n}_2,\dots,\hat{n}_k\}$ for $z\in N$, we obtain the desired result.
\end{proof}

With the preparation above, we now begin to prove Theorem \ref{thm: fermi type coordinate for manifolds with corner}. As the first step in the proof, we find a smooth coordinate around each corner point, such that the normal compatibility condition is satisfied on $N$.

\begin{lemma}\label{lem: Putting the metric into the cone}
    Let $(M^n,g)$ be a compact Riemannian manifold with corner. Let $N^{n-k}\subset \mathcal{S}^k$ be a compact connected submanifold. Assume that $N^{n-k}$ is contained in exactly $k$ facets
$F_1,F_2,\dots,F_k$. Then one can find
\begin{enumerate}
    \item a simplicial polyhedral cone $C\subset \mathbf{R}^k$ with facets
    $G_1,G_2,\dots,G_k$ and vertex $\mathbf{0}$, such that
    \begin{align}\label{eq: angle match 1}
        \angle_g(F_i,F_j)
        =
        \angle_{\mathbf{R}^k}(G_i,G_j)
    \end{align}
    for all $1\leq i<j\leq k$;

    \item a tubular neighborhood $U$ of $N^{n-k}$ in $M^n$, together with
    a smooth coordinate map
    \[
        \Psi: U \longrightarrow N^{n-k}\times (C\cap B_1(\mathbf{0}))
        \subset N^{n-k}\times \mathbf{R}^k,
    \]
    such that
    \[
        \Psi(U\cap F_i)=N^{n-k}\times (G_i\cap B_1(\mathbf{0}))
    \]
    for each $i$.
\end{enumerate}
Moreover, this construction can be arranged so that
\begin{align}\label{eq: 99}
    (\Psi^{-1})^*g
    =
    dy_1^2+\cdots+dy_k^2+\omega_N
\end{align}
along $N^{n-k}\times \{\mathbf{0}\}$.

\end{lemma}

%\begin{lemma}\label{lem: Putting the metric into the cone}
   %Let $(M^n,g)$ be as in Theorem \ref{thm: fermi type coordinate for manifolds with corner}. Then by selecting $\Phi = \Id$, we can find a coordinate chart $U$ for each corner point $x_0\in\partial_{\cor}M$ that satisfies (1)(2)(3) and (4) of Theorem \ref{thm: fermi type coordinate for manifolds with corner}.
%\end{lemma}

\begin{proof}
    We prove the lemma holds for $(M,g)$ by induction on $k$. 
    
    When $k = 1$, we have $N^{n-1} = F_1$, and the lemma is equivalent to the existence of Fermi coordinate near $N^{n-1}$. For $k\ge 2$, let's assume the lemma is true for $k-1$. We divide the proof into $3$ steps:

    \textbf{Step 1: Determine the coordinate map $\Psi_0$ on the bottom facet $F_k$.}

$\quad$
    
     Since $F_1,F_2,\dots, F_k$ intersect along $N^{n-k}$, an $n-k$ dimensional submanifold, the tangent cone of $(M,g)$ at points in $N^{n-k}$ splits as $C\times \mathbf{R}^{n-k}$, with $C\subset\mathbf{R}^k$ satisfying \eqref{eq: angle match 1}. Denote $E_{ij} = F_i\cap F_j$, $\hat{E}_{ij} = \hat{F}_i\cap \hat{F}_j = N\times (G_i\cap G_j) = N\times L_{ij}$ for $i = 1,2,\dots, k-1$. Let $C_0 = G_k$, then $C_0\subset\mathbf{R}^{k-1}$ is a polyhedral cone with $k-1$ facets $L_{ik}\ (i = 1,2,\dots,k-1)$. Moreover, $(F_k,g|_{F_k})$ is an $n-1$-dimensional Riemannian manifold with corner. By using the inductive hypothesis on $(F_k,g|_{F_k})$ and $N^{n-k}$, there is a tubular neighborhood $V$ of $N^{n-k}$ in $F_k$, and a smooth coordinate map
     \begin{align*}
         \Psi_0: V\longrightarrow N^{n-k}\times (C_0\cap B_1(\mathbf{0}))\subset N^{n-k}\times \mathbf{R}^{k-1}
     \end{align*}
     which maps $E_{ik}$ to $\hat{E}_{ik} = N^{n-k}\times L_{ik}$. Moreover, it holds 
     \begin{itemize}
         \item $\Psi_0(z) = (z,\mathbf{0})\in N^{n-k}\times \mathbf{R}^{k-1}$ for $z\in N^{n-k}\subset V$;
         \item $(\Psi_0^{-1})^*g = dy_1^2+dy_2^2+\dots +dy_{k-1}^2+\omega_{N}$ along $N$;
         \item $\angle_{g}(E_{ik},E_{jk}) = \angle_{\mathbf{R}^{k-1}}(L_{ik},L_{jk})$.
     \end{itemize}

%$\Psi_0$ on the bottom face $F_k$ and the last coordinate $y_k$ give all the information in defining the full coordinate map $\Psi$. Thus, it remains only to define the last coordinate $y_k$, which is the purpose of Steps 2 and 3 below.

    \textbf{Step 2: Define $y_k$ on a hypersurface perpendicular to $F_k$.}

$\quad$
    
    First, we define the $y_k$ coordinate to be $0$ on $F_k$. Next we consider the vertical direction: With respect to the metric $g$, for $1\le i\le k$, let $n_i$ denote the unit normal of $F_i$ pointing outward $M$; For $1\le i\le k-1$, let $\nu_i$ denote the unit normal of $E_{ik}$ lying in $TF_k$ pointing outward $F_k$. For $z\in N$ and $y_k\in \mathbf{R}$, We define
   \begin{align}\label{eq: 66}
       \Psi\left (\exp_{(z,\underbrace{(0,\dots,0)}_{k-1\text{ copies of }0})}(-y_k n_k)\right) = (z,\underbrace{(0\dots,0,y_{k})}_{k-1\text{ copies of }0}).
   \end{align}
   
   Summarizing the definitions above, $y_k$ has already been defined for points with image $\{(z,\mathbf{y})\in N\times \mathbf{R}^k,\quad y_k = 0\}$ and $\{(z,\mathbf{y})\in N\times \mathbf{R}^k,\quad y_1 = y_2 = \dots = y_{k-1} = 0\}$. Let $\{\Sigma_t\}_{-\epsilon<t<\epsilon}$ be a smooth codimension $1$ foliation, such that $F_k\subset \Sigma_0$ and
   \begin{align*}
        \{ (\exp_{(z,\underbrace{(0,\dots,0)}_{k-1\text{ copies of }0})}(-y_k n_k),\quad z\in N\}\subset \Sigma_{y_k}.
   \end{align*}
   This enable us to define
   \begin{align*}
       y_k(x) = t\text{ if }x\in \Sigma_t.
   \end{align*}

  Using the same foliation argument for $z$ and $y_1,y_2,\dots,y_{k-1}$, we can extend $\Psi_0$ smoothly to
\begin{align*}
    \Psi_0:U\longrightarrow N^{n-k}\times \mathbf{R}^{k-1},
\end{align*}
where $U$ is a tubular neighborhood of $N^{n-k}$ in $M$. We then define
\begin{align}\label{eq: 117}
    \Psi:U\longrightarrow N^{n-k}\times \mathbf{R}^k,\qquad
    \Psi(x)=(\Psi_0(x),y_k(x)).
\end{align}
The extension $\Psi_0$ can be chosen so that the map $\Psi$ defined in \eqref{eq: 117} is compatible with \eqref{eq: 66}.
   
   Since  $\Psi_0$ is a diffeomorphism when restricted to $F_k$, and $d\Psi(n_k) = -\partial_{y_k}$ on $N^{n-k}$, we have $d\Psi\ne 0$ along $N^{n-k}$. Hence, $\Psi$ is a diffeomorphism by possibly shrinking $U$ small enough. Moreover, we have
   \begin{align}\label{eq: 100}
       (\Psi^{-1})^*g = dy_k^2+(\Psi_0^{-1})^*(g|_{F_k}) =  dy_1^2+dy_2^2+\dots+dy_{k}^2 +\omega_N
   \end{align}
   along $N^{n-k}$.

\textbf{Step 3: Correcting the coordinate map \(\Psi\) by a diffeomorphism.} 

$\quad$

   \textbf{Claim:} As hypersurfaces in $N\times \mathbf{R}^k$, $\Psi(F_i)$ is tangent to $\hat{F}_i = N\times G_i$ along $N\quad (i=1,2,\dots,k)$.

   Consider the metric $h = dy_1^2+dy_2^2+\dots+dy_{k}^2 +\omega_N$ defined on $N\times \mathbf{R}^k$. Define $\hat{n}_i$ to be the unit normal of $\hat{F}_i$ pointing outward $M$ with respect to the metric $h$ $(i=1,2,\dots,k)$; Let $\hat{\nu}_i$ denote the unit normal of $\hat{E}_{ik}$ lying in $T\hat{F}_k$ pointing outward $\hat{F}_k$ with respect to the metric $h\quad (i=1,2,\dots,k-1)$. Since $\Psi(F_k) = \hat{F}_k$  and $h|_{T\hat{F}_k} = dy_1^2+dy_2^2+\dots+dy_{k-1}^2 +\omega_N = (\Psi^{-1})^*g|_{T\hat{F}_k}$ on $N$, we conculde that
   \begin{align}\label{eq: 63}
       d\Psi(\nu_i) = \hat{\nu}_i\text{ on } N\quad (i=1,2,\dots,k-1).
   \end{align}
   
   Since $\nu_{i}\perp_g E_{ij}$ and $d\Psi^{-1}(\partial_{y_k})\perp_g E_{ij}$ for $i = 1,2,\dots, k-1$, we have $E_{ij}^{\perp_g} = \mathrm{Span}\{\nu_i,d\Psi^{-1}(\partial_{y_k})\}$. Because $n_i\perp_g E_{ij}$, $n_i$ can be written as
   \begin{align*}
       n_i = \cos\alpha\cdot d\Psi^{-1}(\partial_{y_k})+ \sin\alpha\cdot \nu_i.
   \end{align*}
   Hence,
   \begin{align}\label{eq: 64}
       &g(n_i,n_k) = g( \cos\alpha\cdot d\Psi^{-1}(\partial_{y_k})+ \sin\alpha\cdot \nu_i,-d\Psi^{-1}(\partial_{y_k}))\\
       =& -(\Psi^{-1})^*g(\partial_{y_k},\partial_{y_k})\cos\alpha=-\cos\alpha
   \end{align}
   along $N$.
   
   By writing
   \begin{align}\label{eq: 125}
       \hat{n}_i = \cos\hat{\alpha}\cdot \partial_{y_k}+\sin\hat{\alpha}\cdot \hat{\nu}_i
   \end{align}
   and applying the same argument as above (under the metric $h$ instead of $g$), we obtain
   \begin{align}\label{eq: 65}
       h(\hat{n}_i,\hat{n}_k) = -\cos\hat{\alpha}
   \end{align}
   along $N$. Recall that $C$ satisfies \eqref{eq: angle match 1}, so
   \begin{align*}
       g(n_i,n_k) = -\cos\angle_g(F_i,F_k) = -\cos\angle_{\mathbf{R}^k}(G_i,G_j) = -\cos\angle_h(\hat{F}_i,\hat{F}_j) = h(\hat{n}_i,\hat{n}_j).
   \end{align*}

   In conjunction with \eqref{eq: 64} and \eqref{eq: 65}, we conclude that $\cos\alpha = \cos\hat{\alpha}$. Plugging this into \eqref{eq: 63} and \eqref{eq: 125}, we obtain $d\Psi(n_i) = \hat{n}_i$ along $N$. This verifies the Claim.

   Once the Claim is verified, $\Psi(F_i)$ satisfies the property satisfied by $F_i'$ in Lemma \ref{lem: linear algebra}. By Lemma \ref{lem: linear algebra}, there exists a diffeomorphism $\zeta$, defined on a neighborhood of $N^{n-k}$ in $N^{n-k}\times \mathbf{R}^k$, which fixes $N^{n-k}\times \{0\}$, satisfies $D\zeta \equiv \Id$ along $N^{n-k}$, and maps $\Psi(U)$ into $N^{n-k}\times C$. We get the desired coordinate map by substituting $\Psi$ by $\zeta\circ\Psi$. The only point is to check \eqref{eq: 99}, which follows from \eqref{eq: 100} and the properties satisfied by $\zeta$ here. This completes the proof.

\end{proof}

The second step in the proof of Theorem \ref{thm: fermi type coordinate for manifolds with corner} is to show that the theorem holds in a local coordinate system $\Psi: U\longrightarrow N^{n-k}\times (C\cap B_1(\mathbf{0}))$, with the correction map $\phi$ only defined locally on the neighborhood $U$.
%Next, we show that Theorem \ref{thm: fermi type coordinate for manifolds with corner} holds in a local coordinate chart $U$ if we replace $g|_U$ by $\Phi^*g|_U$ for some diffeomorphism $\Phi\in C^\infty(\Int M\cup\partial_{\fac}M)\cap C^{1,1}(\bar{M})$.

\begin{lemma}\label{lem: straightening the normal vector locally}
Let $(N^{n-k},\omega_N)$ be a compact Riemannian manifold, and let $C\subset\mathbf{R}^k$ be a simplicial polyhedral cone with $k$ facets $G_1,G_2,\dots, G_k$ and vertex $\mathbf{0}$. Let $g$ be a Lipschitz metric on $N^{n-k}\times C$ which is smooth away from $N^{n-k}\times \partial_{\cor}C$, satisfying $g = h=dy_1^2+dy_2^2+\dots dy_k^2+\omega_N$ along $N = N\times \{\mathbf{0}\}$. Denote $\hat{F}_i = N\times G_i (i=1,2,\dots,k)$, and assume that
\begin{align*}
    \angle_g(\hat{F}_i,\hat{F}_j) = \angle_h(\hat{F}_i,\hat{F}_j) = \angle_{\mathbf{R}^k}(G_i,G_j).
\end{align*}

Then there is a neighborhood $W\subset N\times C$ of $N = N\times \{\mathbf{0}\}$ in $N\times \mathbf{R}^k$ and a $C^{1,1}$ map $\phi: W\longrightarrow N\times C$, smooth away from $N^{n-k}\times \partial_{\cor}C$, mapping $W$ diffeomorphically to its image, such that

    \begin{enumerate}
        \item $\phi(x) = x$ for all $x\in (N^{n-k}\times \partial_{\cor}C)\cap W$;
        \item $d\phi_x = \Id$ for all $x\in N$;
        \item $d\phi_x(n_i) = \hat{n}_i$ for all $x\in \hat{F}_i$ $(i=1,2,\dots,k)$, where $n_i$ and $\hat{n}_i$ are unit normal of $\hat{F}_i$ with respect to $g$ and $h$  pointing outward $N\times C$.
    \end{enumerate}
\end{lemma}

\begin{proof}
    We begin with the simplest case when $k=2$. In this case, $\hat{F}_1\cap \hat{F}_2 = N^{n-2}$. Since $g = dy^2_1+dy^2_2+\omega_{{N}}$ along $N^{n-2}$, we know that $n_i = \hat{n}_i$ along $N^{n-2}$ for $i=1,2$. 
    
    We first consider the case that $N^{n-2}$ is diffeomorphic to a ball $B_1(\mathbf{0})\subset\mathbf{R}^{n-2}$, which enables us to view $N\times C$ as a subset of $\mathbf{R}^n$. For $a\in \hat{F}_i$, let $A_a:T_a\mathbf{R}^n\longrightarrow T_a\mathbf{R}^n$ be the linear map satisfying $A_a(v) = v$ for $v\in T_a \hat{F}_i$ and $A_a(n_i) = \hat{n}_i$. Since $n_i = \hat{n}_i$ along $N^{n-2}$, we have $A_a = \Id$ for all $a\in N$. Since $g$ is  Lipschitz in $N^{n-2}\times C$ and smooth away from $N^{n-2}\times \partial_{\cor}C$, $n_i(a)$ is a smooth vector for $a\in\Int \hat{F}_i$ and is Lipschitz up to $\partial\hat{F}_i$. Consequently, $A_a$ is a smooth matrix  for $a\in\Int \hat{F}_i$ and is Lipschitz up to $\partial\hat{F}_i$. Since $A_a$ is defined to be $\Id$ for   $a\in N=\hat{F}_1\cap\hat{F}_2$, we conclude that $A_a$ is Lipschitz on $\hat{F}_1\cup \hat{F}_2$.

    Next, we verify \eqref{eq: 115} in Lemma \ref{lem: C11 Whitney extension smooth away from corner 2}. If $a$ and $b$ belong to the same $\hat{F}_i$, we have $A_b(v) = v$ for $v\in T{\hat{F}_i}$. Because $a-b\in T\hat{F}_i$, it follows that
    \begin{align*}
        a-b-A_b(a-b) = 0;
    \end{align*}
    If $a$ and $b$ belong to different $\hat{F}_i$, let $c$ be a point in $N^{n-2}$. We have
    \begin{equation}\label{eq: 70}
        \begin{split}
            |a-b-A_b(a-b)| =& |(\Id-A_b)(a-c)+c-b-A_b(c-b)|\\ 
        =& |(\Id-A_b)(a-c)|\le M|b-c|\cdot|a-c|
        \end{split}
    \end{equation}
    Since $\angle acb\ge \angle_{\mathbf{R}^n}(\hat{F}_1,\hat{F}_2) = \gamma_{12}$, we have
    \begin{equation}\label{eq: 71}
        \begin{split}
            |a-b|^2\ge& |a-c|^2+|b-c|^2-2\cos\gamma_{12}|a-c|\cdot|b-c|\\
        \ge& (1-\cos\gamma_{12})(|a-c|^2+|b-c|^2).
        \end{split}
    \end{equation}
    Combining \eqref{eq: 70} and \eqref{eq: 71}, we obtain\eqref{eq: 115}. Therefore, when $N^{n-2}$ is diffeomorphic to a ball, we can apply Lemma \ref{lem: C11 Whitney extension smooth away from corner 2} to obtain the desired $C^{1,1}$ extensio $\phi$. In the general case when $N^{n-2}$ is a manifold, we can apply Lemma \ref{lem: C11 Whitney extension smooth away from corner 2 relative version} (the relative version of Lemma \ref{lem: C11 Whitney extension smooth away from corner 2}) repeatedly to obtain the desired $C^{1,1}$ map $\phi$.

    Now we consider the case $k=3$. Denote $\hat{E}_{ij} = \hat{F}_i\cap \hat{F}_j$. Let $\nu_{ij,i}$ ($\nu_{ij,j}$) denote the unit inner normal of $\hat{E}_{ij}$ in $\hat{F}_i$ ($\hat{F}_j$) with respect to $g$, and let $\hat{\nu}_{ij,i}$ ($\hat{\nu}_{ij,j}$) denote the unit inner normal of $\hat{E}_{ij}$ in $\hat{F}_i$ ($\hat{F}_j$) with respect to the Euclidean metric.
    
    We focus on a single facet $\hat{F}_1$, with the boundary decomposition $\partial \hat{F}_1 = \hat{E}_{12}\cup \hat{E}_{13}$, and $\hat{E}_{12}\cap \hat{E}_{13} = N^{n-3}$. Because $g = dy_1^2+dy_2^2+\dots +dy_3^2+\omega_{N}$ along $N^{n-3}$, we have $\nu_{12,1} = \hat{\nu}_{12,1}$ and $\nu_{13,1} = \hat{\nu}_{13,1}$ along $N^{n-3}$. Now it is easy to see that $(\Psi^{-1})^*g|_{\hat F_1}$ satisfies all the assumptions of the present lemma in the case $k=2$. Therefore, there exists a $C^{1,1}$ diffeomorphism
$\phi_1:\hat F_1\to \hat F_1$, smooth away from $ N^{n-3}$, fixing
$\partial\hat F_1$ pointwise, such that
\[
    d\phi_1(\nu_{12,1})=\hat\nu_{12,1},\qquad
    d\phi_1(\nu_{13,1})=\hat\nu_{13,1}
\]
along $\partial\hat F_1$, and $d\phi_1 = \Id \text{ along }N^{n-3}$.

Similarly, we can construct $\phi_2$ and $\phi_3$ satisfying similar properties on $\hat{F}_2$ and $\hat{F}_3$. Since $\phi_i$ fixes the edges $\hat{E}_{ij}$, we know that $\phi_1,\phi_2$ and $\phi_3$ are compatible along $\hat{E}_{12},\hat{E}_{23}$ and $\hat{E}_{13}$. Moreover, $d\phi_i (i=1,2,3)$ patches into an identity self map of $T(N\times \mathbf{R}^3)$ along $N^{n-3}$. Then we find a $C^{1,1}$ diffeomorphism $\phi_0$ defined on a neighborhood of $N^{n-3}\subset N^{n-3}\times C$, smooth away from $N^{n-3}\times \partial_{\cor}C$, satisfying $\phi_0|_{\hat{F}_i} = \phi_i (i=1,2,3)$.

Since $\angle_{g}(\hat{F}_i,\hat{F}_j) = \angle_{h}(\hat{F}_i,\hat{F}_j):=\gamma_{ij}$, we have
\begin{align*}
    n_1 = \cot\gamma\cdot \nu_{12,1}-\frac{1}{\sin\gamma}\cdot \nu_{12,2} = \cot\gamma\cdot d\phi_0(\nu_{12,1})-\frac{1}{\sin\gamma}\cdot d\phi_0(\nu_{12,2}) = d\phi_0(n_1)
\end{align*}
along $\hat{E}_{12}$. The same thing also holds for $n_1,n_2,n_3$ on other edges. Consequently, $\phi_0$ satisfies (1)(2) and satisfies (3) on $\hat{E}_{12}\cup \hat{E}_{23}\cup\hat{E}_{13}$. By running the argument in the $k=2$ case again, we can find a $\phi$ which satisfies all conditions. This completes the proof in the case $k=3$.

For general $k$, the lemma can be proved in the same way as in the $k=3$ case by induction on $k$.
\end{proof}

%$\begin{lemma}\label{lem: compatibility between different coordinates}
%\end{lemma}

Before proceeding, we introduce a procedure for constructing new coordinate maps from old ones. Let $P^{n-k}\subset\mathcal{S}^{k}$ and assume that $P$ is contained in the intersection of $k$ facets of $M$: $F_1,F_2,\dots F_k$. Assume that there is a neighborhood $U$ of $P$ in $M$ and a smooth coordinate map $\Psi: U\longrightarrow P^{n-k}\times C$, where $C\subset \mathbf{R}^k$ is a simplicial polyhedral cone. Denote $\mathcal{L} = F_2\cap F_3\cap \dots F_k$ to be an edge containing $P^{n-k}$, then $\dim \mathcal{L} = n-k+1$. We will find a submanifold $N^{n-k+1}\subset \mathcal{L}$, a neighborhood $V$ of $N^{n-k+1}$ in $M$, and construct a smooth coordinate map $\tilde{\Psi}: V\longrightarrow N^{n-k+1}\times C'$, with $C'\subset \mathbf{R}^{k-1}$ being a polyhedral cone with $k-1$ facets. %In fact, the tangent cone of $M$ along $\mathcal{L}$ is isometric to $C'\times \mathbf{R}$.

For $T\in O(k)$, we define $\Psi^T:  U\longrightarrow P^{n-k}\times T(C)\subset P^{n-k}\times \mathbf{R}^k$ as
\begin{align*}
    \Psi^T(x) = (z(x),T(\mathbf{y}(x)))
\end{align*}
where $(z(x), \mathbf{y}(x)) = \Psi(x)\in P\times C$. Denote $e = \{(y_1,0,0,\dots ,0)\in \mathbf{R}^k, y_1\ge 0\}$. By possibly replace $\Psi$ by $\Psi^T$ for some suitable $T\in O(k)$, we can assume $\Psi(\mathcal{L})\subset P\times e$, and $\Psi(F_i)\quad (i=2,3,\dots,k)$ are contained in $k-1$ facets of $P\times e\times C'\subset P\times \mathbf{R}^k$. Here $e\subset\mathbf{R}$ is parametrized by $y_1$, and $C'\subset \mathbf{R}^{k-1}$ is parametrized by $y_2,y_3,\dots,y_k$.

Let $\epsilon,\delta>0$. By selecting $\delta$ sufficiently small, one ensures that $(\epsilon,2\epsilon)\times (C'\cap B_\delta^{k-1}(\mathbf{0}))\subset C\subset e\times C'$. Hence, $(\epsilon,2\epsilon)\times (C'\cap B_\delta^{k-1}(\mathbf{0}))\subset e\times C'$ is a corner neighborhood of $(\epsilon,2\epsilon)\subset e$ in $C$. Let $V = \Psi^{-1}(P\times (\epsilon,2\epsilon)\times (C'\cap B_\delta^{k-1}(\mathbf{0})))$ and $N^{n-k+1} = \Psi^{-1}(P\times (\epsilon,2\epsilon)\times \mathbf{0})$. We define $\tilde{\Psi}: V\longrightarrow N^{n-k+1}\times \mathbf{R}^{k-1}$ via
\begin{align}\label{eq: 124}
    \tilde{\Psi}\circ\Psi^{-1}(z,y_1,y_2,\dots,y_k) = (\tau(z,y_1),y_2,\dots,y_k)\subset N^{n-k+1}\times \mathbf{R}^{k-1},
\end{align}
where $\tau(z,y_1) = \Psi^{-1}(z,y_1,0,\dots,0)$. Because $\Psi(\mathcal{L})\subset P\times e$, we know that $N\subset\mathcal{L}$.

For $\mathbf{z} = \tau(z,y_1)\in N^{n-k+1}\subset \mathcal{L}$, we have
\begin{align*}
    \tilde{\Psi}(\mathbf{z}) = \tilde{\Psi}(\tau(z,y_1)) = \tilde{\Psi}\circ \Psi^{-1}(z,y_1,0,\dots,0) = (\tau(z,y_1),0,\dots,0) = (\mathbf{z},0,\dots,0).
\end{align*}
It follows that $\tilde{\Psi}$ is a coordinate map modeled on $N^{n-k+1}$. Since $T\in O(k)$, we know that $\tilde{\Psi}^*g$ satisfies the normal compatibility condition if $\Psi^*g$ does.

\begin{figure}
\centering
\includegraphics[width=0.9\textwidth]{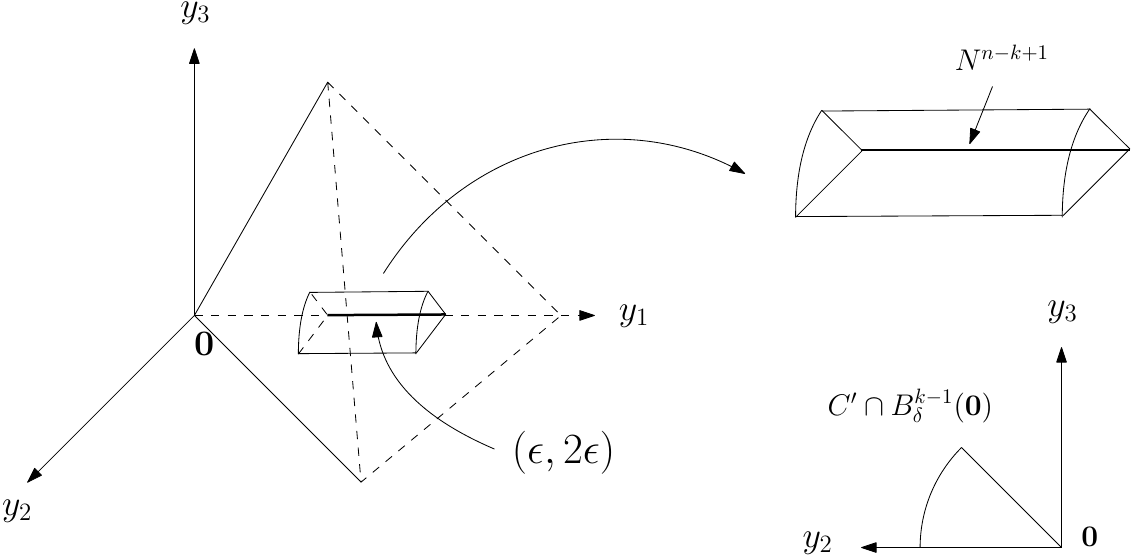}
\caption{The construction of the smooth map in the special case where $n=k=3$, $P^0$ is a point, and $N^1$ is a segment.}
\label{figure: induced coordinate map}
\end{figure}

\begin{definition}\label{defn: induced coordinate map}
    The coordinate map $\tilde{\Psi}: V\longrightarrow N^{n-k+1}\times C'$ constructed from $\Psi: U\longrightarrow P^{n-k}\times C$ above is called the induced coordinate map from $\Psi$.
\end{definition}

\begin{remark}
    In the definition of the induced coordinate map $\tilde{\Psi}$ above, we may use a linear transformation $T\in O(k)$ to adjust the position of $C\subset \mathbf{R}^k$ so that it matches the position of $e\times C'$. In Lemma \ref{lem: linear transition map} below, we will see that this linear adjustment does not affect our definition of the smooth structure on the singular space $X$ in Theorem \ref{thm: coxeter gluing smoothing}.
\end{remark}

Now we are at the final stage of the proof of Theorem \ref{thm: fermi type coordinate for manifolds with corner}. We need to collect the correction maps $\phi$ in Lemma \ref{lem: straightening the normal vector locally} together, which are defined only on corner neighborhoods, and extend them to a global correction map $\Phi$. We provide the details below.

\begin{proof}[Proof of Theorem \ref{thm: fermi type coordinate for manifolds with corner}]

Let $(M^n,g)$ be the Riemannian manifold with corner in the statement of the theorem. We denote $k_{\max}(M) = \max\{ k, \mathcal{S}^k\ne \emptyset\}$.

    When $k_{\max}(M) = 2$, $\mathcal{S}^2$ is the disjoint union of closed manifolds
    \begin{align*}
        \mathcal{S}^{2} = N_1^{n-2}\cup N_2^{n-2}\cup\dots \cup N_{l_1}^{n-2}.
    \end{align*}
    By Lemma \ref{lem: straightening the normal vector locally}, for each $i=1,2,\dots,l_1$, there exist the following objects:
\begin{itemize}
    \item a corner neighborhood $U_i$ of $N_i$ in $M$;
    \item a cone $C_i^2\subset \mathbf{R}^2$;
    \item a smooth coordinate map
    \begin{align*}
        \Psi_i: U_i\longrightarrow
        N_i^{n-2}\times (C_i^2\cap B_1(\mathbf{0}))
        \subset N_i^{n-2}\times \mathbf{R}^2;
    \end{align*}
    \item a $C^{1,1}$ diffeomorphism $\phi_i:U_i\longrightarrow U_i$, which is smooth away from $N_i^{n-2}$,
\end{itemize}
such that $(\phi_i)^*g$ satisfies the normal compatibility condition in the coordinates given by $\Psi_i$. A slight difference from Lemma \ref{lem: straightening the normal vector locally} is that $\phi$ in Lemma \ref{lem: straightening the normal vector locally} was defined on an open set of $N\times C$, while $\phi_i$ is defined on $U_i$ now. This does not matter as we can always replace $\phi$ by $\Psi^{-1}\circ\phi\circ\Psi$, which does not affect the $C^{1,1}$ regularity at the corner. By Lemma \ref{lem: extension}, we can extend $\phi_i$ to a $C^{1,1}$ diffeomorphism $\Phi$ of $M$ smooth away from $\partial_{\cor}M$. $\Phi$ and the coordinate maps $\Psi_i$ then satisfies the desired properties.

    We now consider the case $k_{\max}(M)=3$. In this case, $\mathcal{S}^3$ is disjoint union of closed $n-3$ dimensional manifolds
    \begin{align*}
         \mathcal{S}^{3} = P_1^{n-3}\cup P_2^{n-3}\cup\dots \cup P_{l_1}^{n-2}.
    \end{align*}
    For each $P_i$, there exist the following objects:
\begin{itemize}
    \item a corner neighborhood $U_i$ of $P_i$ in $M$;
    \item a smooth coordinate map
    \begin{align*}
        \psi_i:U_i\longrightarrow
        P_i^{n-3}\times (C_i^3\cap B_1(\mathbf{0}))
        \subset P_i^{n-3}\times \mathbf{R}^3;
    \end{align*}
    \item a $C^{1,1}$ diffeomorphism $\phi_i:U_i\longrightarrow U_i$, which is smooth away from $\partial_{\cor}M$,
\end{itemize}
such that $(\phi_i)^*g$ satisfies the normal compatibility condition in the coordinates given by $\psi_i$.

For simplicity of exposition, we first consider the case where $\mathcal{L}^{n-2}$ is a connected component of $\mathcal S^2$ and $\partial\mathcal{L}^{n-2}$ has exactly two boundary components, denoted by $P_1^{n-3}$ and $P_2^{n-3}$. The argument below extends directly to the general case.

    Let $\rho$ be a Morse function on $\mathcal{L}^{n-2}$, satisfying $\rho = 0$ on $P_1$, $\rho = 1$ on $P_2$, and $0<\rho<1$ in $\mathcal{L}$. Let $N_\epsilon = \rho^{-1}[\epsilon,1-\epsilon]$. From Lemma \ref{lem: Putting the metric into the cone}, there is a neighborhood $U_\epsilon$ of $N_\epsilon$, a cone $C^2\subset\mathbf{R}^2$ and a smooth coordinate  map $\tilde{\psi}: U_\epsilon\longrightarrow N_\epsilon\times (C^2\cap B_1(\mathbf{0})) $, such that $(\tilde{\psi}^{-1})^*g = dy_1^2+dy_2^2+\omega_N$ along $N_\epsilon$. However, $\phi_i^*g$ doesn't necessarily satisfy the normal compatibility condition in $U_i\cap U_\epsilon$. 

    To overcome this, we reconstruct $\tilde\psi$ by adapting the argument used in the proof of Lemma \ref{lem: Putting the metric into the cone}. Throughout the construction, we ensure compatibility with the coordinate maps $\psi_1$ and $\psi_2$. Let $\tilde{\psi}_1: V_{\epsilon,1}\longrightarrow \rho^{-1}[\epsilon,2\epsilon]\times C^2$ be the induced coordinate map in the sense of Definition \ref{defn: induced coordinate map} from $\psi_1$ for a neighborhood $V_{\epsilon,1}$ of $\rho^{-1}[\epsilon,2\epsilon] = P_1\times [\epsilon,2\epsilon]$ in $M$. By Lemma \ref{lem: nice form of the metric on the edge assuming angle matching condition}, we have $\tilde{\psi}_1^*g = dy_1^2+dy_2^2+\omega_N$ along $\rho^{-1}[\epsilon,2\epsilon]$. Similarly, there is a neighborhood $V_{\epsilon,2}$ of $\rho^{-1}[1-2\epsilon,1-\epsilon]$ in $M$, and a induced coordinate map $\tilde{\psi_2}: V_{\epsilon,2}\longrightarrow \rho^{-1}[1-2\epsilon,1-\epsilon]\times C^2$ with $\tilde{\psi}_2^* g = dy_1^2+dy_2^2+\omega_N$ along $\rho^{-1}[1-2\epsilon,1-\epsilon]$. Moreover, $\phi_i^*g$ satisfies the normal compatibility condition with respect to the coordinate map $\tilde{\psi}_i\quad (i=1,2)$. 

    Assume that $\mathcal{L}^{n-2}$ is the common edge of two facets $F_1$ and $F_2$. Let us describe how to extend $\tilde{\psi}_1$ and $\tilde{\psi}_2$ to $\tilde{\psi}: U_\epsilon\longrightarrow N_\epsilon\times \mathbf{R}^2$. Since $y_2 = 0$ in $V_{\epsilon,i}\cap F_1$, we define $y_2 = 0$ on $F_1$. Note $|\nabla y_1| = 1$ in $V_{\epsilon,i}\cap F_1$, we simply extend $y_1$ to a neighborhood of $V_\epsilon$ in $F_1$ satisfying $|\nabla y_1| = 1$. This completes Step 1 of Lemma \ref{lem: Putting the metric into the cone}. For Step 2, we define $y_2$ on a hypersurface perpendicular to $F_1$ along $N_\epsilon$, which extend $y_2|_{V_{\epsilon,1}\cup V_{\epsilon,2}}$ and satisfies $|\nabla y_2| = 1$  and $g(\nabla y_1,\nabla y_2) = 0$ along $N_\epsilon$. Then we use the same foliation argument (In the present case it is a foliation extension argument) as in Lemma \ref{lem: Putting the metric into the cone} to define $y_2$ in the whole $U_\epsilon$. For Step 3, Since $\psi_i$ has already mapped $\partial V_{\epsilon,i}$ to straight edges $\partial C^2$ in $\mathbf{R}^2$, the diffeomorphism $\zeta$ in Lemma \ref{lem: Putting the metric into the cone} can be arranged to be the identity on $\rho^{-1}[\epsilon,2\epsilon]\cup\rho^{-1}[1-2\epsilon,1-\epsilon]$. $\phi_i^*g$ satisfies the normal compatibility condition in $U_i\cap U_\epsilon (i=1,2)$ with respect to the coordinate map $\tilde{\psi}$ constructed in this way.

    Finally, by an argument similar to that in Lemma \ref{lem: straightening the normal vector locally}, we can extend $\phi_i$ to a map $\phi$ on $U_\epsilon$, after possibly shrinking $U_\epsilon$. At this time we have to apply the relative version Lemma \ref{lem: C11 Whitney extension smooth away from corner 2 relative version} to leave $\phi_i$ invariant on $U_i\quad (i=1,2)$. By performing this on all components of $\mathcal{S}^2$, we find a non-degenerate $C^{1,1}$ map $\Phi$, defined in a neighborhood of $\partial_{\cor} M$, fixing $\partial_{\cor}M$, such that for all $x\in \partial_{\cor}M$, there is a smooth coordinate chart such that $\Phi^* g$ satisfies the normal compatibility condition in that chart. Then we use Lemma \ref{lem: extension} to extend $\Phi$ into a global map, which completes the proof when $k_{\max}(M)=3$.

    For general $k_{\max}(M)$, the proof is the same as that of the case $k_{\max}(M)=3$. Denote $k_0 = k_{\max}(M)$, we begin with $\mathcal{S}^{k_0}$, which is decomposed into the union of closed manifolds
    \begin{align*}
         \mathcal{S}^{k_0} = P_1^{n-k_0}\cup P_2^{n-k_0}\cup\dots \cup P_{l_1}^{n-k_0}.
    \end{align*}
    Then we construct the coordinate maps modeled on submanifolds of $\mathcal{S}^{k_0-1}, \mathcal{S}^{k_0-2},\dots ,\mathcal{S}^2$ inductively. This completes the proof for general $k_{\max}(M)$.
    
\end{proof}

\subsection{Distributional scalar curvature lower bound}

$\quad$

In this subsection, we prove a proposition concerning the distributional scalar curvature lower bound on manifolds decomposed into a union of polyhedral domains. This type of result was first proved by Lee-Lefloch  \cite[Theorem 5.1]{LL15} for manifolds with hypersurface singularities. The proposition below extends \cite[Theorem 5.1]{LL15} to manifolds with polyhedral singularities.

\begin{proposition}\label{prop: distributional scalar curvature lower bound}
    Let $X$ be a smooth closed Riemannian manifold. Assume that $X$ admits a decomposition into polyhedral domains
\begin{align*}
X=\bigcup_{i=1}^{l}\Omega_i,
\end{align*}
such that the interiors of the $\Omega_i$ are pairwise disjoint and every nonempty intersection $\partial_{\fac}\Omega_i\cap\partial_{\fac}\Omega_j$ $(i\neq j)$ is a common boundary facet of $\Omega_i$ and $\Omega_j$. Let
\begin{align*}
\mathcal{S}
:=
\bigcup_{i=1}^{l}\partial\Omega_i .
\end{align*}
Denote by $\mathcal{S}_{\mathrm{fac}}$ the union of all open facets of the domains $\Omega_i$, and define
\begin{align*}
\mathcal{S}_{\mathrm{cor}}
:=
\mathcal{S}\setminus\mathcal{S}_{\mathrm{fac}}.
\end{align*}

Let ${g}$ be a Lipschitz metric on $X$ which is smooth away from $\mathcal{S}$, satisfying ${g}|_{\Omega_i}$ is smooth up to $\partial_{\fac}\Omega_i$. Assume that
\begin{enumerate}
    \item For each point $x\in \mathcal{S}_{\cor}$, there exists an open neighborhood $\mathcal{U}$ of $x$ in $X$ and a smooth coordinate map $\Psi: \mathcal{U}\longrightarrow \mathbf{R}^n$, such that each component of $\Omega_i\cap \mathcal{U}$ is mapped to a corresponding polyhedral cone $\mathcal{C}\subset \mathbf{R}^n$ (Note that the polyhedral cones associated all these components tessellate $\mathbf{R}^n$). Moreover, the restriction of $(\Psi^{-1})^*g$ on $\mathcal{C}$ satisfies the normal compatibility condition ;
    \item\label{item: 1} $H_i+H_j\ge 0$ on $\partial_{\fac}\Omega_i\cap \partial_{\fac}\Omega_j$, where $H_i$ denotes the mean curvature of $\partial_{\fac}\Omega_i$ with respect to the outward pointing normal of $\Omega_i\quad (1\le i\ne j\le l)$.
    \item $R_{{g}}\ge R_0$ away from $\mathcal{S}$ for some $R_0\in \mathbf{R}$.
\end{enumerate}
Then the scalar curvature of ${g}$ is no less than $R_0$ in the distributional sense in the sense of \cite{LL15}. 
\end{proposition}
\begin{proof}
    We follow the argument of \cite[Proposition 5.1]{LL15}. Clearly, the distributional scalar curvature is greater than or equal to $R_0$ away from $\mathcal{S}$ by definition. Within $\mathcal{S}_{\fac}$ the distributional scalar curvature is also greater than or equal to $R_0$ as a consequence of \cite[Theorem 5.1]{LL15}. Thus, it suffices to consider what happens near $x\in \mathcal{S}_{\cor}$. 

    We work in the coordinate chart $\Psi: \mathcal{U}\longrightarrow \Psi(\mathcal{U})\subset \mathbf{R}^n$  assumed in the statement of the proposition. Choose the background metric $h = dx_0^2+dx_1^2+dx_2^2+\dots+ dx_{n-1}^2$ and choose $u$ to be a smooth function supported on $\Psi(\mathcal{U})\subset \mathbf{R}^n$. Denote by $X_{\varepsilon}$ the complement of the $\varepsilon$-neighborhood of $\mathcal{S}$ in $X$. Let $\nu$ be the unit normal of $\partial X_{\varepsilon}$ pointing outward $X_{\varepsilon}$.

    Now we work with a fixed connected component of $\Omega_i\cap \mathcal{U}$ for a fixed $i$, and denote the corresponding Euclidean polyhedral cone by $\mathcal{C}$. Assume that $x$ lies in $k$ facets $F_1, F_2,\dots, F_k$ of this component, and let $\hat{F}_j = \Psi(F_j)\quad (j=1,2,\dots,k)$ denote the corresponding $k$ flat facets of $\mathcal{C}$. In the following computation, the indices $p,q,r,s$ are used for coordinate indices and are summed over $0,1,\dots,n-1$. From the calculation in \cite[Proposition 5.1]{LL15}, we have

\begin{equation}\label{eq: 72}
    \begin{split}
        \langle\!\langle R_g,u\rangle\!\rangle
&=
\lim_{\epsilon\to 0}
\int_{X_\epsilon} R_g u\,d\mu_g
-
\int_{\partial X_\epsilon}
(V\cdot \nu)
\left(
u\frac{d\mu_g}{d\mu_h}
\right)
\,d\sigma_h
\\
&\ge
\lim_{\epsilon\to 0}
\int_{X_\epsilon} R_0 u\,d\mu_g
+
\lim_{\epsilon\to 0}\int_{\partial X_\epsilon}
\left(
-\Gamma^r_{pq}g^{pq}\nu_r
+
\Gamma^r_{rp}g^{pq}\nu_q
\right)
\left(
u\frac{d\mu_g}{d\mu_h}
\right)
\,dx_1\cdots dx_{n-1},
    \end{split}
\end{equation}
where $V$ is an intrinsic vector field defined by \cite[equation (2.3)]{LL15}. When $\varepsilon$ is sufficiently small, $\partial X_{\varepsilon}\cap \mathcal{C}$ is the union of $k$ hypersurfaces $\Sigma_{\varepsilon,1},\Sigma_{\varepsilon,2},\dots,\Sigma_{\varepsilon,k}$.
    By composing $\Psi$ with an orthogonal transformation on $\mathbf{R}^n$ if necessary, we can assume $\hat{F}_k\subset \{x_0 = 0\}$ and $\mathcal{C}\subset \{x_0\ge 0\}$. Let $n_0$ be the unit normal vector of $\hat{F}_k$ pointing inward $\mathcal{C}$ with respect to $g$, then the normal compatibility condition implies $n_0 = \partial_0$ on $\hat{F}_k$. Hence we obtain
    \begin{align}\label{eq: 73}
        \nu_q = g(\nu,\partial_q) \to -g(n_0,\partial_q) = -\delta_{0q}
    \end{align}
    as $\varepsilon\to 0$. 

    Now we consider the boundary integral over $\Sigma_{\varepsilon,k}$. Since $g$ is Lipschitz, $\Gamma_{pq}^r$ is uniformly bounded. It follows from \eqref{eq: 73} that
\begin{equation}\label{eq: 111}
\begin{split}
        &\big|\int_{\Sigma_{\varepsilon,k}}
\left(
-\Gamma^r_{pq}g^{pq}\nu_r
+
\Gamma^r_{rp}g^{pq}\nu_q
\right)
\left(
u\frac{d\mu_g}{d\mu_h}
\right)
dx_1\cdots dx_{n-1} \\
&\quad -
\int_{\Sigma_{\varepsilon,k}}
\left(
\Gamma^0_{pq}g^{pq}
-
\Gamma^r_{rp}g^{p0}
\right)
\left(
u\frac{d\mu_g}{d\mu_h}
\right)
dx_1\cdots dx_{n-1}\big|\to 0\qquad (\varepsilon\to 0)
\end{split}
\end{equation}
Observe that
\begin{equation}\label{eq: 112}
    \begin{split}
        &\quad 
\int_{\Sigma_{\varepsilon,k}}
\left(
\Gamma^0_{pq}g^{pq}
-
\Gamma^r_{rp}g^{p0}
\right)
\left(
u\frac{d\mu_g}{d\mu_h}
\right)
dx_1\cdots dx_{n-1} \\
&\quad =
\int_{\Sigma_{\varepsilon,k}}
\left(
\frac{1}{2}
\left(
\bar{\nabla}_p g_{q0}
+
\bar{\nabla}_q g_{p0}
-
\bar{\nabla}_0 g_{pq}
\right)\right.
g^{pq}
\\
&\qquad \qquad \left.-
\frac{1}{2}
g^{rs}
\left(
\bar{\nabla}_r g_{0s}
+
\bar{\nabla}_0 g_{rs}
-
\bar{\nabla}_s g_{r0}
\right)
\right)
\left(
u\frac{d\mu_g}{d\mu_h}
\right)
dx_1\cdots dx_{n-1} \\
&\quad =
\int_{\Sigma_{\varepsilon,k}}
-g^{pq}\bar{\nabla}_0 g_{pq}
\left(
u\frac{d\mu_g}{d\mu_h}
\right)
dx_1\cdots dx_{n-1}.
    \end{split}
\end{equation}
Since $g$ is Lipschitz and is smooth up to open faces of $\mathcal{C}$, we obtain
\begin{equation}\label{eq: 113}
    \begin{split}
        &
\int_{\Sigma_{\varepsilon,k}}
-g^{pq}\bar{\nabla}_0 g_{pq}
\left(
u\frac{d\mu_g}{d\mu_h}
\right)
dx_1\cdots dx_{n-1} \to \int_{\hat{F}_k}
-g^{pq}\bar{\nabla}_0 g_{pq}
\left(
u\frac{d\mu_g}{d\mu_h}
\right)
dx_1\cdots dx_{n-1}\\
=&
\int_{\hat{F}_k}
2 H
\left(
u\frac{d\mu_g}{d\mu_h}
\right)
dx_1\cdots dx_{n-1}\ge 0, \qquad(\varepsilon\to 0)
    \end{split}
\end{equation}
where $H$ denotes the mean curvature scalar of $\hat{F}_k$ with respect to the normal vector pointing outward $\mathcal C_i$ (Note that $n_0 = \partial_0$ is the normal vector pointing inward $\mathcal C_i$, so $H$ here is with respect to $-n_0$). Using the same reasoning to $\Sigma_{\varepsilon,1},\dots,\Sigma_{\varepsilon,k-1}$ as well as all other chambers of $\Omega_i\cap \mathcal{U}$ for all $i$, and plugging \eqref{eq: 111}\eqref{eq: 112}\eqref{eq: 113} into \eqref{eq: 72}, we obtain
\begin{equation*}
    \begin{split}
        &\langle\!\langle R_g,u\rangle\!\rangle\\
=&\int_{M\backslash \mathcal{S}} R_g u d\mu_g+ \sum_{i\ne j}\int_{\partial_{\fac} \Omega_i\cap \partial_{\fac}\Omega_j}2(H_i+H_j)
\left(
u\frac{d\mu_g}{d\mu_h}
\right)
\,dx_1\cdots dx_{n-1}\\
\ge& \int_M R_0 u d\mu_g,
    \end{split}
\end{equation*}
which implies that $R_g\ge R_0$ in the distributional sense.
\end{proof}

\subsection{Determining the smooth structure}

$\quad$

    Let $X, \mathcal{U}, N^{n-k}$ and $G$ be as in Theorem \ref{thm: coxeter gluing smoothing}. Let $U_0$ be a chamber, and let
    \begin{align*}
        \Psi_0: U_0\longrightarrow N^{n-k}\times C\subset N^{n-k}\times \mathbf{R}^k
    \end{align*}
    be a smooth coordinate map, where $C\subset\mathbf{R}^k$ is a Coxeter chamber. We denote
    \begin{align*}
        \Psi_0(x) = (z(x),\mathbf{y}(x))\in N^{n-k}\times \mathbf{R}^k.
    \end{align*}
    Using the $G$ action on the Coxeter tessellation $\mathcal{T}_C$, we are able to extend $\Psi_0$ to a map $\Psi_G: \mathcal{U}\longrightarrow N^{n-k}\times \mathbf{R}^k$ by
    \begin{align}\label{eq: 118}
        \Psi_G(w\cdot x) = (z(x), w\cdot \mathbf{y}(x)), \quad x\in U_0,\quad w\in G.
    \end{align}

Let $T\in O(k)$ be an orthogonal transformation on $\mathbf{R}^k$. We define
\begin{equation}\label{eq: 121}
    \begin{split}
        &\Psi_0^T: U_0\longrightarrow N^{n-k}\times T(C)\subset N^{n-k}\times \mathbf{R}^k,\\
    &\Psi_0^T(x) = (z(x), T(\mathbf{y}(x))).
    \end{split}
\end{equation}
Similarly, $\Psi_0^T$ is naturally extended to a map $\Psi^T: \mathcal{U}\longrightarrow N^{n-k}\times \mathbf{R}^k$ using the $G$ action by
    \begin{align}\label{eq: 119}
        \Psi^T_G(w\cdot x) = (z(x), w^T\cdot T(\mathbf{y}(x))), \quad x\in U_0,\quad w\in G.
    \end{align}
for $x\in U_0$ and $w\in G$. Here $w^T\cdot$ denotes the action of $w$ on the Coxeter tessellation
\begin{align*}
    \mathcal{T}^T_C = \{T(w(C)): w\in G\},
\end{align*}
and is therefore given by
\begin{align}\label{eq: 120}
    w^T\cdot \mathbf{y} = T(w\cdot T^{-1}(\mathbf{y})).
\end{align}

We thus obtain the following lemma, which shows that the linear map $T$ that may appear in Definition \ref{defn: induced coordinate map} of the induced coordinate map does not affect the smooth structure.
\begin{lemma}\label{lem: linear transition map}
    $\Psi^T_G\circ \Psi^{-1}_G: N^{n-k}\times \mathbf{R}^k\longrightarrow N^{n-k}\times \mathbf{R}^k$ is a smooth map.
\end{lemma}
\begin{proof}
    By \eqref{eq: 118}\eqref{eq: 119} and \eqref{eq: 120}, we obtain
    \begin{align*}
        \Psi^T_G\circ \Psi^{-1}_G(z(x), w\cdot \mathbf{y}(x)) = (z(x), w^T\cdot T(\mathbf{y}(x))) = (z(x),T(w\cdot\mathbf{y}(x)))
    \end{align*}
    for $x\in U_0$ and $w\in G$, which implies
    \begin{align*}
        \Psi^T_G\circ \Psi^{-1}_G(z,\mathbf{y}) = (z,T\mathbf{y})
    \end{align*}
    for all $(z,y)\in N^{n-k}\times \mathbf{R}^k$. This proves the smoothness of $\Psi^T_G\circ \Psi^{-1}_G$.
\end{proof}

%The following remark shows that $\Phi_\lambda$ can be chosen such that after this adjustment, the tesselation of $X$ into the pieces $(M_\lambda,g_\lambda')$ still satisfies the assumption of Theorem \ref{thm: coxeter gluing smoothing}.

\begin{construction}[Determining the correction maps $\Phi_\lambda (\lambda\in \Lambda)$]\label{construction: the choice of Phi lambda}
    %Recall that in the preceding paragraphs we work with $(M,g')=(M,\Phi^*g)$, where $\Phi$ is chosen as in Theorem \ref{thm: fermi type coordinate for manifolds with corner}. We now briefly explain how to determine $\Phi_\lambda$ more precisely for each building block $(M_\lambda,g_\lambda)$ in $X$.
    Recall that from Theorem \ref{thm: fermi type coordinate for manifolds with corner} that for any building block $(M_\lambda,g_\lambda)$ in $X$, there exists a $C^{1,1}$ self diffeomorphism $\Phi_\lambda$ of $M$, which is smooth away from $\partial_{\cor}M$, such that the normal compatible smooth coordinate map exists around each corner point of $(M_\lambda,g_\lambda') = (M_\lambda,\Phi^*_\lambda g)$. 
    
Because the tessellation of $X$ by $(M_\lambda,g_\lambda)$ satisfies the Coxeter compatibility condition around each corner point, and because the induced metrics agree isometrically on each open facet, the correction map $\Phi_\lambda$ can be chosen to be invariant under the local Coxeter group action in a neighborhood of each corner point, and to agree across each facet. This can be seen by tracing the construction of $\Phi$ in the proof of Theorem \ref{thm: fermi type coordinate for manifolds with corner}: the construction starts from the corner neighborhood of a building block, then proceeds to the open facets, and finally extends to the interior. Then $(M_\lambda,g'_\lambda)$ still satisfies all the assumptions of Theorem \ref{thm: coxeter gluing smoothing}.

From now on, we work with the adjusted building blocks $(M_\lambda,g'_\lambda) \quad(\lambda\in \Lambda)$ instead of the original ones $(M_\lambda,g_\lambda)\quad (\lambda\in \Lambda)$.
\end{construction}

$\quad$

Next, we study the smoothness of the transition map induced from the induced coordinate map defined by Definition \ref{defn: induced coordinate map}. Let $P^{n-k}\subset\mathcal{S}^k\subset (M,g')$. Assume that
\begin{itemize}
    \item $U$ is a neighborhood of $P^{n-k}$ in $M$, and $\Psi: U\longrightarrow P^{n-k}\times C $ is a smooth coordinate map;
    \item $\tilde{\Psi}: V\longrightarrow N^{n-k+1}\times C'$ is the coordinate map induced from $\Psi$ as defined in Definition \ref{defn: induced coordinate map}.
\end{itemize}
From \eqref{eq: 124}, $\tilde{\Psi}\circ\Psi^{-1}$ maps between open subsets of $P^{n-k}\times C$ and $N^{n-k+1}\times C'$, and is given by
\begin{align*}
    \tilde{\Psi}\circ\Psi^{-1}(z,y_1,y_2,\dots,y_k) = (\tau(z,y_1),y_2,\dots,y_k)\subset N^{n-k+1}\times \mathbf{R}^{k-1},
\end{align*}
where $\tau$ is a smooth map. 

Viewing $P^{n-k}$ and $N^{n-k+1}$ as submanifolds of $X$, we denote by $G_P$ and $G_N$ the Coxeter groups associated with $P^{n-k}$ and $N^{n-k+1}$, respectively. Note that $G_N<G_P$ is the stabilizer subgroup of $N^{n-k+1}$.
\begin{lemma}\label{lem: smoothness of the induced transition map}
    $\tilde{\Psi}_{G_N}\circ\Psi_{G_P}^{-1}$ is smooth.
\end{lemma}
\begin{proof}
    Let $w\in G_N$, by \eqref{eq: 118}, we have
\begin{align*}
    &\tilde{\Psi}_{G_N}\circ\Psi_{G_P}^{-1}(z,y_1,w\cdot(y_2,\dots,y_k))\\
     =&\tilde{\Psi}_{G_N}\circ\Psi_{G_P}^{-1}(z,w\cdot(y_1,y_2,\dots,y_k)) = \tilde{\Psi}_{G_N}(w\cdot \Psi_{G_P}^{-1}(z,y_1,y_2,\dots,y_k)) \\= &w\cdot  \tilde{\Psi}\circ\Psi^{-1}(z,y_1,y_2,\dots,y_k) = w\cdot(\tau(z,y_1),y_2,\dots,y_k) = (\tau(z,y_1), w\cdot(y_2,\dots,y_k)).
\end{align*}
In the first equality we used that $G_N$ fixes the $y_1$ axis in $\mathbf{R}^k$. Therefore, the transition map $ \tilde{\Psi}_{G_N}\circ\Psi_{G_P}^{-1}$ has the form
\begin{align*}
    \tilde{\Psi}_{G_N}\circ\Psi_{G_P}^{-1}(z,y_1,y_2,\dots,y_k) = (\tau(z,y_1),y_2,\dots,y_k),
\end{align*}
which implies its smoothness.
\end{proof}

    We are now ready to construct a smooth structure on $X$. Recall that $k_{\max}(M) = \max\{ k, \mathcal{S}^k\ne \emptyset\}$. We first give a detailed construction of the $k_{\max}(M) = 2$ case of Theorem \ref{thm: coxeter gluing smoothing}. Let $(M,g') = (M,\Phi^*g)$ be a building block, and $\mathcal{S}^2$ is the disjoint union of closed manifolds
    \begin{align*}
        \mathcal{S}^{2} = N_1^{n-2}\cup N_2^{n-2}\cup\dots \cup N_{l_1}^{n-2}.
    \end{align*}
By Theorem \ref{thm: fermi type coordinate for manifolds with corner}, there is a neighborhood $U_i$ of $N_i$ in $M$ and a smooth coordinate map
\begin{align*}
    \psi_i: U_i\longrightarrow N_i^{n-2}\times C^2_i\subset N_i^{n-2}\times \mathbf{R}^2,\quad 1\le i\le l_1
\end{align*}
such that $(\psi^{-1}_i)^* g'$ satisfies the normal compatibility condition. Here, $C_i^2$ is a Coxeter chamber in $\mathbf{R}^2$. 

Let $F_0$ be a facet of $M$ and denote $F_0^-$ to be the part of $F_0$ obtained from itself with a tubular neighborhood of its boundary removed. We need to construct a coordinate map $\tilde{\psi}_0: W_0^-\longrightarrow F_0^-\times [0,\delta)$ which is compatible with $\psi_i\quad (i=1,2,\dots,l_1)$ above after we perform the Coxeter reflection. Here $W_0^-$ is a neighborhood of $F_0^-$ in $M$. Without of loss of generality, assume that $\partial F_0 = N_1^{n-2}\cup N_2^{n-2}\cup\dots \cup N_{l_2}^{n-2}\quad (l_2\le l_1)$. A tubular neighborhood of a component of $\partial F_0^-$ is diffeomorphic to $N_i^{n-2}\times (\epsilon,2\epsilon)$. Each coordinate map $\psi_i$ gives rise to a induced coordinate map $\tilde{\psi}_i: V_i\longrightarrow N_i^{n-2}\times (\epsilon,2\epsilon)\times [0,\delta)$, satisfying the normal compatibility condition. Here $V_i$ is a neighborhood of $N_i^{n-2}\times (\epsilon,2\epsilon)$ in $M$. Then we extend $\tilde{\psi}_i$ to obtain a smooth coordinate map
\begin{align*}
    \tilde{\psi}_0: W_0^-\longrightarrow F_0^-\times [0,\delta)
\end{align*}
that satisfies the normal compatibility condition.

%We denote $\Int M$ to be the part of $M$ with a tubular neighborhood of its boundary removed. 

By repeating the construction above for each facet of the building block $(M,g')$, we obtain the following collection of maps
\begin{itemize}
    \item A neighborhood $U_i$ of $N_i^{n-2}$ in $M$ and a smooth map
    \begin{align}\label{eq: 122}
        \psi_i: U_i\longrightarrow N_i^{n-2}\times C^2_i\subset N_i^{n-2}\times \mathbf{R}^2
    \end{align}
    satisfying the normal compatibility condition $(i=1,2,\dots,l_1)$;
    \item A neighborhood $W_b^-$ of $F_b^-$ in $M$ for each facet $F_b$ and a smooth map
    \begin{align}\label{eq: 123}
        \tilde{\psi}_b: W_b^-\longrightarrow F_b^-\times [0,\delta)
    \end{align}
    satisfying the normal compatibility condition $(b\in \mathscr{B})$;
    \item The identity map $\Id: \Int M\longrightarrow \Int M$.
\end{itemize}

%With the preparation above, we now begin to assign a smooth structure to the topological manifold $X$ in Theorem \ref{thm: coxeter gluing smoothing}.We first introduce some notations. Let
Let $X$ be the topological manifold in Theorem \ref{thm: coxeter gluing smoothing}. We introduce some notations. Let
\begin{itemize}
   \item $N_1^{n-2}\cup N_2^{n-2}\cup\dots \cup N_{L_1}^{n-2}$ denote the collection of all codimension $2$ singularities in the topological manifold $X$, $i.e.$
    \begin{align*}
        \mathcal{S}^2 = N_1^{n-2}\cup N_2^{n-2}\cup\dots \cup N_{L_1}^{n-2}.
    \end{align*}
    Each $N_i^{n-2}$ is a closed manifold;
    \item $F_{b}\quad\{b\in \mathscr{B}^+\}$ denote the collection of all open facets in the topological manifold $X$;
    \item $\Int M_\lambda\quad (\lambda\in \Lambda)$ denote the collection of all interior of the building blocks in the topological manifold $X$.
\end{itemize}

We first consider the simpler case that the chamber decomposition near $N_i^{n-2}$ is isomorphic to $N_i^{n-2}\times \mathcal{T}_{C_i^2}$. By the Coxeter compatibility condition, we can use \eqref{eq: 118} to find
\begin{itemize}
    \item a neighborhood $\mathcal{U}_i$ of $N_i^{n-2}$ in $X$ and a map
    \begin{align}\label{eq: 126}
        (\psi_i)_{G_i}: \mathcal{U}_i\longrightarrow N_i^{n-2}\times \mathbf{R}^2
    \end{align}
    that extends $\psi_i$ in \eqref{eq: 122} in the manner of \eqref{eq: 118}, where $G_i$ is the Coxeter group associated to the Coxeter chamber $C_i^2\subset \mathbf{R}^2$;
    \item a neighborhood of $\mathcal{W}_{b}$ of $F_b^-$ in $X$ and a map
    \begin{align}\label{eq: 127}
        (\tilde{\psi}_b)_{\mathbf{Z}_2}: \mathcal{W}_b\longrightarrow F_b^-\times (-\delta,\delta)
    \end{align}
    that extends $\tilde{\psi}_b$ in \eqref{eq: 123} by reflection;
    \item the identity map $\Id: \Int M_\lambda\longrightarrow \Int M_\lambda\quad (\lambda \in \Lambda)$.
\end{itemize}

\begin{figure}
\centering
\includegraphics[width=0.9\textwidth]{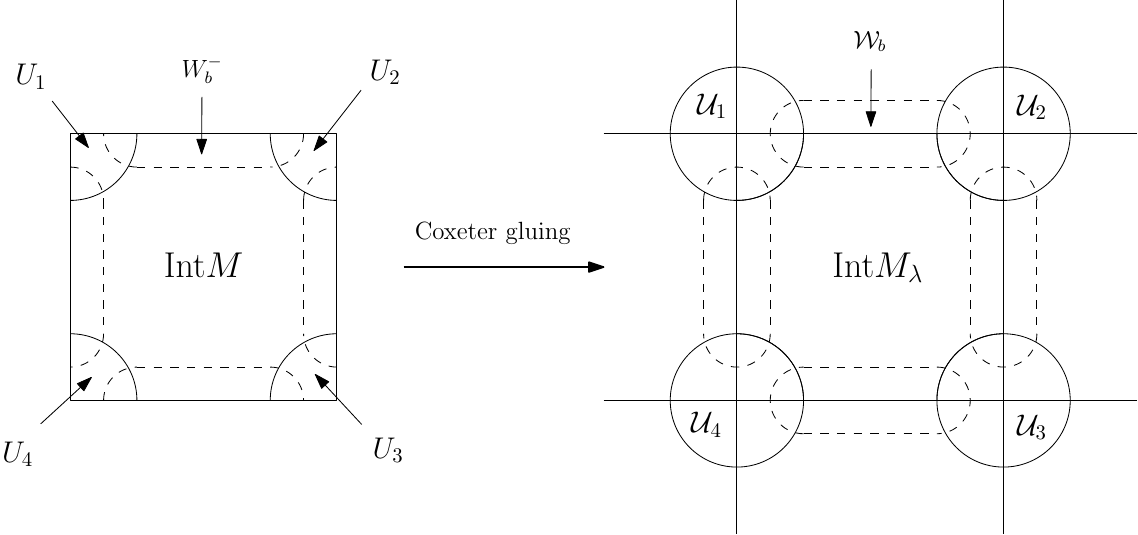}
\caption{The smooth structure on the topological manifold $X$}
\label{figure: smooth structure}
\end{figure}

We verify the smooth compatibility of these coordinate maps. The smooth compatibility between $\mathcal{U}_i$ ($\mathcal{W}_b$) and $\Int M_\lambda$ follows from the smoothness of $\psi_i$ ($\tilde{\psi}_b$). By shrinking $\mathcal{U}_i$ ($\mathcal{W}_b$) if necessary, we can always ensure that $\mathcal{U}_i\cap \mathcal{U}_{i'} = \emptyset$ ($\mathcal{W}_b\cap \mathcal{W}_{b'}=\emptyset$) whenever $i\ne i'$ ($b\ne b'$). Thus, it suffices to verify the smooth compatibility between $(\psi_i)_{G_i}$ and $(\tilde{\psi}_b)_{\mathbf{Z}_2}$. Up to a linear correction, the transition map $(\tilde{\psi}_b)_{\mathbf{Z}_2}\circ (\psi_i)_{G_i}$ is completely described by Lemma \ref{lem: smoothness of the induced transition map}, in the special case that $G_N = \mathbf{Z}_2$ and $G_P = G_i$. This shows that $(\psi_i)_{G_i}$ and $(\tilde{\psi}_b)_{\mathbf{Z}_2}$ is smooth. By Lemma \ref{lem: linear transition map}, the linear correction also gives rise to a smooth transition map on $\mathcal{U}_i$. Hence the coordinate system listed above defines a smooth structure on $X$.

In the general case, the chamber decomposition near $N_i^{n-2}$ is not necessarily isomorphic to $N_i^{n-2}\times \mathcal{T}_{C_i}$. We then cover $N_i^{n-2}$ by union of balls $B_{i,1}, B_{i,2},\dots,B_{i,s}$, such that the chamber decomposition near each $B_{i,j}$ is isometric to $B_{i,j}\times \mathcal{T}_{C_i^2}$. For each $j$ we have a neighborhood $\mathcal{U}_{i,j}$ of $B_{i,j}$ in $X$ and a coordinate map
\begin{align*}
    (\psi_{i,j})_{G_i}: \mathcal{U}_{i,j}\longrightarrow B_{i,j}\times \mathbf{R}^2.
\end{align*}
For $j_1\ne j_2$, the transition map has the form
\begin{align*}
    (\psi_{i,j_1})_{G_i}\circ (\psi_{i,j_2})_{G_i}^{-1}: &(B_{i,j_1}\cap B_{i,j_2})\times \mathbf{R}^2&\longrightarrow &(B_{i,j_1}\cap B_{i,j_2})\times \mathbf{R}^2\\
    &(z,\mathbf{y})&\mapsto &(z,T\mathbf{y}),
\end{align*}
where $T\in O(2)$ preserves the tessellation $\mathcal{T}_{C_i^2}$, although it may interchange chambers. It's clear that the transition map is smooth. This completes the construction of the smooth structure in the case $k_{\max}(M) = 2$.

For general $k_{\max}(M)$, we begin with the lowest-dimensional singular strata and, for each building block, construct the induced coordinate maps inductively as in the proof of Theorem \ref{thm: fermi type coordinate for manifolds with corner}. We then use \eqref{eq: 118} to extend these maps to coordinate maps on open subsets of $X$. The rest of the proof is the same as in the case $k_{\max}(M)=2$, except that we use Lemma \ref{lem: smoothness of the induced transition map} to verify the smoothness of the transition maps between coordinate maps modeled on higher-codimensional singular strata.

\subsection{Proof of the Coxeter smoothing theorem}

$\quad$

We are now in a position to prove Theorem \ref{thm: coxeter gluing smoothing}. We first prove the following lemma, which enables us to glue metrics on building blocks into a Lipschitz metric.

\begin{lemma}\label{lem: gluing into a Lipschitz metric}
    Let $C\subset\mathbf{R}^k$ be a Coxeter chamber corresponding to a Coxeter group $G$, and let $N^{n-k}$ be a Riemannian manifold. Let $g$ be a metric on $N\times C$, which is smooth on $N\times( C\backslash\partial_{\mathrm{cor}}C)$ and Lipschitz up to $N\times\partial_{\mathrm{cor}}C$. The Coxeter action on $N\times \mathcal{T}_C$ induces a metric $g_w$ on $N\times w(C)$ for each $w\in G$. Assume that $(N\times C,g)$ satisfies the normal compatibility condition. Then $\{g_w\}_{w\in G}$ patch together into a Lipschitz metric $\hat{g}$ on $N\times \mathbf{R}^k$.
\end{lemma}

\begin{proof}
    For $w\in G$, let $T_w\in O(k)$ denote the action of $w$ on $\mathbf{R}^k$, and let
    \begin{align*}
        \hat{T}_w = \Id\times T_w: N\times \mathbf{R}^k\longrightarrow N\times \mathbf{R}^k
    \end{align*}
    be the induced action of $w$ on $N\times \mathbf{R}^k$. The normal compatibility condition satisfied by $(N\times C,g)$ implies that
    \begin{align*}
        g = dy_1^2+\dots+dy_k^2+\omega_N
    \end{align*}
    along $N=N\times\{\mathbf{0}\}$. Since $T_w\in O(k)$, we have
    \begin{align*}
        g_w = \hat{T}_w^*g = dy_1^2+\dots+dy_k^2+\omega_N
    \end{align*}
    along $N=N\times\{\mathbf{0}\}$. Hence, we simply define $\hat{g}$ to be $dy_1^2+\dots+dy_k^2+\omega_N$ along $N=N\times\{\mathbf{0}\}$. This shows that $\hat{g}$ is well defined along $N=N\times\{\mathbf{0}\}$.

    Similarly, using Lemma \ref{lem: nice form of the metric on the edge assuming angle matching condition}, we can prove that $\hat{g}$ is well defined on other parts of $\mathcal{S}_{\mathrm{cor}}$. For a facet $F_0$ that is the common part of two chambers $N\times w_1(C)$ and $N\times w_2(C)$, the unit normal vectors with respect to $g_{w_1}$ and $g_{w_2}$ coincide along $F_0$, as a consequence of the normal compatibility condition of $(N\times C,g)$. This implies that $\hat{g}$ is also well defined on $\mathcal{S}_{\mathrm{fac}}$. To summarize, $\{g_w\}_{w\in G}$ patch together into a well defined continuous metric $\hat{g}$ on $N\times \mathbf{R}^k$. Because $g_w$ is Lipschitz in the chamber $N\times w(C)$, we know that $\hat{g}$ is Lipschitz in $N\times \mathbf{R}^k$.
\end{proof}
\begin{proof}[Proof of Theorem \ref{thm: coxeter gluing smoothing}]
Let the topological manifold $X$ be endowed with the smooth structure constructed in Section 4.3. We first show that the metrics $(M_\lambda,g'_\lambda)$, $\lambda\in\Lambda$, patch together to define a Lipschitz metric $\hat{g}$ on $X$ with respect to this smooth structure. For each $x\in \mathcal{S}_{\cor}$, let $N^{n-k}$, $\mathcal{U}$, $C$, and $G$ be as in the statement of Theorem \ref{thm: coxeter gluing smoothing}. The coordinate map
\begin{align*}
\Psi_G:\mathcal{U}\longrightarrow N^{n-k}\times\mathbf{R}^k
\end{align*}
defined as in \eqref{eq: 126} identifies each chamber of $\mathcal{U}$ with a chamber of the form $N^{n-k}\times \mathcal{T}_C$. By the Coxeter compatibility condition for $(M_\lambda,g'_\lambda)$, see Construction \ref{construction: the choice of Phi lambda}, the metrics expressed in these coordinates satisfy precisely the assumptions imposed on the family $\{g_w\}_{w\in G}$ in Lemma \ref{lem: gluing into a Lipschitz metric}. Therefore, Lemma \ref{lem: gluing into a Lipschitz metric} allows us to define $\hat{g}$ in the coordinate chart $\Psi_G:\mathcal{U}\to N^{n-k}\times\mathbf{R}^k$. We thus obtain a Lipschitz metric $\hat{g}$ on $X$.

Consequently, $(X,\hat{g})$ satisfies all the hypotheses of Proposition \ref{prop: distributional scalar curvature lower bound}. It follows that the scalar curvature of $(X,\hat{g})$ is bounded below by $R_0$ in the distributional sense. This proves assertions (1) and (2) of Theorem \ref{thm: coxeter gluing smoothing}. Assertion (3) follows directly from the Ricci flow mollification result of Lee--Liu \cite[Theorem 1.1]{LL26}. This completes the mollification process and hence the proof of Theorem \ref{thm: coxeter gluing smoothing}.

\end{proof}

\section*{Part III. Examples and rigidity results}
\section{Construction of PSC $3$-manifolds using building blocks}\label{section: construction}

In this section, we construct some examples of 3-dimensional PSC capillary prisms with quantitatively large relative $\pi_2$-systoles. In Section 5.1, for each angle $\alpha\in (0,\frac\pi2)$, we construct the standard cylindrical building blocks $(M_\alpha,g)$; See Example \ref{example: standard cylinder block}. They seem to be the canonical geometric object among all PSC capillary prisms. 

The standard cylindrical block $(M_\alpha,g)$ naturally serves as a rigid model for the relative $\pi_2$-systole estimate under certain curvature pinching conditions. We prove this rigidity result in Section 7; see Theorem \ref{thm: 2 systole estimate, rigidity}. However, we also find that $(M_\alpha,g)$ is generally not rigid if one imposes only the scalar curvature condition. In Section 5.2, we construct such counterexamples and describe the phenomenon that arises. This shows that determining the optimal constant in Theorem \ref{thm: 2 systole estimate} is a rather delicate problem. 

\subsection{Construction of standard cylindrical blocks}

$\quad$

Fix $R_0 >0$ and $\alpha\in (0,\frac{\pi}{2})$. 

We start with a slightly more general set up. Consider a warp product metric  on $S$
\begin{equation}\label{eq: 35}
    \begin{split}
        &g_0 = dr^2+u^2(r)d\theta^2, \theta\in (0,2\pi), r\in (0,b];\\
    &u(0) = 0, u'(0) = 1.
    \end{split}
\end{equation}
Eventually we will take $S= S^2_+$ and $u(r)=\sin r$. Consider the Riemannian product with $\mathbf{R}$:
$$g = ds^2+g_0 = ds^2+dr^2+u^2(r)d\theta^2.$$
This can also be written as
\begin{align}\label{eq: 26}
    g = g_E +u(r)^2d\theta^2
\end{align}
where $E = \mathbf{R}\times (0,b)$ is the $(s,r)$-plane equipped with the flat metric $g_E = ds^2+dr^2$.

Let $ [a,b]\subset(0,b]$ and let $\rho: [t_-,t_+] \longrightarrow [a,b]$ be a positive increasing function to be determined later, where $[t_-,t_+]\subset \mathbf{R}$. Define
\begin{align}\label{eq: 41}
    &M_{\alpha}=\left\{(s,r,\theta)\in \mathbf{R}\times S \;\middle|\; t_-\le s\le t_+,\ 0\le r\le \rho(s)\right\},\\
    &E_{\alpha} = \left\{(s,r)\in E \;\middle|\; t_-\le s\le t_+,\ 0\le r\le \rho(s)\right\}.
\end{align}
Then
\begin{align*}
    \partial_{\pm}M_{\alpha} &= (\{t_{\pm}\}\times S)\cap \partial M_{\alpha},\\
    \partial_0M_{\alpha} &= ((t_-,t_+)\times \partial S)\cap \partial M_{\alpha}.
\end{align*}

\begin{figure}
\centering
\includegraphics[width=0.6\textwidth]{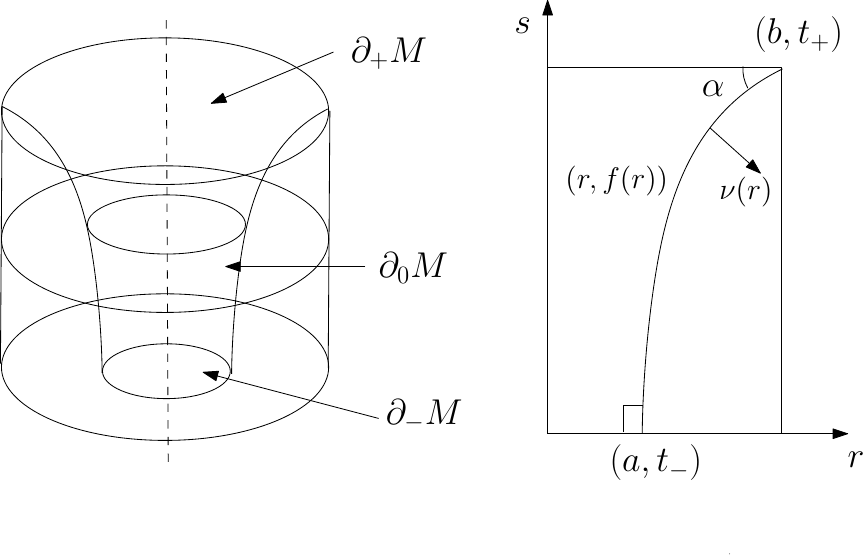}
\caption{Standard cylindrical building blocks. The left-hand figure is obtained by rotating the right-hand figure about the $s$-axis.}
\label{figure: example 1}
\end{figure}

Let $f:[a,b]\longrightarrow [t_-,t_+]$ be the inverse function of $\rho$. In $E$, the set
\begin{align*}
    (r,f(r)),\quad a\le r\le b,
\end{align*}
defines a smooth curve, which we denote by $\gamma(r)$. We impose the following angle condition:
\begin{align}\label{eq: 28}
    f'_+(a)=+\infty,\qquad f'(b)=\tan\alpha.
\end{align}
The downward-pointing unit normal of $\gamma(r)$ in $E$ is given by
\begin{align*}
    \nu(r)=\frac{1}{\sqrt{1+(f'(r))^2}}(f'(r),-1).
\end{align*}
See Figure~\ref{figure: example 1} for an illustration. The geodesic curvature of $\gamma(r)$ with respect to $g_E$ is given by
\begin{align*}
    \kappa(r) = \frac{f''(r)}{(1+f'(r)^2)^{\frac32}}.
\end{align*}

From the warped product expression \eqref{eq: 26} for $g$, the mean curvature $\bar{H}$ of $\partial_0 M_{\alpha}$ with respect to the outward unit normal is computed as follows:
\begin{align*}
    \bar{H} &= \kappa+u^{-1}\partial_\nu u = \kappa+u^{-1}\langle \nabla_E u,\nu\rangle_{g_E}\\
    & = \frac{f''(r)}{(1+f'(r)^2)^{\frac32}}+\frac{1}{u(r)}\frac{1}{\sqrt{1+f'(r)^2}}\langle (u'(r),0),(f'(r),-1)\rangle\\
    &=  \frac{f''(r)}{(1+f'(r)^2)^{\frac32}}+\frac{u'(r)}{u(r)}\frac{f'(r)}{\sqrt{1+f'(r)^2}}.
\end{align*}
Thus, the condition $\bar{H}=0$ is equivalent to
\begin{align*}
    f''(r)+\frac{u'(r)}{u(r)}f'(r)(1+f'(r)^2)=0.
\end{align*}
Multiplying both sides of the above equality by $f'(r)$ and integrating, we obtain
\begin{align}\label{eq: 27}
    \frac{f'(b)^2(f'(r)^2+1)}{f'(r)^2(f'(b)^2+1)}=\frac{u^2(r)}{u^2(b)}.
\end{align}
Let $\theta(r)=\arctan f'(r)$. By \eqref{eq: 28}, we have $\theta(a)=\frac{\pi}{2}$ and $\theta(b)=\alpha$. Combining this with \eqref{eq: 27}, we obtain
\begin{align}\label{eq: 29}
    &\sin\theta(r)\cdot u(r)=\sin\theta(b)\cdot u(b)=\sin\alpha\cdot u(b).
\end{align}

$\quad$

From now on, we specify $S=S^2_+$ and
\begin{align}\label{eq: 39}
    u(r) = \sin r, \quad b = \frac{\pi}{2}.
\end{align}
The scalar curvature of $(M_\alpha,g)$ is computed as
\begin{align}\label{eq: 128}
    R_g = -\frac{2u''(r)}{u(r)} = 2.
\end{align}
From \eqref{eq: 29} we know that
\begin{align*}
    \sin a = \sin a\cdot\sin\theta(a) = \sin b\cdot\sin\theta(b) = \sin\theta(b) = \sin\alpha,
\end{align*}
So $a=\alpha$. Thus $\partial_-M$ is a geodesic ball of radius $a$ on $S_+^2$, hence
\begin{align}\label{eq: 129}
    |\partial_-M| = 2\pi\int_0^a u(r')dr' = 2\pi(1-\cos a) = 2\pi(1-\cos\alpha)
\end{align}

Combining \eqref{eq: 128} and \eqref{eq: 129}, and rescaling the metric if necessary, we summarize the above construction as follows.

\begin{example}\label{example: standard cylinder block}
    For any $R_0>0$ and $\alpha\in (0,\frac{\pi}{2})$, the capillary prism $(M_\alpha,g)$ constructed above has the following properties:
    \begin{enumerate}
        \item $(M_\alpha,g)$ has constant scalar curvature $R_0$;
        \item $(M_\alpha,g)$ is rotationally symmetric with respect to an axis;
        \item $\angle_{0+}(M_\alpha,g) = \alpha$, $\angle_{0-}(M_\alpha,g) = \frac\pi2$; $\partial_{\pm}M_\alpha$ are totally geodesic, and $\partial_0M_\alpha$ is minimal;
        \item $M_\alpha$ is foliated by a one-parameter family of totally geodesic capillary surfaces, with the capillary angle varying monotonically in $(\alpha,\frac{\pi}{2})$.
        \item \begin{align}\label{eq: 30}
    (\inf_{M_\alpha} R_g)\cdot|\partial_-M_\alpha| = R_0\cdot|\partial_-M_\alpha| =  4\pi(1-\cos\alpha).
\end{align}
    \end{enumerate}
\end{example}
We call $(M_\alpha,g)$ constructed above the \textit{standard cylindrical block} with upper angle $\alpha$.

\subsection{Non-rigidity: deformations that preserve scalar lower bound while increasing the systole}

$\quad$

In this subsection, we modify Example \ref{example: standard cylinder block} to construct a capillary prism with upper angle $\alpha$ and scalar curvature bounded below by $R_0$, whose relative $\pi_2$-systole is slightly larger than that of the original one. We continue to use the notation from Section 4.1. Let $\alpha\in (0,\frac{\pi}{2})$ be a fixed angle, and let $S_+^2$ be the upper hemisphere defined by \eqref{eq: 35}. Consider the product space $\mathbf{R}\times S_+^2$ equipped with the doubly warped product metric
\begin{align}\label{eq: 36}
    g = v^2(r)ds^2+g_{S^2} = v^2(r)ds^2+dr^2+u^2(r)d\theta^2,
\end{align}
where $v(r)\in C^2(0,b]$ is a nonnegative function satisfying $v'(0)=0$. This time we denote the $(s,r)$-plane by $\hat{E}$, so that the metric $g$ takes the form
$$g = g_{\hat{E}} +u(r)^2d\theta^2,$$
where
$$g_{\hat{E}} = v^2(r)ds^2+dr^2.$$

Let $\rho: [t_-,t_+] \longrightarrow [a,b]$ and $f: [a,b]\longrightarrow [t_-,t_+]$ be inverse functions defining the curve $\gamma(r) = (r,f(r))\quad(a\le r\le b)$ in $\hat{E}$ as in Section 5.1, which generate the side boundary $\partial_0 M$ by rotation. Recall
\begin{align*}
    \partial_-M = \{(t_-,r,\theta) \,\big| 0\le r\le a\},\quad \partial_+ M = \{(t_+,r,\theta) \,\big| 0\le r\le b\}.
\end{align*}
As before, we impose on $f$ the boundary condition \eqref{eq: 28}. Note from the expression of $g$ \eqref{eq: 36} that each $s$-level set is totally geodesic. Therefore, both $\partial_-M$ and $\partial_+ M$ are totally geodesic.

Next we find the condition that ensures the minimality of the lateral boundary $\partial_0M$. The area of $\partial_0M$ can be written as follows, which is in a form of the area functional of the surface of revolution determined by $f$:
\begin{align*}
    \mathcal{A}(f) = \int_a^b u(r)\sqrt{1+v^2(r)f'(r)^2}\,dr.
\end{align*}
Its Euler--Lagrange equation is
\begin{align*}
    \frac{uv^2f'}{\sqrt{1+v^2(f')^2}} = \mathrm{const}.
\end{align*}
In conjunction with the boundary condition \eqref{eq: 28}, we obtain the relation
\begin{align}\label{eq: 37}
    u(b)v(b)\sin\alpha = u(a)v(a).
\end{align}

By a direct calculation, the scalar curvature of $(M^3,g)$ is given by
\begin{align}\label{eq: 38}
    R_g = -2\left(\frac{u''}{u}+\frac{v''}{v}+\frac{u'v'}{uv}\right).
\end{align}

With some elementary calculations carried out in Appendix A, we have the following example:

\begin{example}\label{example: perturbation counterexample}
    Let $\alpha\in (0,\alpha_0)$ be an angle, where $\alpha_0\approx 0.35\pi$ is defined at the beginning of Appendix A. Let $(M^3,g)$ be a capillary prism constructed as above, with $u(r)$ and $v(r)$ selected as in Lemma \ref{lem: A2}. Then
    \begin{align*}
        (\inf_M R_g)\cdot \sys_2(M,\partial_0 M,g)>4\pi(1-\cos\alpha).
    \end{align*}
    
    To see this, first note that \eqref{eq: a6} coupled with \eqref{eq: 38} implies $R_g \ge 2$. Since $M$ is foliated by totally geodesic capillary surfaces, the same calibration argument as in the proof of Proposition \ref{prop: energy-essential} implies that $|\partial_-M| = \sys_2(M,\partial_0 M,g)$. From Lemma \ref{lem: A2}, by taking $\mu$ sufficiently small, it holds
    \begin{align*}
        |\partial_- M| = 2\pi \int_0^a u(r)dr >2\pi (1-\cos\alpha).
    \end{align*}
    Summing these up we obtain the desired result.
\end{example}

\subsection{Connected sums of $S^2\times S^1$, and lens spaces}\label{subsection: construction of S2xS1 and lens spaces}

$\quad$

In this subsection, we construct PSC metrics on a large class of $3$-manifolds by tessellating them into building blocks. 

\subsubsection{The $\mathcal{P}_3$}

$\quad$

One fundamental component in our construction of PSC $3$-manifolds with large $\pi_2$-systole is $\mathcal{P}_3$, which is diffeomorphic to a three-dimensional sphere with three disjoint $3$-balls removed. We first construct PSC metrics on $\mathcal{P}_3$ with minimal boundary. More precisely, we describe a geometric realization of $\mathcal{P}_3$ with scalar curvature bounded below by $R_0$ and area-minimizing boundary achieving its $\pi_2$-systole.

To construct such a metric, we use six copies of the capillary prism $(M,g)$ with upper angle $\frac{\pi}{3}$ and lower angle $\frac{\pi}{2}$ constructed in Example \ref{example: perturbation counterexample} as fundamental building blocks (Note that $\frac\pi3<\alpha_0\approx 0.35\pi$ in Example \ref{example: perturbation counterexample}). The block $(M,g)$ satisfies
\begin{align}\label{eq: 102}
R_g\ge R_0>0,\qquad \sys_2(M,\partial_0 M,g) = |\partial_- M| >\frac{2\pi}{R_0}.
\end{align}
Furthermore, each facet of this block is weakly mean convex. For simplicity, we denote the six copies by $M_i$, $1\le i\le 6$. We now specify how to identify their boundary facets. Let
\begin{align}\label{eq: 106}
\partial_+M_{2k}\sim \partial_+ M_{2k+1},\qquad \partial_0 M_{2k-1}\sim \partial_0M_{2k}.
\end{align}
Here the subscripts are understood modulo $6$. All of these identifications are given by reflections, and it is straightforward to see that the resulting quotient space is homeomorphic to $\mathcal{P}_3$.

\begin{figure}
\centering
\includegraphics[width=0.9\textwidth]{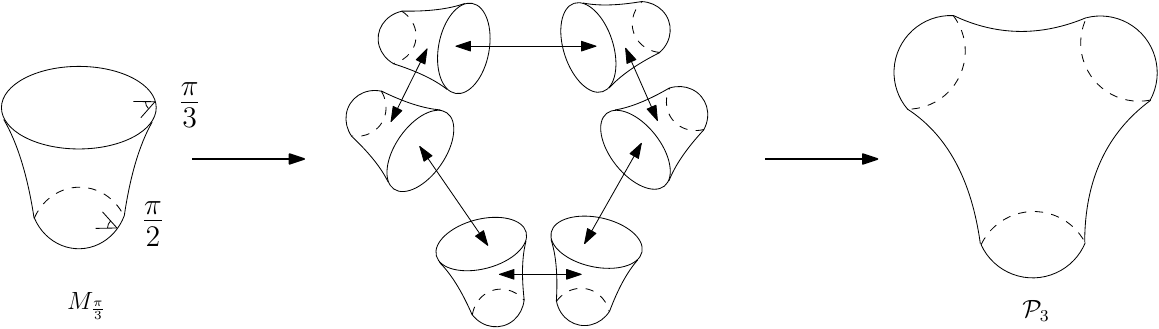}
\caption{The construction of $\mathcal{P}_3$ using the building block $M_{\frac\pi3}$.}
\label{figure: P3}
\end{figure}

In the identification above, the facets $\partial_- M_i$ are left unglued. However, since $\partial_0 M_{2k-1}$ is identified with $\partial_0M_{2k}$ and the dihedral angle between $\partial_0 M_i$ and $\partial_- M_i$ is $\frac{\pi}{2}$, the two facets $\partial_- M_{2k-1}$ and $\partial_- M_{2k}$ are glued together to form a closed sphere. By \eqref{eq: 102}, the area of each such sphere is greater than $\frac{4\pi}{R_0}$.

We now analyze the singular structure of $\mathcal{P}_3$ obtained from the identification above. The singular set $\mathcal{S}$ decomposes as $\mathcal{S}_{\fac}\cup\mathcal{S}_{\cor} = \mathcal{S}_{\fac}\cup\mathcal{S}^2$. Since $(M,g)$ has weakly mean convex facets, the mean curvature jump across $\mathcal{S}_{\fac}$ is nonnegative. It remains to understand the codimension-two singular strata. Note that the curves $\partial(\partial_+ M_i)$, $1\le i\le 6$, are identified into a single $S^1$. This $S^1$ is exactly the codimension-two singular strata $\mathcal{S}^2$. By the identification of the facets among $M_i$ $(1\le i\le 6)$ and the fact that $(M,g)$ has capillary angle $\pi/3$, there exists a neighborhood $\mathcal{U}$ of $\mathcal{S}^2\cong S^1$ in $\mathcal{P}_3$ such that the chamber decomposition of $\mathcal{U}$ is modeled on $S^1\times \mathcal{T}_{C(3)}$, where $C(3)$ is the Coxeter chamber introduced in Example \ref{example: Coxeter chambers on the plane}. This shows that the tessellation of $\mathcal{P}_3$ by the pieces $M_i$ $(1\le i\le 6)$ is Coxeter compatible.

Finally, we perform a doubling trick in order to apply Theorem \ref{thm: coxeter gluing smoothing}. The tessellation of $\mathcal{P}_3$ above gives rise to a tessellation of
$D\mathcal{P}_3\cong S^2\times S^1\# S^2\times S^1$
into $12$ copies of $(M,g)$. The additional codimension $2$ singularities arise along the boundary and are modeled on $S^1\times\mathcal{T}_{C(2)}$, due to the fact that $\partial_0 M$ is perpendicular to $\partial_-M$. Hence the induced tessellation on
$D\mathcal{P}_3\cong S^2\times S^1\# S^2\times S^1$
still satisfies the Coxeter compatibility condition.

Applying Theorem \ref{thm: coxeter gluing smoothing}, and then restricting the resulting metric to one copy of $\mathcal{P}_3$, we obtain a smooth metric $\tilde{g}$ on $\mathcal{P}_3$ with area-minimizing boundary spheres, satisfying
\begin{align*}
    R_{\tilde{g}}\ge R_0>0,
    \qquad
    \sys_2(\mathcal{P}_3,\tilde{g})>\frac{4\pi}{R_0}.
\end{align*}
This finishes the construction.

\subsubsection{The lens space $L(p,q)$ with a ball $D^3$ removed}

$\quad$

Let $L(p,q)$ be the lens space. If $p = 2$, then $L(2,1)\backslash D^3 = RP^3\backslash D^3$ is double covered by $S^2\times [-1,1]$. The product metric on $S^2\times [-1,1]$ then descends to a metric $\tilde{g}$ on $L(2,1)\backslash D^3$ with area-minimizing boundary sphere, satisfying
\begin{align*}
\inf R_{\tilde{g}}\cdot\sys_2(L(2,1)\setminus D^3,\tilde{g})=8\pi.
\end{align*}

For $p\ge 3$, we describe a geometric realization of $L(p,q)\backslash D^3 $ with scalar curvature bounded below by $R_0$, area-minimizing boundary achieving its $\pi_2$-systole.

We first recall a description of the lens space. Let $\mathbf{D}^3\subset \mathbf{R}^3$ be the unit ball, and write points in $\mathbf{R}^3=\mathbf{C}\times \mathbf{R}$ as $(z,t)$. Then $L(p,q)$ is obtained from $\mathbf{D}^3$ by identifying boundary points via
\begin{align}\label{eq: 101}
    (z,t)\sim \left(e^{\frac{2\pi qi}{p}}z,-t\right),
\qquad (z,t)\in \partial\mathbf{D}^3, \quad t>0.
\end{align}

Denote $(M_\alpha,g)$ to be the capillary prism given in Example \ref{example: perturbation counterexample} with capillary angle $\alpha = \frac{\pi}{p}$ (Note that $\frac{\pi}{p}< \alpha_0\approx 0.35\pi$ in Example \ref{example: perturbation counterexample} when $p\ge 3$), satisfying
\begin{align}\label{eq: 103}
    R_g\ge R_0 >0, \qquad\sys_2(M_\alpha,\partial_0 M_\alpha,g) = |\partial_- M_\alpha| >\frac{4\pi (1-\cos\alpha)}{R_0}.
\end{align}
Set $t_- = -t_0$ and $t_+ = 0$ in the construction of Example \ref{example: perturbation counterexample}, and let
\begin{align*}
    M_\alpha^0 = \left\{(s,r,\theta)\in \mathbf{R}\times S_+^2 \;\middle|\; -t_0\le s\le 0,\ 0\le r\le \rho(s)\right\}.
\end{align*}
be a copy of $M_\alpha$, where $\rho: [-t_0,0]\longrightarrow [a,b]$ gives rise to the defining curve of $\partial_0 M_\alpha^0$. We define
\begin{align*}
    M_\alpha^1 = \left\{(s,r,\theta)\in \mathbf{R}\times S_+^2 \;\middle|\; 0\le s\le t_0,\ 0\le r\le \rho(-s)\right\}
\end{align*}
to be another copy of $M_\alpha$. Note that $M_\alpha^0$ and $M_\alpha^1$ are glued together along their top faces:
\begin{align*}
    \partial_+ M_\alpha^0 = \partial_+ M_\alpha^1 \subset S_+^2\times \{0\}
\end{align*}

Next, we glue $\partial_0 M_\alpha^0$ and $\partial_0 M_\alpha^1$ together via the following identification:
\begin{align}\label{eq: 105}
    ((-s,\rho(s),\theta)\in \partial_0 M_\alpha^1)\sim ((s,\rho(s),\theta+\frac{2q\pi}{p})\in \partial_0 M_\alpha^0),\qquad s\in [-t_0,0].
\end{align}
Comparing this with \eqref{eq: 101}, we know that the resulting space is $L(p,q)\backslash D^3$. Since $\partial_- M_\alpha^{\varepsilon}$ is perpendicular to $\partial_0 M_\alpha^{\varepsilon}\quad (\varepsilon\in \{0,1\})$, $\partial_- M_\alpha^0$ and $\partial_-M_\alpha^1$ are glued into a closed boundary sphere. From \eqref{eq: 103}, the area of this sphere is greater than $\frac{8\pi (1-\cos\alpha)}{R_0}$.

Now we analyze the singular structure of the space $L(p,q)\setminus D^3$ obtained above. We focus on the codimension $2$ singular strata $\mathcal{S}^2$, which is the quotient of the circle $\partial(\partial_+ M_\alpha^0)=\partial(\partial_+ M_\alpha^1)$ under the $p$-sheeted covering map $z\mapsto z^p$. A neighborhood $\mathcal{U}$ of $\mathcal{S}^2$ in $L(p,q)\setminus D^3$ is a standard fibered solid torus of coefficient $(p,q)$.

Let $I$ be a segment in the singular strata $\mathcal{S}^2\cong S^1$. Then there exists a neighborhood $\mathcal{V}$ of $I$ in $L(p,q)\setminus D^3$ such that $\mathcal{V}$ is modeled on the Coxeter tessellation $I\times \mathcal{T}_{C(p)}$, where $C(p)$ is defined as in Example \ref{example: Coxeter chambers on the plane}. More precisely, $\mathcal{V}$ is divided into $2p$ chambers arranged cyclically and alternately around $I$, each with dihedral angle $\pi/p$; half of these chambers belong to $M_\alpha^0$, and the other half belong to $M_\alpha^1$.

Hence the tessellation of $L(p,q)\setminus D^3$ by the pieces $M_\alpha^\epsilon$ $(\epsilon\in\{0,1\})$ is Coxeter compatible. Applying Theorem \ref{thm: coxeter gluing smoothing}, together with the doubling trick used at the end of Section 5.3.1, we obtain a smooth metric $\tilde{g}$ on $L(p,q)\setminus D^3$ with an area-minimizing boundary sphere, satisfying
\begin{align*}
R_{\tilde{g}}\ge R_0>0,
\qquad
\sys_2(L(p,q)\setminus D^3,\tilde{g})>\frac{8\pi (1-\cos\alpha)}{R_0}.
\end{align*}

\subsubsection{Connected sums of copies of $S^2\times S^1$}

$\quad$

We describe the way to construct PSC metric on connected sums of $S^2\times S^1$. Fix $\beta_0 = \frac\pi3$ and let $M_{\beta_0}$ be the capillary prism with upper angle $\frac\pi 3$ and lower angle $\frac\pi 2$ constructed in Example \ref{example: perturbation counterexample}. Let $\partial_i \mathcal{P}_3\quad (i=1,2,3)$ be three boundary components of $\mathcal{P}_3$. Recall that $\mathcal{P}_3$ can be divided into six copies of $M_{\beta_0}$. By identifying $\partial_1 \mathcal{P}_3$ and $\partial_2\mathcal{P}_3$, we obtain a tessellation of $S^2\times S^1\backslash D^3$ into six copies of $M_{\beta_0}$. Let $\mathcal{P}_3'$ be another copy of $\mathcal{P}_3$. By identifying $\partial_i \mathcal{P}_3$ with $\partial_i\mathcal{P}_3'\quad (i=1,2)$, we obtain a tessellation of  $S^2\times S^1\backslash (D_1^3\bigsqcup D_2^3)$ into 12 copies $M_{\beta_0}$. Both tessellations satisfy the  Coxeter compatibility condition. Applying Theorem \ref{thm: coxeter gluing smoothing}, together with the doubling trick used at the end of Section 5.3.1, we obtain smooth metrics on $S^2\times S^1\backslash D^3$ and $S^2\times S^1\backslash (D_1^3\bigsqcup D_2^3)$ with area-minimizing boundary spheres satisfying
\begin{align}\label{eq: 104}
    (\inf R)\cdot \sys_2>4\pi.
\end{align}
By gluing these manifolds together, we obtain smooth metrics on connected sums of $k$ copies of $S^2\times S^1$ which also satisfies \eqref{eq: 104}.

\subsubsection{Connected sums of copies of $S^2\times S^1$ and lens spaces}

$\quad$

We need to construct two blocks, the modified block $M_\alpha^{\mathrm{mod}}$ and the connecting block $M_{\alpha,\beta}^{\mathrm{con}}$. Without loss of generality we assume $R_0 = 2$. Throughout this part, we use $S_\alpha^2$ to denote the spherical cap of geodesic radius $\alpha$ in the unit sphere.

Set $[t_-,t_+]= [0, t_0]\quad (t_0 >0)$ in the construction of Example \ref{example: perturbation counterexample}, and recall that the cylindrical block in Example \ref{example: perturbation counterexample} of upper angle $\alpha\in (0,\frac\pi2)$ is given by
\begin{align*}
    M_{\alpha}=\left\{(s,r,\theta)\in \mathbf{R}\times S_+^2 \;\middle|\; 0\le s\le t_0,\ 0\le r\le \rho(s)\right\}.
\end{align*}
Let $\delta>0$ be a small number. We define
\begin{align*}
    M_\alpha^\delta =  M_{\alpha}=\left\{(s,r,\theta)\in \mathbf{R}\times S_+^2 \;\middle|\; -1\le s\le t_0,\ 0\le r\le \rho^\delta(s)\right\},
\end{align*}
where
\begin{align*}
    \rho^\delta(s)
    =
    \begin{cases}
        \rho(-\delta), & -1\le s\le -\delta,\\
        \rho(-s), & -\delta\le s\le 0,\\
        \rho(s), & 0\le s\le t_0.
    \end{cases}
\end{align*}
The left panel of Figure \ref{figure: modified blocks} (a) illustrates $M_\alpha^\delta$.

The lateral boundary $\partial_0 M_\alpha^\delta$ consists of two smooth pieces
\begin{align*}
\Sigma_1 &= \left\{(s,r,\theta)\in \mathbf{R}\times S_+^2 \;\middle|\; -\delta\le s\le t_0,\ r= \rho^\delta(s)\right\},\\ \Sigma_2 &= \left\{(s,r,\theta)\in \mathbf{R}\times S_+^2 \;\middle|\; -1\le s\le -\delta,\ r= \rho^\delta(s)\right\},
\end{align*}
where $\Sigma_1$ is minimal and $\Sigma_2$ is weakly mean-convex. Moreover, the dihedral angle between $\Sigma_1$ and $\Sigma_2$ along their intersection is less than $\pi$. Hence, by the round-corner smoothing construction for rotationally symmetric surfaces, as in \cite[Theorem 4.1]{HMW2022}, we can smooth $\Sigma_1\cup\Sigma_2$ in an arbitrarily small neighborhood of their intersection, leaving it unchanged away from this neighborhood, so that the resulting boundary is weakly mean-convex. We denote the resulting modified block by $M_{\alpha}^{\mathrm{mod}}$, see the right panel of Figure \ref{figure: modified blocks} (a). A neighborhood of $\partial_-M_{\alpha}^{\mathrm{mod}}$ in  $M_{\alpha}^{\mathrm{mod}}$ is isometric to $S_{\alpha+\epsilon}^2\times [0,\frac{1}{2})$ for some $\epsilon$ sufficiently small. Moreover, we have
\begin{align*}
    \sys_2(M_{\alpha}^{\mathrm{mod}},\partial_0 M_{\alpha}^{\mathrm{mod}})\ge \sys_2(M_\alpha,\partial_0 M_\alpha)= 2\pi(1-\cos\alpha).
\end{align*}

Next we construct connecting block $M_{\alpha,\beta}^{\mathrm{con}}$ for given $0<\beta\le\alpha<\frac\pi2$. If $\beta = \alpha$, $M_{\alpha,\beta}^{\mathrm{con}}$ is simply constructed as the Riemannian product $S_{\alpha+\epsilon}^2\times [0,1]$ for some $\epsilon>0$ sufficiently small. If $\beta<\alpha$. let us start with $M_\beta$, which has bottom face $S_\beta^2$. There is a parallel slice in $M_\beta$ which is isometric to $S_\alpha^2$. Denote $M_{\alpha,\beta}$ to be the domain bounded by $S_{\alpha+\epsilon}^2$ and $S_\beta^2$ in $M_\beta$. By gluing $M_{\alpha,\beta}$ along $\partial_+ M_{\alpha,\beta} = S_{\alpha+\epsilon}^2$ with $S_{\alpha+\epsilon}^2\times [0,1]$, mollifying along the intersection, and at the same time perform the same construction near $\partial_- M_{\alpha,\beta} = S_\beta^2$ in the same way of constructing $M_{\beta,\mathrm{mod}}$ presented above, we obtain a new block $M_{\alpha,\beta}^{\mathrm{con}}$ which satisfies
\begin{itemize}
    \item $\partial_+M_{\alpha,\beta}^{\mathrm{con}}\,(\partial_-M_{\alpha,\beta}^{\mathrm{con}})$ is isometric to $S_{\alpha+\epsilon}^2\,(S_{\beta+\epsilon}^2)$, and a tubular neighborhood of $\partial_+M_{\alpha,\beta}^{\mathrm{con}}\,(\partial_-M_{\alpha,\beta}^{\mathrm{con}})$ is isometric to $S_{\alpha+\epsilon}^2\times [0,\frac12)(S_{\beta+\epsilon}^2\times [0,\frac12))$;
    \item $\partial_0 M_{\alpha,\beta}^{\mathrm{con}}$ is weakly mean convex;
    \item \begin{align*}
        \sys_2(M_{\alpha,\beta}^{\mathrm{con}},\partial_0 M_{\alpha,\beta}^{\mathrm{con}})\ge \sys_2(M_\beta,\partial_0 M_\beta)\ge 2\pi(1-\cos\beta).
    \end{align*}
\end{itemize}
See Figure \ref{figure: modified blocks} (b) for an illustration.

\begin{figure}
\centering
\includegraphics[width=0.9\textwidth]{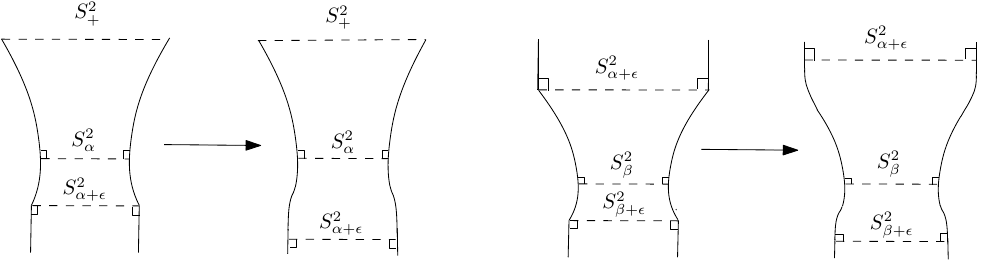}
\caption{An illustration of the construction of the modified block $M_\alpha^{\mathrm{mod}}$ (a) and the connecting block $M_{\alpha,\beta}^{\mathrm{con}}$ (b).}
\label{figure: modified blocks}
\end{figure}

We now describe our construction of PSC metrics on connected sums of copies of $S^2\times S^1$ and lens spaces using building blocks. For simplicity, we focus on the special case of $S^2\times S^1\# L(p,q)$; the case of general connected sums with multiple prime factors is similar. The construction $S^2\times S^1\# L(p,q)$ uses $10$ blocks, which we describe in the following: Let $\alpha = \frac{\pi}{3}$ and $\beta = \frac{\pi}{p}$. The blocks we need are
\begin{itemize}
    \item six copies of $M_{\alpha}^{\mathrm{mod}}$: $M_{\alpha,i}^{\mathrm{mod}}\quad (1\le i\le 6)$;
    \item two copies of $M_{\beta}^{\mathrm{mod}}$: $M_{\beta,j}^{\mathrm{mod}}\quad (j=1,2)$;
    \item two copies of $M_{\alpha,\beta}^{\mathrm{con}}$: $M_{\alpha,\beta,k}^{\mathrm{con}}\quad (k=1,2)$.
    
\end{itemize}

As in Sections 5.3.1 and 5.3.3, we construct $(S^2\times S^1)\backslash D^3$ from the blocks $M_{\alpha,i}^{\mathrm{mod}}$ $(1\le i\le 6)$ via the identifications in \eqref{eq: 106}, with
\begin{align*}
\partial((S^2\times S^1)\backslash D^3)
= \partial_- M_{\alpha,1}^{\mathrm{mod}}\cup \partial_- M_{\alpha,2}^{\mathrm{mod}}.
\end{align*}
Hence, a neighborhood of the boundary sphere is isometric to $DS_{\alpha+\epsilon}^2\times [0,\frac{1}{2})$, where $DS_{\alpha+\epsilon}^2$ denotes the double of the spherical cap $S_{\alpha+\epsilon}^2$.

Similarly, as in Section 5.3.2, we construct $L(p,q)\backslash D^3$ from the blocks $M_{\beta,j}^{\mathrm{mod}}$ $(j=1,2)$ by identifying the lateral faces via \eqref{eq: 105}. A neighborhood of its boundary sphere is isometric to $DS_{\beta+\epsilon}^2\times [0,\frac{1}{2})$.

Finally, by identifying $\partial_0M_{\alpha,\beta,1}^{\mathrm{con}}$ with $\partial_0 M_{\alpha,\beta,2}^{\mathrm{con}}$, we obtain a topological $S^2\times I$. Two neighborhoods of its two boundary components are isometric to $DS_{\alpha+\epsilon}^2\times [0,\frac{1}{2})$ and $DS_{\beta+\epsilon}^2\times [0,\frac{1}{2})$ respectively. By gluing $(S^2\times S^1)\backslash D^3$, $S^2\times I$, and $L(p,q)\backslash D^3$ along their boundaries, we obtain a tessellation of $S^2\times S^1\# L(p,q)$ into 10 building blocks listed above. The codimension $2$ singularities newly produced by $M_{\alpha,\beta,k}^{\mathrm{con}}$ $(k=1,2)$ have neighborhoods modeled on $S^1\times \mathcal{T}_{C(2)}$, which implies that this tessellation still satisfies the Coxeter compatibility condition. We then apply Theorem \ref{thm: coxeter gluing smoothing} to obtain
\begin{align*}
    \fs(M)>\min\{4\pi, 8\pi(1-\cos\beta)\}.
\end{align*}

\subsection{Prism $3$-manifolds}\label{subsection: prism manifolds}

In this Subsection and the next, we briefly discuss a potential extension of our approach to construct PSC metrics with large $\pi_2$-systole on a general $(S^3/G)\setminus D^3$. Let us first briefly recall basic topological facts of $3$-dimensional space forms. We refer the readers to the book by Martelli \cite[Section 12.2]{Martelli2025GeometricTopology} for complete details.

Identify $S^3$ with the set of unit quaternions with the standard group structure given by multiplications. Consider the homomorphism 
\[\Psi: S^3\times S^3\to \Isom ^+(S^3) =SO(4),\]
\[(q_1,q_2)\mapsto \{x\mapsto q_1 x q_2^{-1}\}. \]
$\Psi$ is a degree-$2$ covering with kernel $\{\pm(1,1)\}$, and hence 
\[SO(4) = S^3\times S^3 / \{\pm(1,1)\}. \]
By \cite[Proposition 6.2.15]{Martelli2025GeometricTopology}, the finite subgroups of $SO(3)$, up to conjugation, are:
\[C_n, \quad D_{2m}, \quad T_{12} \cong A_4,\quad  O_{24}\cong S_4, \quad I_{60}\cong A_5.\]
Here $C_n$ is the cyclic group of order $n$, $D_{2m}, T, O, I$ are the orientation-preserving isometry groups of the regular $m$-prism (the dihedral group), tetrahedron, octahedron (or cube), and icosahedron (or dodecahedron). Under the standard degree-$2$ covering map $\Phi: S^3\to SO(3)$, these groups lift to
\[D_{4m}^*, \quad T^*_{24}, \quad O^*_{48}, \quad I^*_{120},\]
called the binary dihedral, tetrahedral, octahedral, and icosahedral group. It follows that together with the cyclic groups, they form all finite subgroups of $S^3$.

\begin{proposition}[{\cite[Proposition 12.2.9]{Martelli2025GeometricTopology}}]\label{prop: finite subgroups of S3}
    Every finite subgroup of $S^3$ is one of the following:
    \[C_n, \quad D_{4m}^*, \quad T^*_{24}, \quad O^*_{48}, \quad I^*_{120}.\]
\end{proposition}
We proceed to describe all finite subgroups acting freely on $S^3$. It turns out that they include the cyclic group, the groups in the form
\[\Psi(G\times C_n), \]
where  $G$ is one of the binary groups in Proposition \ref{prop: finite subgroups of S3} and $C_n$ the cyclic group such that $|G|$ and $n$ are coprime, and certain finite index subgroups of them. See \cite[Table 12.4]{Martelli2025GeometricTopology}. When the second factor $C_n$ is the trivial group, the groups $\Psi(G\times \{1\})$ give rises to a special class of spherical space forms, called \textit{single-action spherical space forms}, or in a more geometric manner, \textit{homogeneous spherical space forms}. Observing \cite[Table 12.4]{Martelli2025GeometricTopology}, every spherical space form is a finite cyclic quotient of a single-action spherical space form. Equivalently, every spherical space form is regularly cyclically covered by one of the homogeneous spherical space forms $S^3/G$, where $G<S^3$ is cyclic, binary dihedral, binary tetrahedral, binary octahedral, or binary icosahedral.

Following Theorem \ref{thm: coxeter gluing smoothing}, we would like to construct PSC metrics with large $\pi_2$-systole on other spherical space forms. A challenging first attempt is to focus on single-action spherical space forms other than the lens spaces. When $G = D^*_{4m}$ is the binary dihedral group, $S^3/G$ is called a prism manifold. Recall that 
\begin{equation}\label{eq: binary dihedral}
    D_{4m}^* = \langle a,b | a^{2m} = 1, b^2=a^m, bab^{-1} =a^{-1}\rangle.
\end{equation}
A realization of $S^3/G$ is the following: take a prism $P = D^{2m}\times I$ where $D^{2m}$ is a regular planar $2m$-gon centered at $0$. Define $a$ the rotation of $P$ by $\tfrac{\pi}{m}$ around $\{0\}\times I$, and $b$ the action that identifies each of the rectangular side with the opposite one, with a rotation by $\frac\pi2$. \eqref{eq: binary dihedral} is thus satisfied, and hence $P/\langle a,b \rangle$ is diffeomorphic to $S^3/D_{4m}^*$.

We propose the following approach to construct PSC metrics on single-action spherical space forms with a $3$-ball removed achieving quantitatively large $\pi_2$-systole. Here we provide the necessary combinatorial and geometric conditions. We did some preliminary numerical experiments in search of explicit constructions, and found evidence of their existence. However, we do not have any elegant explicit solutions at hand.

We first handle prism manifolds. Fix a $3$-ball $B^3$ at the center of the prism $P$. The goal is to tessellate $P\setminus B^3$ into capillary building blocks which satisfy the assumptions of Theorem \ref{thm: coxeter gluing smoothing} so they glue smoothly. We thus need the following capillary blocks.

\begin{construction}[Building blocks for prism manifolds]\label{construction: block for prism manifolds}
    Start with a triangular prism $\Delta \times [-1,1]$, where $\Delta\subset \R^2$ is a planar triangle with vertices $o = (0,0), v_1, v_2$. Remove a ball $B_{1/2}$ centered at $(0,0,0)$ and obtain $M = (\Delta \times [-1,1])\setminus B_{1/2}$. The capillary building block for the prism manifold is a Riemannian polyhedron $(M,g)$ satisfying:
    \begin{enumerate}
        \item All facets of $M$ are mean convex with respect to the outward unit normal.
        \item The top and bottom faces $\partial_\pm M = \Delta \times \{\pm 1\}$ are isometric; the two radial faces $R_j = (\overline{ov_j} \times [-1,1])\setminus B_{1/2}$ are isometric.
        \item The lateral face $S = \overline{v_1 v_2}\times [-1,1]$ admits an isometry $(s, t)\mapsto (1-t, 1-s)$.
        \item We prescribe the dihedral angles between facets:
        \[\angle(R_1,R_2) = \frac{\pi}{m}, \quad \angle(\partial_\pm M, R_j ) = \frac{\pi}{2}, \quad \angle(S, R_j) = \frac{\pi}{3},\quad \angle(\partial_\pm M, S) = \frac{2\pi}{3}. \]
        \item Finally, letting $\Sigma = (\Delta \times [-1,1])\cap \partial B_{1/2}$, we require that $\Sigma$ is area-minimizing and prescribe the dihedral angle
        \[\angle(\Sigma, R_j) = \frac{\pi}{2}.\]
    \end{enumerate}
\end{construction}

Denote a Riemannian polyhedron in Construction \ref{construction: block for prism manifolds} by $M_m$. One may take $2m$ isometric copies of $M_m$ and glue along their radial faces to obtain $(D_{2m}\times [-1.1])\setminus B^3$, a prism with a ball removed at the center, where the isometric actions $a, b$ can then be used to glue the outer faces to obtain $(S^3/D^*_{4m})\setminus B^3$. The mean convexity assumption and the dihedral angles of $M_m$ guarantee assumptions of Theorem \ref{thm: coxeter gluing smoothing}, and hence it can be used to construct a smooth metric on the prism manifold with a $3$-ball removed, where the relative $\pi_2$-systole is $2m|\Sigma|$.

\subsection{Quaternionic, octahedral and dodecahedral spaces}\label{subsection: 3 special spaces}

We subsequentially describe the construction of capillary building blocks achieving large relative $\pi_2$-systole, for the quaternionic space, the binary octahedral space and the binary dodecahedral space, with a $3$-ball removed. These spaces inherit elegant platonic building blocks from which we take advantage. We note that the binary octahedral space is actually a prism manifold (with the binary dihedral group $D^*_4)$. However, the construction described below is different from the one in Construction \ref{construction: block for prism manifolds} when $m=2$.

\begin{construction}[Building blocks for the quaternionic space, the binary octahedral space and the binary dodecahedral space]\label{construction: 3 special spaces}
    For $k=3,4,5$, let $P_k$ be a planar regular $k$-gon. The building block for the quaternionic space, the binary octahedral space and the binary dodecahedral space is a capillary prism $M = P_k\times [-1,1]$, $k=3,4,5$, respectively, with a Riemannian metric $g$, satisfying:
    \begin{enumerate}
        \item Each face of $M$ is mean convex with respect to the outward unit normal.
        \item $(M,g)$ admits the isometry of rotation by $\tfrac{2\pi}{k}$.
        \item The bottom face $\partial_{-}M$ is area-minimizing in the relative homotopy class.
        \item We prescribe the dihedral angles as follows: for any lateral face $S$, 
        \[\angle (S,\partial_+ M) =\frac{\pi}{3},\quad \angle(S,\partial_{-}M) = \frac{\pi}{2},\]
        and for any two adjacent lateral faces $S_1,S_2$, we require that:
        \[\angle(S_1,S_2) = \begin{cases}
            &\frac{\pi}{2},\quad \text{ if }k=3;\\
            &\frac{2\pi}{3}, \quad \text{ if }k\in \{4,5\}.
        \end{cases} \]
    \end{enumerate}
\end{construction}

We may then use appropriately chosen isometric copies of these building blocks to construct the desired spherical space forms with a $3$-ball removed. The construction is summarized in Figure \ref{figure: Platonic building blocks}.

\begin{figure}
\centering
\includegraphics[width=0.8\textwidth]{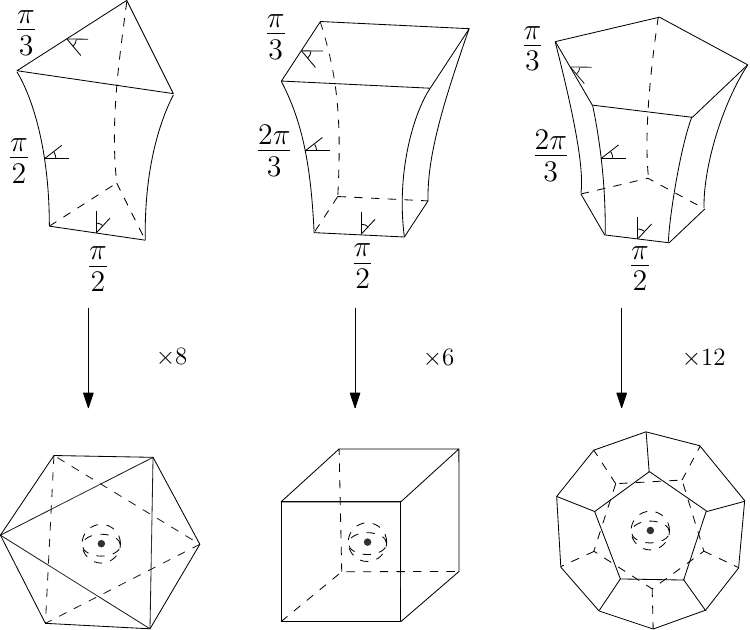}
\caption{The construction of Platonic solids using building blocks}
\label{figure: Platonic building blocks}
\end{figure}

\section{Comparison with 3-manifolds constructed by Gromov-Lawson-Schoen-Yau surgery}\label{section:comparison with surgery}

In this section, we compare the examples constructed in Section 5 with manifolds constructed by performing Gromov-Lawson/Schoen-Yau surgeries \cite{GL80}\cite{SY79b} on the round sphere $S^3$. Let $(S^n,g_{S^n})$ be the round $n$-sphere with constant curvature 1. Fix $r_0\in (0,\frac{\pi}{2})$ and $p\in S^3$. Let $B_{r_0}(p)\subset S^3$ be the geodesic ball of radius $r_0$ with center $p$ in $(S^3,g_{S^3})$, then $\partial B_{r_0}(p)$ is isometric to $(S^2,(\sin^2r_0)g_{S^2})$, and the mean curvature of $\partial B_{r_0}(p)$ in $(S^3,g_{S^3})$ equals $2\cot r_0$ with respect to the outward pointing normal.

Next, we construct a neck metric $g_{\mathcal N}$ on $\mathcal N=S^2\times[0,L]$ with the following properties:
\begin{itemize}
\item $\partial_+\mathcal N=S^2\times\{L\}$ is isometric to $(S^2,(\sin^2 r_0)g_{S^2})$, and its mean curvature with respect to the outward pointing normal is $2\cot r_0$;
\item $\partial_-\mathcal N=S^2\times\{0\}$ is minimal;
\item the scalar curvature satisfies $R_{g_{\mathcal N}}\ge R_0$, where $R_0>0$ is a constant to be determined.
\end{itemize}
We restrict our attention to the case when $g_{\mathcal N}$ is a warped product metric of the form
\begin{align*}
g_{\mathcal N}=d\rho^2+f(\rho)^2g_{S^2},\qquad \rho\in[0,L],
\end{align*}
where $f'(\rho)>0$ for $\rho\in(0,L)$ (this is the form of metric used by Gromov-Lawson). The coordinate spheres $S^2\times{\rho}$ then form a smooth foliation of $\mathcal N$. After reparametrizing this foliation, we obtain a solution $\{\Sigma_t\}_{0\le t\le T}$ of the inverse mean curvature flow such that
\begin{align*}
\Sigma_0=\partial_-\mathcal N,\qquad \Sigma_T=\partial_+\mathcal N .
\end{align*}

From the Geroch monotonicity, we have
\begin{align}\label{eq: 107}
        m_H(\partial_-\mathcal{N})\le m_H(\partial_+\mathcal{N}),
    \end{align}
    where 
    \begin{align*}
        m_H(\Sigma) = \left(\frac{|\Sigma|}{16\pi}\right)^{\frac{1}{2}}\left( 1-\frac{1}{16\pi}\int_\Sigma H^2-\frac{R_0}{24\pi}|\Sigma|\right)
    \end{align*}
    is the modified Hawking mass. Denote $x = \left(\frac{|\partial_-\mathcal{N}|}{16\pi}\right)^{\frac{1}{2}}$. We compute
    \begin{align}\label{eq: 108}
        m_H(\partial_-\mathcal{N})= x(1-\frac{2}{3}R_0x^2),\qquad m_H(\partial_+\mathcal{N})= \frac{\sin^3r_0}{2}(1-\frac{R_0}{6}).
    \end{align}
    Because $f(\rho)$ is increasing, we have $x\le \sqrt{\frac{|\partial_+ \mathcal{N}|}{16\pi}} = \frac{\sin r_0}{2}$. By \eqref{eq: 107} and \eqref{eq: 108}, we obtain
    \begin{align*}
        R_0\le \frac{\frac{\sin^3 r_0}{2}-x}{\frac{\sin^3 r_0}{12}-\frac{2}{3}x^3}.
    \end{align*}
    Therefore,
    \begin{align*}
        |\partial_-\mathcal{N}|\cdot R_0\le 16\pi \frac{x^2(\frac{\sin^3 r_0}{2}-x)}{\frac{\sin^3 r_0}{12}-\frac{2}{3}x^3} := F(x).
    \end{align*}
    
    Let $y_*(r_0)\in (0,1)$ denote the smallest positive solution of the equation
    \begin{align}\label{eq: 109}
        y^3-3y+2\sin^3r_0 = 0.
    \end{align}
    By direct calculation, $F'(x) = 0$ if and only if $(2x)^3-3\cdot (2x)+2\sin^3 r_0 = 0$. Hence, $F(x)$ attains its maximum in $(0,\frac{\sin r_0}{2})$ when $x = \frac{1}{2}y_*(r_0)$. In the general case, we have the upper bound estimate
    \begin{align*}
        |\partial_-\mathcal{N}|\cdot R_0 = F(x)\le F(\frac{y_*(r_0)}{2})\le 8\pi y_*(r_0)^2.
    \end{align*}
    
    In the equality case, we have $x = \frac{y_*(r_0)}{2}$ and
    \begin{align*}
        |\partial_-\mathcal{N}| = 4\pi y_*(r_0)^2,\qquad R_0 = 2,\qquad F(\frac{y_*(r_0)}{2}) = 8\pi y_*(r_0)^2.
    \end{align*}
    Furthermore, the equality implies that $g_{\mathcal{N}}$ has constant scalar curvature. Since $g_{\mathcal{N}}$ is a warped product metric, it is a Schwarzschild–de Sitter metric with constant scalar curvature $2$. From \eqref{eq: 109}, we have $y_*(r_0)\sim \frac23 r_0^3$ as $r_0\to 0$. Hence, the equality model $(\mathcal{N},g_{\mathcal{N}})$ satisfies
    \begin{align}\label{eq: 110}
        |\partial_- \mathcal{N}|\cdot R_0 \sim \frac{32\pi}{9}r_0^6.
    \end{align}
    We denote the neck constructed above as $(\mathcal{N}_{r_0}, g_{\mathcal{N}_{r_0}})$
    
The following table lists the values of $y_*(r_0)$ for several special choices of $r_0$:
\[
\begin{array}{c|c|c}
r_0 & y_*(r_0) & 8\pi y_*^2(r_0) \\ \hline
\frac{\pi}{2} & 1 & 8\pi \\
\frac{\pi}{3} & 0.467 & 1.744\pi \\
\frac{\pi}{4} & 0.240 & 0.462\pi \\
\frac{\pi}{10} & 0.019 & 0.003\pi
\end{array}
\]

\begin{example}
    Let $P_1,P_2,P_3$ be 3 points on a great circle in $(S^3,g_{S^3})$, the distance between each other equals $\frac{2\pi}{3}$. Let $\alpha = \frac{\pi}{3}-\epsilon$ for $\epsilon$ arbitrarily small. By gluing 3 copies of $(\mathcal{N}_\alpha,g_{\mathcal{N}_\alpha})$ to $S^3\backslash (B_\alpha(P_1)\cup B_\alpha(P_2)\cup B_\alpha(P_3))$ along $\partial_+ \mathcal{N}_\alpha$, we obtain a geometric realization of $\mathcal{P}_3$ with minimal boundary. Since $\partial_+\mathcal{N}_\alpha$ is isometric to $\partial B_\alpha(P_i)\quad (i=1,2,3)$ and the mean curvature matches along the boundary, we can apply  the mollification theorem by \cite{Miao2002} to obtain a metric $g$ on $\mathcal{P}_3$, satisfying
    \begin{align*}
        R_g\ge R_0 = 2,\qquad R_0\cdot \sys_2(\mathcal{P}_3,g)\approx 8\pi y_*(\frac\pi3)^2 \approx 1.744\pi.
    \end{align*}
    This can be compared with the value $4\pi$ obtained by the construction in Section 5.3.1.
\end{example}

\begin{example}
%    Let $P_1,P_2,\dots, P_p$ be $p$ points on a great circle in $(S^3,g_{S^3})$, the distance between each other equals $\frac{2\pi}{p}$. Let $\beta = \frac\pi p$. By gluing $p$ copies of $(\mathcal{N}_\beta,g_{\mathcal{N}_\beta})$ to $S^3\backslash (B_\beta(P_1)\cup B_\beta(P_2)\cup\dots \cup  B_\beta(P_p))$, we obtain a geometric realization of the universal covering of $L(p,q)\backslash D^3$ with minimal boundary. Quotienting this by $Z_p$, we obtain a geometric realization of $L(p,q)\backslash D^3$ with minimal boundary. Once again, we can apply \cite{Miao2002} to obtain a metric $g$ on $L(p,q)\backslash D^3$, satisfying
Recall that the injective radius of  $L(p,q)$ obtained from the quotient of $(S^3,g_{S^3})$ is $\frac{\pi}{p}$. Let $\beta = \frac\pi p-\epsilon$ for $\epsilon$ arbitrarily small. Fix a point $P\in L(p,q)$. By gluing $L(p,q)\backslash B_{\beta}(P)$ with $(\mathcal{N}_\beta,g_{\mathcal{N}_\beta})$ along $\partial_+ \mathcal{N}_\beta$, we obtain a geometric realization of $L(p,q)\backslash D^3$ with minimal boundary. Once again, we apply \cite{Miao2002} to obtain a metric $g$ on $L(p,q)\backslash D^3$, satisfying
    \begin{align*}
        R_g\ge R_0 = 2,\qquad R_0\cdot \sys_2(L(p,q)\backslash D^3, g)\approx 8\pi y_*(\frac\pi p)^2 = O(p^{-6}),
    \end{align*}
    where we have used \eqref{eq: 110}. This can be compared with the value $8\pi(1-\cos\frac \pi p)\sim 4\pi^3p^{-2}$ obtained in Section 5.3.2.
\end{example}

\begin{example}
    Let $\Sigma(2,3,5)$ be the Poincare sphere. For $\Sigma(2,3,5)\backslash D^3$, the injective radius of $S^3/ I_{120}^*$ obtained from the quotient of $(S^3,g_{S^3})$ is $\frac{\pi}{10}$. The method above gives the value $8\pi x_*^2(\frac\pi {10})\approx 0.003\pi$.
\end{example}

\section{Rigidity of the standard cylindrical block under pinched curvature conditions} 

$\quad$

In this section, we give the proof of Theorem \ref{thm: 2 systole estimate, rigidity}.

We need the following isoperimetric inequality for surfaces with Gauss curvature
bounded from above. The inequality follows from \cite[Theorem 4.3]{Os78}
and \cite[Theorem 2.2.1]{BuragoZalgaller1988}, and the rigidity statement is
a special case of \cite[Theorem 4.3.1]{BuragoZalgaller1988}.

\begin{lemma}\label{lem: isoperimetric inequality assuming sectional curvature upper bound}
    Let $(D,g_D)$ be a Riemannian disk with Gauss curvature $K\le K_0$, where
    $K_0>0$ is a constant. Then
    \begin{align*}
        |\partial D|^2\ge 4\pi |D|-K_0|D|^2.
    \end{align*}
    Equality holds if and only if $(D,g_D)$ is isometric to a geodesic disk in
    the round sphere of constant Gauss curvature $K_0$.
\end{lemma}

The following proposition specifies Proposition \ref{prop: differential inequality for I} to the curvature conditions assumed by Theorem \ref{thm: 2 systole estimate, rigidity}.

\begin{proposition}\label{prop: free energy profile monotonicity 2}
     Let $(M^3,g)$ be an energy-essential Riemannian cylinder with scalar curvature $R_g\ge R_0>0$, sectional curvature $\sec\le \frac{R_0}{2}$, and let $\alpha\in (0,\frac{\pi}{2}]$. Assume $0<\angle_{+0}(M,g)\le \alpha$, $\angle_{-0}(M,g) = \frac{\pi}{2}$, $\partial_{\pm}M$ are both minimal surfaces, and $\bar{H}\ge 0$ on $\partial_0 M$. We also assume $(M^3,g)$ satisfies Assumption A.

     Let $I:(0,|\partial_0 M|)\longrightarrow (0,\infty)$ be the free energy profile. Given $\delta>0$ sufficiently small. Then for every $s_0\in (\delta, |\partial_0M|-\delta)$, there exists $\tilde{\delta}\in (0,\delta)$ and $\psi\in C^\infty(s_0-\tilde{\delta},s_0+\tilde{\delta})$ such that
\begin{enumerate}
    \item\label{sec4 item1} $\psi(s_0) = I(s_0)$;
    \item\label{sec4 item2} $\psi(s)\ge I(s)$ whenever $s\in (s_0-\tilde{\delta},s_0+\tilde{\delta})$;
    \item\label{sec4 item3} $\psi'(s_0) = \Phi(\sigma(s_0))$;
    \item\label{sec4 item4} 
    \begin{align}\label{eq: 74}
\psi''(s_0)
\le
\frac{
\bigl(1-\psi'(s_0)^2\bigr)
\left(2\pi-\frac{R_0}{2}\psi(s_0)\right)
}{
\psi(s_0)\left(4\pi-\frac{R_0}{2}\psi(s_0)\right)
}.
\end{align}
\end{enumerate}
Moreover, $\psi\in C^\infty(s_0-\tilde{\delta},s_0+\tilde{\delta})$ can be chosen such that the strict inequality in \eqref{eq: 74} holds, unless
\begin{itemize}
    \item $s_0 = s_{\sigma(s_0)}^-$ or $s_0 = s_{\sigma(s_0)}^+$;
    \item every $\Sigma_{\sigma(s_0)}\in\mathcal{E}_{\sigma(s_0)}$ is isometric to a geodesic disk in a round sphere of Gauss curvature $\frac{R_0}{2}$; 
    \item $\Sigma_{\sigma(s_0)}\in\mathcal{E}_{\sigma(s_0)}$ is totally geodesic;
    \item $\bar{H} = 0$ on $\partial\Sigma_{\sigma(s_0)}$ for every $\Sigma_{\sigma(s_0)}\in \mathcal{E}_{\sigma(s_0)}$.
\end{itemize}

\end{proposition}

\begin{proof}
    When $s_0\in (0,\partial_0 M)$, we have $0<I(s_0)<\frac{4\pi}{R_0}$ and $\Phi(\sigma(s_0))\in (0,\cos\alpha)$ (To be made clear later). If $s_0\in (s_{\sigma(s_0)}^-,s_{\sigma(s_0)}^+)$, then recall from Proposition \ref{prop: differential inequality for I} that $\psi$ can be chosen such that $\psi''(s_0) = 0$. Since $I$ is increasing by Lemma \ref{lem: I is increasing}, it follows from \eqref{eq: 33} that
    \begin{align*}
        \psi(s_0) = I(s_0) = I(s_{\sigma(s_0)}^+) = |\Sigma_{\sigma(s_0)}^+|\le \frac{4\pi}{R_0}.
    \end{align*}
    Therefore, the right hand side of \eqref{eq: 74} is strictly positive, which implies that \eqref{eq: 74} holds strictly. In the following, we focus on the case that $s_0 = s_{\sigma(s_0)}^{\pm}$.

    Since $\Sigma_{\sigma(s_0)}$ is a capillary minimal surface in $M$, the Gauss equation implies
    \begin{align*}
        K_{\Sigma_{\sigma(s_0)}} = K_{12}-\frac{1}{2}|A|^2\le K_{12}\le \frac{R_0}{2}.
    \end{align*}
    By applying Lemma \ref{lem: isoperimetric inequality assuming sectional curvature upper bound} we obtain
    \begin{align}\label{eq: 81}
        |\partial \Sigma_{\sigma(s_0)}|^2\ge 4\pi |\Sigma_{\sigma(s_0)}|-\frac{R_0}{2}|\Sigma_{\sigma(s_0)}|^2.
    \end{align}
    Plugging into \eqref{eq: 5} and using $R_g\ge R_0$ and $\bar{H}\ge 0$, we obtain \eqref{eq: 74}.

    If equality holds, the argument above shows that the case $s_0\in (s_{\sigma(s_0)}^-,s_{\sigma(s_0)}^+)$ is impossible. Hence either $s_0=s_{\sigma(s_0)}^-$ or $s_0=s_{\sigma(s_0)}^+$. Moreover, it follows from \eqref{eq: 5} that $\Sigma_{\sigma(s_0)}$ is totally geodesic, that $\bar{H}=0$ on $\partial\Sigma_{\sigma(s_0)}$, and that equality holds in \eqref{eq: 81}. In view of Lemma \ref{lem: isoperimetric inequality assuming sectional curvature upper bound}, $\Sigma_{\sigma(s_0)}$ is isometric to a geodesic disk in a round sphere of Gauss curvature $\frac{R_0}{2}$.
\end{proof}

We define the mass functions
\begin{align*}
&m_0(s) = I(s)(\frac{8\pi}{R_0}-I(s)),\\
&m(s) = m_0(s)(1-(\bar{I}'(s))^2).
\end{align*}
In the case that $I$ is smooth, direct calculation yields
    \begin{align*}
       (\log m(s))' &= (\log I(s)(\frac{8\pi}{R_0}-I(s)))'+(\log(1-I'(s)^2))'\\
       & = \frac{2I'(\frac{4\pi}{R_0}-I)}{I(\frac{8\pi}{R_0}-I)} -\frac{2I'I''}{1-I'^2}\\
       &\ge \frac{2I'(\frac{4\pi}{R_0}-I)}{I(\frac{8\pi}{R_0}-I)}-\frac{2I'(2\pi-\frac{R_0}{2}I)}{4\pi I-\frac{R_0}{2}I^2} = 0.
    \end{align*}
In the general case, we can prove that $m(s)$ is nondecreasing by arguing as in \cite[Lemma 3]{Bray97}.

Similar to Lemma \ref{lem: comparison of mass function for different angles}, we have the following lemma:
\begin{lemma}\label{lem: comparison of mass function for different angles 2}
    For $\alpha<\theta_2<\theta_1<\frac\pi2$, let $t_i = \Phi^{-1}(\cos\theta_i)\quad(i=1,2)$, and let $\Sigma_{t_i}\in\mathcal{E}_{t_i}$ be an area-minimizing surface with capillary angle $\theta_i$ in $(M,g)$. Then
    \begin{align}\label{eq: 82}
    |\Sigma_{t_1}|(\frac{8\pi}{R_0}-|\Sigma_{t_1}|)\sin^2\theta_1\le |\Sigma_{t_2}|(\frac{8\pi}{R_0}-|\Sigma_{t_2}|)\sin^2\theta_2.
\end{align}
Moreover, if the equality holds, then
\begin{enumerate}
    \item for any $t\in \Phi^{-1}((\cos\theta_1,\cos\theta_2))$, we have $s_t^- = s_t^+$, and for any $\Sigma_t\in \mathcal{E}_t$, it holds
\begin{itemize}
\item $\Sigma_t$ is isometric to a geodesic disk in a sphere of constant Gauss curvature $\frac{R_0}{2}$;
    \item $\Sigma_t$ is totally geodesic;
    \item $\bar{H} = 0$ on $\partial\Sigma_t$.
\end{itemize}
\item $I(s)$ is differentiable on $(s_{t_1},s_{t_2})$. Moreover,
\begin{align}\label{eq: 86}
    I(s)(\frac{8\pi}{R_0}-I(s))(1-I'(s)^2) = |\Sigma_{t_1}|(\frac{8\pi}{R_0}-|\Sigma_{t_1}|)\sin^2\theta_1
\end{align}
for all $s\in (s_{t_1},s_{t_2})$.

%Let $(M_\alpha,g_0)$ be the standard cylinder block with upper angle $\alpha$ and scalar curvature $R_0$, and let $I_0(s)$ be the free energy profile for $(M_\alpha,g_0)$, then $|\partial_0 M_\alpha| = |\partial_0 M|$, and $I(s) = I_0(s)$ for all $s\in (0,|\partial_0 M|)$.
\end{enumerate}

\end{lemma}
\begin{proof}
    The inequality part follows from the same proof as that of Lemma \ref{lem: comparison of mass function for different angles}. For the equality part, (1) follows from Proposition \ref{prop: free energy profile monotonicity 2}. For (2), we first note that by tracing the proof of Lemma \ref{lem: comparison of mass function for different angles}, one obtains that \eqref{eq: 86} holds when $I$ is differentiable at $s$. For general $s$, let $s_k\nearrow s$ be a sequence of points such that $I$ is differentiable at $s_k$. By Lemma \ref{lem: semiconcavity of I} (2) and the continuity of $I$, we have
    \begin{align*}
        m_0(s_k)(1-I'(s_k)^2)\to m_0(s)(1-\underline{I}'(s)^2) = |\Sigma_{t_1}|(\frac{8\pi}{R_0}-|\Sigma_{t_1}|)\sin^2\theta_1.
    \end{align*}
    Similarly, we have
    \begin{align*}
         m_0(s)(1-\bar{I}'(s)^2) = |\Sigma_{t_1}|(\frac{8\pi}{R_0}-|\Sigma_{t_1}|)\sin^2\theta_1.
    \end{align*}
    Since $I$ is increasing, we have $\underline{I}'(s) = \bar{I}'(s)$, which implies $I$ is differentiable at $s$.
\end{proof}

$\quad$

\begin{proof}[Proof of Theorem \ref{thm: 2 systole estimate, rigidity} (inequality part)]
    By letting $\theta_1\to\frac\pi2$ in Lemma \ref{lem: comparison of mass function for different angles 2}, we have $t_1\to 0$ and $\Sigma_{t_1}$ converges smoothly to a free boundary area-minimizing surface $\Sigma_0$. From the outermost minimizing property of $\partial_-M$, we conclude that $\Sigma_0 = \partial_- M$. It follows from \eqref{eq: 82} that
\begin{align}\label{eq: 87}
    |\partial_- M|(\frac{8\pi}{R_0}-|\partial_- M|)\le|\Sigma_{t}|(\frac{8\pi}{R_0}-|\Sigma_{t}|)\sin^2\theta\le  \frac{16\pi^2}{R_0^2}\sin^2 \alpha
\end{align}
for $\Sigma_t\in\mathcal{E}_t,\theta = \arccos\Phi(t)$ and $t\in (0,\Phi^{-1}(\cos\alpha))$. Therefore we obtain
\begin{align*}
    R_0\cdot|\partial_- M|\le 4\pi(1-\cos\alpha)\qquad\text{or}\qquad R_0\cdot|\partial_- M|\ge 4\pi(1+\cos\alpha)
\end{align*}
The latter is impossible from Lemma \ref{lem: spectral positivity on surfaces}, so
\begin{align}\label{eq: 75}
    R_0\cdot|\partial_- M|\le 4\pi(1-\cos\alpha).
\end{align}
\end{proof}

Now we consider the equality case of \eqref{eq: 75}. By letting $\theta_2\to\alpha$, we have $t_2 \to \Phi^{-1}(\cos\alpha)$ and $\Sigma_{t_2}$ converges smoothly to an area-minimizing capillary surface $\Sigma_{\Phi^{-1}(\cos\alpha)}$ with capillary angle $\alpha$. From the innermost minimizing property of $\partial_+M$, we conclude $\Sigma_{\Phi^{-1}(\cos\alpha)} = \partial_+ M$. From the equality in \eqref{eq: 87} we see that
\begin{align*}
    |\Sigma_{t}|(\frac{8\pi}{R_0}-|\Sigma_{t}|)\sin^2\theta = \frac{16\pi^2}{R_0^2}\sin^2 \alpha
\end{align*}
for all $\Sigma_t\in\mathcal{E}_t$ with $t\in (0,\Phi^{-1}(\cos\alpha))$, with $\theta = \arccos(\Phi(t))$. We also have $|\partial_+ M| = \frac{4\pi}{R_0}$. By Lemma \ref{lem: comparison of mass function for different angles 2}, we have $s_t^-=s_t^+$ for each $t\in (0,\Phi^{-1}(\cos\alpha))$. Moreover, for each $t\in (0,\Phi^{-1}(\cos\alpha))$ and any $\Sigma_t\in\mathcal{E}_t$, it holds
\begin{itemize}
    \item $\Sigma_t$ is isometric to a geodesic ball in a sphere of constant Gauss curvature $\frac{R_0}{2}$;
    \item $\Sigma_t$ is totally geodesic;
    \item $\bar{H} = 0$ on $\partial\Sigma_t$.
\end{itemize}
Furthermore, $I(s)$ is differentiable on $(0,|\partial_0 M|)$, with
\begin{align}\label{eq: 90}
    I(s)(\frac{8\pi}{R_0}-I(s))(1-I'(s)^2) = \frac{16\pi^2}{R_0^2}\sin^2 \alpha
\end{align}
holds for all $s\in (0,|\partial_0 M|)$, and
\begin{align}\label{eq: 88}
    I(0) = |\partial_- M| = \frac{4\pi(1-\cos\alpha)}{R_0},\qquad I(|\partial_0 M|) = |\partial_+ M| = \frac{4\pi}{R_0}.
\end{align}

Let $(M_\alpha,g_0)$ be the standard cylindrical block with upper angle $\alpha$ and scalar curvature $R_0$, and let $I_0(s)$ be the free energy profile for $(M_\alpha,g_0)$. From \eqref{eq: 90}\eqref{eq: 88}, we see that $I(s)$ and $I_0(s)$ satisfies the same ODE, which yields $|\partial_0 M| = |\partial_0M_\alpha|$, and
\begin{align}\label{eq: 89}
    I(s) = I_0(s) \text{ for all } s\in (0,|\partial_0 M|).
\end{align}

$\quad$

In the following, we derive some structural description of the family of surfaces
$\Sigma_t\in\mathcal{E}_t$ for $t\in (0,\Phi^{-1}(\cos\alpha))$. We first show that $\Sigma_{t_1}\cap \Sigma_{t_2} = \emptyset$ when $0<t_1<t_2<\Phi^{-1}(\cos\alpha)$. This follows from an argument in the proof of \cite[Theorem 1.3]{EK24}. Assume $\Sigma_{t_1}\cap \Sigma_{t_2} \ne \emptyset$, then the minimizing property of $\Sigma_i (i=1,2)$ implies that $\partial\Sigma_{t_1}\cap \partial\Sigma_{t_2} \ne\emptyset$, and $\partial\Sigma_{t_1}$ and $\partial\Sigma_{t_2}$ intersect transversally. Let $T_i (i=1,2)$ be the tangent vector of $\partial\Sigma_{t_i}$, and denote $\cos\theta_i = \Phi(t_i)$, we have
\begin{align*}
    \nabla_{T_i}\bar{N} = -\cos\theta_i\cdot\nabla_{T_i} N+\sin\theta\cdot\nabla_{T_i}\nu = (k_{\Sigma_{t_i}}\sin\theta_i)T_i,
\end{align*}
where we have used the fact that $\Sigma_{t_i}$ is totally geodesic. Since $k_{\Sigma_{t_i}}>0$ (to be explained later), $T_i$ is a principle direction for $\partial_0 M$ along $\partial\Sigma_{t_i}$ with positive principle curvature. It follows that the second fundamental form $\mathrm{I\!I}$ of $\partial_0 M$ is positive definite on $\partial\Sigma_{t_1}\cap \partial\Sigma_{t_2}$, contradicting the minimality of $\partial_0 M$ along $\Sigma_{t_1}$.

Next, we show that for $t\in (0,\Phi^{-1}(\cos\alpha))$, there is exactly one element in $\mathcal{E}_t$. Assume that $\Sigma_t,\Sigma_t'\in \mathcal{E}_t$. Since $s_t^- = s_t^+$, we have $|S(\Sigma_t)| = |S(\Sigma_t')|$. If $\partial\Sigma_t\ne \partial\Sigma_t'$, then $\partial\Sigma_t$ must intersect $\partial\Sigma_t'$ transversally somewhere, so we can use the argument above to obtain a contradiction. Hence we have $\partial\Sigma_t = \partial\Sigma_t'$. However, both $\Sigma_t$ and $\Sigma_t'$ are totally geodesic and have the same capillary angle $\Phi(t)$. It follows that $\Sigma_t$ and $\Sigma_t'$ are tangent along $\partial\Sigma_t = \partial\Sigma_t'$. The totally geodesic property forces them to coincide in a neighborhood of $\partial\Sigma_t = \partial\Sigma_t'$, a contradiction.

If there exists $p\in \Int \partial_0 M$ which does not belong to any of $\partial\Sigma_t$, then it's impossible that $s_t^- = s_t^+$ for all $t\in(0,\Phi^{-1}(\cos\alpha)) $. Hence we have
\begin{align}\label{eq: 83}
    \Int\partial_0 M = \bigcup_{t\in (0,\Phi^{-1}(\cos\alpha))}\partial\Sigma_t,
\end{align}
and $\partial_0 M$ is $C^0$ foliated by $\partial\Sigma_t$.  As a consequence, $\Int M$ is $C^0$ foliated by $\Sigma_t (t\in (0,\Phi^{-1}(\cos\alpha))$. To see this, it suffices to show that for each $p\in \Int M$, there exists a $\Sigma_t$ passing through $p$. If this is false, then there exist $\partial_-M\subset\Omega_1\subset\Omega_2\subset M\backslash\partial_+ M$, such that $p\notin \Omega_1$ and $p\in \Omega_2$, and $\Sigma_1 = \Omega_1\cap \mathring{M}$ and $\Sigma_2 = \Omega_2\cap \mathring{M}$ both belong to $\mathcal{F}(M)$. Furthermore, there's no element in $\mathcal{F}(M)$ lying between $\Sigma_1$ and $\Sigma_2$. From \eqref{eq: 83} we know that $\partial\Sigma_1 = \partial\Sigma_2$. It follows that $\Sigma_1$ and $\Sigma_2$ have the same capillary angle, which contradicts the uniqueness of the capillary surface $\Sigma_t\in\mathcal{E}_t$.

\begin{lemma}
    $\{\Sigma_t\}_{t\in (0,\Phi^{-1}(\cos\alpha))}$ is a smooth foliation.
\end{lemma}
\begin{proof}
    For $t\in (0,\Phi^{-1}(\cos\alpha))$, let $Y$ be a smooth vector field along $Y$, such that $Y$ is tangential to $\partial_0 M$ along $\partial\Sigma_t$, and $\langle Y,N\rangle = 1$ along $\Sigma_t$. Let $\phi(x,t)$ be a flow of $Y$. For $f\in C^{2,\beta}(\Sigma_t)$, let $\Sigma(f) = \{\phi(x,f(x))\big|\, x\in\Sigma_t\}\subset M$ be a graph over $\Sigma_t$. Consider $\tilde{\mathcal{F}}: C^{2,\beta}(\Sigma_t)\longrightarrow C^{0,\beta}(\Sigma_t)\times C^{1,\beta}(\partial\Sigma_t)$ given by
    \begin{align*}
        \tilde{\mathcal{F}}(f) = (H(\Sigma(f)), -\langle N(\Sigma(f)),\bar{N}\rangle)
    \end{align*}
    By \cite[Lemma A.2, Lemma A.3]{Li20}, we obtain the linearization operator of $\tilde{\mathcal{F}}$ at $0$ is given by
    \begin{align}\label{eq: 84}
        D_0\tilde{\mathcal{F}}(f) = \left(-\Delta_{\Sigma_t}f-(\Ric (N,N)+|A|^2)f,\,  \sin\theta\frac{\partial f}{\partial\nu}-(\cos\theta) A(N,N)f-\mathrm{I\!I}(\bar{N},\bar{N})f\right )
    \end{align}
    Since $K_{\Sigma_t} = \frac{R_0}{2}$, $R_g = R_0$ on $\Sigma_t$ and $|A|^2 = 0$, we have $\Ric(N,N) = 0$. Note that \eqref{eq: 83}, along with the fact that $\bar{H} = 0$ on $\partial\Sigma_t$ imply $\partial_0 M$ is minimal. Combined with \eqref{eq: 14}, \eqref{eq: 84} reduces to
    \begin{align}\label{eq: 85}
        D_0\tilde{\mathcal{F}}(f) = \left(-\Delta_{\Sigma_t}f,\,  \sin\theta(\frac{\partial f}{\partial\nu}+k_{t}f)\right ),
    \end{align}
    where $k_t$ denotes the geodesic curvature of $\partial\Sigma_t\subset\Sigma_t$ with respect to $\nu$.

    Since $k_t>0,\sin\theta>0$, we obtain $D_0\tilde{\mathcal{F}}$ is an isomorphism. In other word, $\Sigma_t$ is non-degenerate. Hence, the implicit function theorem implies that there is a smooth foliation of capillary minimal surfaces $\Sigma(f_{\tau}), \tau\in (-\epsilon,\epsilon)$. Here $f_{\tau}\in C^{2,\beta}(\Sigma_t)$ depends smoothly on $\tau$, satisfying $f_0 = 0$. Furthermore, the capillary angle $\gamma_\tau$ of $\Sigma(f_\tau)$ satisfies $\frac{d}{d\tau}\big|_{\tau = 0}\gamma_\tau \ne 0$.

    For $t'\in (0,\Phi^{-1}(\cos\alpha))$, when $\Sigma_{t'}$ is sufficiently close to $\Sigma_t$, the implicit function theorem further implies that $\Sigma_{t'}$ coincides with a leaf of $\Sigma(f_\tau)$. Since $\Sigma_{t'}$ has different capillary angle for different $t'$ and $\frac{d}{d\tau}\big|_{\tau = 0}\gamma_\tau \ne 0$, we know that $\{\Sigma_{t'}\}_{t'\in (0,\Phi^{-1}(\cos\alpha))}$ coincides with $\{\Sigma(f_\tau)\}_{\tau\in (-\epsilon,\epsilon)}$ locally. This concludes the proof that $\{\Sigma_{t'}\}_{t'\in (0,\Phi^{-1}(\cos\alpha))}$ is a smooth foliation.
\end{proof}

\begin{proof}[Proof of Theorem \ref{thm: 2 systole estimate, rigidity} (rigidity part)]
We reparametrize the foliation $\{\Sigma_t\}_{t\in (0,\Phi^{-1}(\cos\alpha))}$ as $\{\Sigma(s)\}_{s\in (0,|\partial_0 M|)}$ by $s = |S(\Sigma_t)|$. By definition, we have $|\Sigma(s)| = I(s)$. Let $\theta(s)$ and $k(s)$ denote the capillary angle of $\Sigma(s)$ and the geodesic curvature of $\partial\Sigma(s)\subset\Sigma(s)$ with respect to $\nu$. 

For fixed $\Sigma = \Sigma(s_0)$, let $v$ be the lapse function defined on $\Sigma$ induced by the foliation $\{\Sigma(s)\}_{s\in (0,|\partial_0 M|)}$. By \cite[Lemma A.2, Lemma A.3]{Li20} we obtain
\begin{align*}
    &\Delta_{\Sigma} v = 0,\quad\text{in } \Sigma,\\
    &\frac{d}{ds}\big|_{s = s_0}\cos\theta(s) = \sin\theta(s_0)(\frac{\partial v}{\partial\nu}+kv)\quad\text{on }\partial\Sigma.
\end{align*}
Let $\bar v$ denote the average of $v$ on $\partial\Sigma$. By integration by parts we have $k\bar{v} = -\theta'(s_0)$, so we have $\frac{\partial v}{\partial\nu}+k(v-\bar v) = 0$. Applying integration by parts once more, we obtain
\begin{align*}
    \int_{\Sigma}|\nabla_\Sigma (v-\bar v)|^2 = \int_{\partial\Sigma}(v-\bar{v})\frac{\partial v}{\partial \nu} = -k\int_{\partial\Sigma}(v-\bar v)^2.
\end{align*}
Since $k>0$, we conclude that $v$ is a constant. Therefore the distance between nearby leaves in the family $\{\Sigma(s)\}_{s\in (0,|\partial_0 M|)}$ is constant along each leaf.

By the same argument as \cite[Corollary 3.7]{BBN10}, the unit normal $N$ of the foliation  $\{\Sigma(s)\}_{s\in (0,|\partial_0 M|)}$ is a parallel vector field. In particular, each point in $(M,g)$ has a neighborhood which is isometric to a Riemannian product. Since $\Sigma(s)$ is isometric to a geodesic disk in the round sphere of constant curvature $\frac{R_0}{2}$, $(M,g)$ can be isometrically embedded in $S^2\times \mathbf{R}$. 

Now we can show that $(M,g)$ is rotationally symmetric. Let $z$ be the parameter of the $\mathbf{R}$ factor of $S^2\times \mathbf{R}$, and reparametrize $\{\Sigma(s)\}_{s\in (0,|\partial_0 M|)}$ as $\{\Sigma_z\}_{z_-<z<z_+}$. Without loss of generality, we assume $R_0 = 2$ and $S^2\subset\mathbf{R}^3$ is the unit sphere. Recall $\Sigma_z\subset S^2$ is a geodesic disk. Let $a(z)\in S^2$, which is also viewed as a unit vector in $\mathbf{R}^3$, be the center of $\Sigma_z$. Then $\partial_0 M$ can be parametrize as
\begin{align*}
    X(z,\varphi) = &\left( (a(z)\cos r(z)+q(\varphi,z)\sin r(z), \,z\right)\in S^2\times \mathbf{R}\subset \mathbf{R}^3\times \mathbf{R},\\
    &z\in (z_-,z_+)\qquad  \varphi\in [0,2\pi),
\end{align*}
where $q(\varphi,z)\in \mathbf{R}^3$ is a unit vector perpendicular to $a(z)$, and $r(z)$ is the radius of the geodesic disk $\Sigma_z$. In fact, $q(z,\varphi)$ can be expressed as
\begin{align}\label{eq: 91}
    q(z,\varphi) = b(z)\cos\varphi+c(z)\sin\varphi,
\end{align}
where $b(z), c(z)$ and $a(z)$ form an orthonormal basis of $\mathbf{R}^3$. The outer normal of $\partial\Sigma_z\subset\Sigma_z$ is given by
\begin{align*}
    \nu = -\sin r(z) a(z)+\cos r(z)q(z,\varphi).
\end{align*}
Since the dihedral angle between $\Sigma_z$ and $\partial_0 M$ is a constant $\theta\in (0,\frac{\pi}{2})$ along $\partial\Sigma_z$, the normal vector $\bar N$ of $\partial_0 M$ must be of the form
\begin{align*}
    \bar{N} = (\cos\theta)\nu +(\sin\theta)\partial_z.
\end{align*}
By direct computation we have
\begin{align*}
    X_z = a'(z)\cos r(z)+q_z(\varphi,z)\sin r(z)+r'(z)\nu+\partial_z.
\end{align*}
Note that
\begin{align*}
    &\langle a'(z)\cos r(z)+q_z(\varphi,z)\sin r(z), \nu\rangle\\
    = &\langle a'(z)\cos r(z)+q_z(\varphi,z)\sin r(z), -\sin r(z) a(z)+\cos r(z)q(z,\varphi)\rangle\\
    =& a'(z)q_z(z,\varphi),
\end{align*}
where we have used $|a(z)| = |q(\varphi,z)| = 1$. Using that $\langle X_z,\bar N\rangle = 0$, we conclude that
\begin{align*}
    0 = \cos{\theta}\,(r'(z)+a'(z)q_z(z,\varphi))+\sin\theta = 0.
\end{align*}
Letting $\varphi$ vary and using \eqref{eq: 91}, we obtain $a'(z)\equiv 0$. Hence the leaves in the family $\{\Sigma_z\}_{z\in (z_-,z_+)}$ have a common center in $S^2$. This proves the rotational symmetry of $(M,g)$. Once the rotational symmetry is established, the minimality of $\partial_0 M$ allows us to solve the resulting ODE as in Section 4.1. It follows that $(M,g)$ is a standard cylindrical block.
\end{proof}

\appendix
\section{Elementary inequalities}

As an elementary fact, the function $(1+\cos \alpha)(\frac{\pi^2}{4}-\alpha^2)-4\cos \alpha$ has a unique zero in $(0,\frac{\pi}{2})$, located at $\alpha_0\approx 0.35\pi$. Consequently, for $\alpha\in (0,\frac{\pi}{2})$, the inequality
\begin{align}\label{eq: a1}
        (1+\cos \alpha)(\frac{\pi^2}{4}-\alpha^2)>4\cos \alpha
    \end{align} 
holds if and only if $0<\alpha<\alpha_0$. We have the following elementary lemma.

\begin{lemma}\label{lem: A1}
       Fix $\alpha\in (0,\alpha_0)$. Then there exist $\delta>0$ and $\mu_0>0$ such that for every $x\in (\alpha-\delta,\alpha+\delta)$ and every $\mu\in (0,\mu_0)$, it holds
    \begin{align*}
        (1-2\mu^2)(1-\cos x)>1-\sqrt{1-\sin^2 x\left(\frac{\cosh(\mu x)}{\cosh(\frac{\pi}{2}\mu)}\right)^2}.
    \end{align*}
\end{lemma}

\begin{proof}[Proof of Lemma \ref{lem: A1}]
    It suffices to show
    \begin{align}\label{eq: a2}
        1-\sin^2 x\left(\frac{\cosh(\mu x)}{\cosh(\frac{\pi}{2}\mu)}\right)^2> [1-(1-2\mu^2)(1-\cos x)]^2.
    \end{align}
    Subtracting the two sides of \eqref{eq: a2} and using the Taylor expansion on $\mu$, we obtain
    \begin{align*}
        &1-\sin^2 x\left(\frac{\cosh(\mu x)}{\cosh(\frac{\pi}{2}\mu)}\right)^2
        - [1-(1-2\mu^2)(1-\cos x)]^2\\
        =&\ \mu^2(1-\cos x)\left[(1+\cos x)\left(\frac{\pi^2}{4}-x^2\right)-4\cos x\right]+O(\mu^4).
    \end{align*}
    Since $\alpha\in (0,\alpha_0)$, by continuity there exists $\delta>0,\epsilon_0>0$, such that
    \begin{align*}
        (1+\cos x)\left(\frac{\pi^2}{4}-x^2\right)-4\cos x>\epsilon_0
    \end{align*}
    for all $x\in (\alpha-\delta,\alpha+\delta)$. Moreover, the $O(\mu^4)$ terms above are locally uniform for $x\in (\alpha-\delta,\alpha+\delta)$. Hence, for $\mu_0>0$ sufficiently small, the right-hand side is positive for all $x\in (\alpha-\delta,\alpha+\delta)$ and all $\mu\in (0,\mu_0)$. This proves the lemma.
\end{proof}

\begin{lemma}\label{lem: A2}
Assume $0<\alpha<\alpha_0$ and let
\begin{align*}
    &\mu > 0, \qquad k = \frac{1}{\sqrt{1-2\mu^2}}, \qquad
    \lambda = \mu k, \qquad b = \frac{\pi}{2k},\\
    &u(r) = \frac{\sin kr}{k}, \qquad v(r) = \cosh(\lambda r), \qquad r\in (0,b).
\end{align*}
Then for $\mu$ sufficiently small, there exists $a\in (0,b)$ satisfying
\begin{align}\label{eq: a3}
    u(b)v(b)\sin\alpha = u(a)v(a),
\end{align}
\begin{align}\label{eq: a4}
    \int_0^a u(r)dr > 1-\cos\alpha,
\end{align}
and
\begin{align}\label{eq: a6}
    \frac{u''}{u}+\frac{v''}{v}+\frac{u'v'}{uv}\le -1 \text{ for } r\in (0,b).
\end{align}
\end{lemma}
\begin{proof}
    Since $kb=\frac{\pi}{2}$, \eqref{eq: a3} is equivalent to
    \begin{align}\label{eq: a5}
        \sin ka\cosh \lambda a = \sin\alpha\cosh \lambda b.
    \end{align}
    Note the left hand side is increasing in $a$ for $a\in [0,b]$, and is greater than the right hand side when $a = b$. This implies there exists a unique $a\in (0,b)$ such that \eqref{eq: a3} holds.
    
    Let $x = ka$. When $\mu\to 0$, we have $k\to 1$ and $\lambda\to 0$. Combined with \eqref{eq: a3} we have $x = ka\to\alpha$. Let $\delta>0$ be the number given by Lemma \ref{lem: A1}, then there exists $\mu_1>0$, such that $x\in (\alpha-\delta,\alpha+\delta)$ whenever $0<\mu<\mu_1$. We calculate
    \begin{align*}
        \int_0^au(r)dr = \frac{1}{k^2}(1-\cos ka) = (1-2\mu^2)(1-\cos x).
    \end{align*}
    Using \eqref{eq: a5} we obtain
    \begin{align*}
        &1-\cos\alpha  = 1- \sqrt{1-\sin^2\alpha}\\
       =& 1-\sqrt{1-\sin^2ka \left(\frac{\cosh \lambda a}{\cosh \lambda b}\right)^2} = 1-\sqrt{1-\sin^2x \left(\frac{\cosh \mu x}{\cosh \frac{\pi}{2}\mu}\right)^2}
    \end{align*}
    Taking $\mu<\min\{\mu_0,\mu_1\}$, we see from Lemma \ref{lem: A1} that \eqref{eq: a4} holds. 

    It remains to verify \eqref{eq: a6}. We compute
\[
\frac{u''}{u}=-k^2,
\qquad
\frac{v''}{v}=\lambda^2,
\qquad
\frac{u'v'}{uv}
=
k\lambda\cot(k r)\,\tanh(\lambda r).
\]
Write
\[
x=kr,\qquad \mu=\frac{\lambda}{k}.
\]
Since $0<kr<kb = \frac{\pi}{2}$, we have $kr\cot kr \le 1$ and $\tanh\lambda r\le \lambda r$. We obtain
\begin{align*}
    k\lambda\cot(k r)\,\tanh(\lambda r) \le \lambda ^2
\end{align*}
Therefore
\[
\frac{u''}{u}+\frac{v''}{v}+\frac{u'v'}{uv}
\le
-k^2+\lambda^2+\lambda^2
=
-k^2(1-2\mu^2)
=
-1.
\]
This proves the lemma.
\end{proof}

\section{A version of the Whitney extension theorem on polyhedral domains}

We first recall the following $C^{1,1}$ version of the Whitney extension theorem:

\begin{lemma}\label{lem: vector valued C11 Whitney extension}(\cite{Whitney1934},\cite[p.1]{LeGruyer2009})
Let $E\subset \mathbf{R}^n$ be closed. Suppose that
$f:E\to \mathbf{R}^m$ and
$A:E\to \operatorname{Hom}(\mathbf{R}^n,\mathbf{R}^m)$ satisfy, for some $M>0$,
\begin{align}\label{eq: 68'}
    |A_a-A_b|\le M|a-b|,
\end{align}
and
\begin{align}\label{eq: 69'}
     |f(a)-f(b)-A_b(a-b)|\le M|a-b|^2
\end{align}
for all $a,b\in E$. Then there exists
$F\in C^{1,1}(\mathbf{R}^n,\mathbf{R}^m)$ such that
\[
    F|_E=f,\qquad DF|_E=A.
\]
Moreover, $F$ is real analytic in $\mathbf{R}^n\backslash E$.
\end{lemma}

In this appendix, we prove the following improved version of the Whitney extension theorem for polyhedral domains and its variants.

\begin{lemma}\label{lem: C11 Whitney extension smooth away from corner}
Let $\mathcal{C}\subset \mathbf{R}^k$ be a simplicial polyhedral cone with $k$ facets $G_1,G_2,\dots, G_k$ and vertex $\mathbf{0}$. Suppose that $A: \partial \mathcal{C}\longrightarrow \mathrm{Hom}(\mathbf{R}^k,\mathbf{R}^k)$ satisfies, for some $M>0$,
\begin{align}\label{eq: 68}
    |A_a-A_b|\le M|a-b|,
\end{align}
and
\begin{align}\label{eq: 69}
     |a-b-A_b(a-b)|\le M|a-b|^2
\end{align}
for all $a,b\in \partial \mathcal{C}$. Moreover, assume $A$ is smooth on each open facet. Then there exists $F\in C^\infty (\mathcal{C}\setminus \partial_{\cor}\mathcal{C})\cap C^{1,1}(\mathcal{C})$ such that
\begin{align*}
    F|_{\partial \mathcal{C}} = \Id,\qquad DF|_{\partial \mathcal{C}} = A.
\end{align*}
\end{lemma}

\begin{proof}
    We first point out some basic properties satisfied by $A$. When $a$ and $b$ lie in the same facet $G_i$, \eqref{eq: 69} implies that the restriction of $A_b$ on $TG_i$ is the identity. Furthermore, by letting $b = 0$ in \eqref{eq: 69} and letting $|a|\to 0$, we obtain
    \begin{align*}
        A_{\mathbf{0}} = \Id.
    \end{align*}
    
    Note that
    \begin{align*}
        &f: \partial \mathcal{C}\longrightarrow \mathbf{R}^k,\qquad f(x) = x,\\
        &A: \partial \mathcal{C}\longrightarrow \mathrm{Hom}(\mathbf{R}^k,\mathbf{R}^k)
    \end{align*}
    satisfy all conditions in Lemma \ref{lem: vector valued C11 Whitney extension}. Hence, there exists $W\in C^\infty(\mathcal{C}\backslash\partial_{\cor}C)\cap  C^{1,1}(\mathcal{C})$, satisfying
    \begin{align*}
        W|_{\partial\mathcal{C}} = \Id,\qquad DW|_{\partial\mathcal{C}} = A.
    \end{align*}

The remaining task is to modify $W$ near each open facet $\Int G_i$, $i=1,2,\dots,k$, so that the resulting map $F$ is smooth up to each $\Int G_i$ while remaining $C^{1,1}$ up to the vertex $\mathbf{R}^{n-k}\times \{\mathbf{0}\}$.

    For simplicity of exposition, we primarily focus on the case when $k=2$. In this case, $\mathcal{C}\subset\mathbf{R}^2$ has two edges, $G_1$ and $G_2$. By possibly rotating $\mathcal{C}$, we can specify a Euclidean coordinate $(s,t)$ on $\mathbf{R}^2$, such that $G_i = \{(s,0),s\ge 0\}$. We define
    \begin{align}\label{eq: 116}
        V_i(s,t) = (s,0)+tA_{(s,0)}(\partial_t).
    \end{align}
    From the smoothness of $A$ on $\Int G_i$, $V_i$ is also smooth in a neighborhood of $\Int G_i$. It's easy to see
    \begin{align}\label{eq: 94}
        V_i = W = \Id\text{ on } G_i,\qquad DV_i = DW = A\text{ on } G_i.
    \end{align}
    In other word, $V_i$ and $W$ have the same $1$-jet along $G_i$.

    Let $t_j>0\, (j=1,2,\dots)$ be a sequence of numbers to be specified. We define
    \begin{align*}
        &\Omega_{i,j} = \{(s,t),\quad 2^{-j-1}<s<2^{-j+1}, 0\le t<t_j\}\\
        &\tilde{\Omega}_{i,j} = \{(s,t),\quad 2^{-j-1}-t_j<s<2^{-j+1}+t_j, 0\le t<2t_j\}.
    \end{align*}
    Clearly, $\Omega_{i,j} \, (j=1,2,\dots)$ cover $\Int G_i$. We also define a smooth cut off fuction $\chi_{i,j}$, which equals $1$ in $\Omega_{i,j}$ and equals $0$ outside $\tilde{\Omega}_j$, such that
    \begin{align}\label{eq: 92}
        0<\chi_{i,j}<1,\qquad |D\chi_{i,j}|<t_j^{-1},\qquad |D^2\chi_{i,j}|<t_j^{-2}.
    \end{align}
    Let $\alpha_0$ denote the angle of $C$ at its vertex $\mathbf{0}$. By choosing
    \begin{align}\label{eq: 93}
        t_j\in (2^{-j-4}\tan\frac{\alpha_0}{2},2^{-j-3}\tan\frac{\alpha_0}{2})
    \end{align}
    one ensures that $\tilde{\Omega}_{i,j}$ is contained in $\mathcal{C}$, and $\tilde\Omega_{1,j_1}\cap \tilde\Omega_{2,j_2}=\emptyset$ for all $j_1,j_2$. Moreover, every point $x\in \mathcal{C}$ is contained in at most three of the sets $\tilde\Omega_{i,j}$.

    We define
    \begin{align}\label{eq: 95}
        F = W+\sum_{i=1}^2\sum_{j=1}^\infty \chi_{i,j}(V_i-W).
    \end{align}
    For each point $x\in \mathcal{C}$, at most two terms in the above sum is non-zero. We have $F = V_i$ along $G_i$, which shows that $F$ is smooth in $\mathcal{C}\backslash\{\mathbf{0}\}$.

    It remains to check that $F$ is $C^{1,1}$ at $\mathbf{0}$. Since both $V_i$ and $W$ are $C^{1,1}$, $\Lip (D(V_i-W))$ is uniformly bounded. Using \eqref{eq: 94}\eqref{eq: 93} and the Taylor expansion for $V_i-W$ at points in $\tilde{\Omega}_{i,j}\cap G_i$, we obtain
    \begin{align*}
        |D(V_i-W)|<Ct_j,\qquad |V_i-W|<Ct_j^2
    \end{align*}
    in $\tilde{\Omega}_{i,j}$. Combined with \eqref{eq: 92}, we obtain
    \begin{align*}
    &|D(\chi_{i,j}(V_i-W))|<Ct_j,\\
    &\Lip (D(\chi_{i,j}(V_i-W)))<C.
    \end{align*}
    Combined with the fact that for each point $x\in \mathcal{C}$, at most two terms in the sum in \eqref{eq: 95} is non-zero, it is direct to check $DF_{\mathbf{0}} = \Id$, and $DF$ is Lipschitz in $\mathcal{C}$. This concludes the proof when $k=2$.

    For general $k$, the proof is the same except for a slight modification on the choice of $\Omega_{i,j}$ and $\tilde{\Omega}_{i,j}$. In this case our facet $G_i$ is modeled on
    \begin{align*}
        G_i = \{(s_1,s_2,\dots,s_{k-1},t),\, s_l\ge 0, l=1,2,\dots,k-1;\, t=0\}.
    \end{align*}
    We define
    \begin{align*}
        &\Omega_{i,j_1j_2,\dots j_{k-1}} = \{(s_1,s_2,\dots,s_{k-1},t),\,2^{-j_l-1}<s_l<2^{-j_l+1}, 0\le t<t_j\},\\
        &\tilde{\Omega}_{i,j_1j_2,\dots j_{k-1}} = \{(s_1,s_2,\dots,s_{k-1},t),\,2^{-j_l-1}-t_j<s_l<2^{-j_l+1}+t_j, 0\le t<2t_j\}.
    \end{align*}
By choosing $t_j\in (c_1 2^{-j},c_2 2^{-j})$, where $c_1$ and $c_2$ depend only on the dihedral angles between the facets of $\mathcal{C}$, we may ensure that $\tilde{\Omega}_{i,j_1j_2,\dots j_{k-1}}\subset\mathcal{C}$ and $\tilde{\Omega}_{i,j_1j_2,\dots j_{k-1}}\cap \tilde{\Omega}_{i',j_1'j_2',\dots j_{k-1}'} = \emptyset$ whenever $i\ne i'$. Moreover, if $c_2$ is sufficiently small, then each point
$x\in \mathcal{C}$ belongs to at most $3^{k-1}$ of the sets $\tilde{\Omega}_{i,j_1j_2,\dots,j_{k-1}}$. We can then argue exactly as in the case $k=2$.
    
\end{proof}

By the same argument, we also have the following generalized version of Lemma \ref{lem: C11 Whitney extension smooth away from corner}.
\begin{lemma}\label{lem: C11 Whitney extension smooth away from corner 2}
Let $\mathcal{C}\subset \mathbf{R}^k$ be a simplicial polyhedral cone with $k$ facets $G_1,G_2,\dots, G_k$ and vertex $\mathbf{0}$. Let $B_1^{n-k}(\mathbf{0})\subset \mathbf{R}^{n-k}$ be the unit ball. Suppose that
\begin{align*}
    A:  B_1^{n-k}(\mathbf{0})\times\partial\mathcal{C}\longrightarrow \mathrm{Hom}(T(B_1^{n-k}(\mathbf{0})\times \mathbf{R}^k), T(B_1^{n-k}(\mathbf{0})\times \mathbf{R}^k)) = \mathrm{Hom}(\mathbf{R}^n,\mathbf{R}^n)
\end{align*}
satisfies, for some $M>0$,
\begin{align}\label{eq: 114}
    |A_a-A_b|\le M|a-b|,
\end{align}
and
\begin{align}\label{eq: 115}
     |a-b-A_b(a-b)|\le M|a-b|^2
\end{align}
for all $a,b\in  B_1^{n-k}(\mathbf{0})\times\partial\mathcal{C}$. Moreover, assume $A$ is smooth on each open facet. Then there exists $F\in C^\infty (B_1^{n-k}(\mathbf{0})\times(\mathcal{C}\setminus \partial_{\cor}\mathcal{C}))\cap C^{1,1}(B_1^{n-k}(\mathbf{0})\times\mathcal{C})$ such that
\begin{align*}
    F|_{B_1^{n-k}(\mathbf{0})\times\partial \mathcal{C}} = \Id,\qquad DF|_{B_1^{n-k}(\mathbf{0})\times\partial \mathcal{C}} = A.
\end{align*}
\end{lemma}

Using the construction \eqref{eq: 116} and a simple cut-off function argument, we have the following result.
\begin{lemma}\label{lem: an elementary construction for the smooth model near the boundary}
    Consider the upper half Euclidean space $\mathbf{R}^n_+$ with boundary $\partial \mathbf{R}^n_+ = \mathbf{R}^{n-1}\times \{0\}$. Let
    \begin{align*}
        A:\mathbf{R}^{n-1}\times \{0\}\longrightarrow \mathrm{Hom}(\mathbf{R}^n,\mathbf{R}^n)
    \end{align*}
    be a smooth map. Assume that $A$ satisfies $A_x = \Id$ when restricted on the linear subspace $\mathbf{R}^{n-1}\subset\mathbf{R}^n$.
    
    Suppose $U$ is an open subset of $\overline{\mathbf{R}^n_+}$, and $F_0:U\longrightarrow \overline{\mathbf{R}^n_+}$ is a smooth map, such that
    \begin{itemize}
        \item $F_0(x) = x$ for $x\in (\mathbf{R}^{n-1}\times \{0\})\cap U$;
        \item $DF = A$ on $\mathbf{R}^{n-1}\times \{0\}$.
    \end{itemize}
    Then for any $U_0\subset U$ satisfying $\bar U_0\subset U$, there exists a smooth map $V$ defined on $\mathbf{R}^{n-1}\times [0,\epsilon)$ for some $\epsilon>0$, such that
    \begin{align*}
        V|_{\mathbf{R}^{n-1}\times \{0\}} = \Id,\qquad DV|_{\mathbf{R}^{n-1}\times \{0\}} = A,
    \end{align*}
    and $V = F_0$ in $(\mathbf{R}^{n-1}\times [0,\epsilon))\cap U_0$.
\end{lemma}
\begin{proof}
    Let $\chi\in C^\infty(\mathbf{R}_n^+)$ be a cut off function satisfying $\chi = 1$ in $U_0$ and $\chi = 0$ outside $U$. We use $x=(x',t)\in \mathbf{R}^{n-1}\times [0,\infty)$ to represent points in $\overline{\mathbf{R}^n_+}$. Define
    \begin{align*}
        V_0(x',t) = (x',0)+tA_{(x',0)}(\partial_t),
    \end{align*}
    then $V_0(x) = x$ and $DV_0 = A$ on $\mathbf{R}^{n-1}\times \{0\}$. Because $A$ is smooth, $V_0$ is also smooth. It is then direct to check that the map $V = \chi F_0+(1-\chi)V_0$ satisfies all desired properties.
\end{proof}

As a consequence, we have the following relative version of Lemma \ref{lem: C11 Whitney extension smooth away from corner 2}
\begin{lemma}\label{lem: C11 Whitney extension smooth away from corner 2 relative version}
Let $\mathcal{C}$ and $A$ be as in Lemma \ref{lem: C11 Whitney extension smooth away from corner 2}. Let $U$ be an open subset of $B_1^{n-k}(\mathbf{0})$, and let $F_0: U\times \mathcal{C}\longrightarrow B_1^{n-k}(\mathbf{0})\times \mathcal{C}$ be a $C^{1,1}$ map which is smooth away from $B_1^{n-k}(\mathbf{0})\times \partial_{\mathrm{cor}}\mathcal{C}$, satisfying
\begin{itemize}
\item $F_0(x) = x$ for $x\in U\times \partial\mathcal{C}$;
    \item $DF_0 = A$ on $U\times \partial\mathcal{C}$.
\end{itemize}
 Then for any $U_0\subset U$ satisfying $\bar U_0\subset U$, there exists $F\in C^\infty (B_1^{n-k}(\mathbf{0})\times(\mathcal{C}\setminus \partial_{\cor}\mathcal{C})\cap C^{1,1}(B_1^{n-k}(\mathbf{0})\times\mathcal{C})$ such that
\begin{align*}
    F|_{B_1^{n-k}(\mathbf{0})\times\partial \mathcal{C}} = \Id,\qquad DF|_{B_1^{n-k}(\mathbf{0})\times\partial \mathcal{C}} = A,
\end{align*}
and $F = F_0$ on $U\times \mathcal{C}$.
\end{lemma}
\begin{proof}
    The proof is the same as that of Lemma \ref{lem: C11 Whitney extension smooth away from corner}. We first apply the Whitney extension theorem \ref{lem: vector valued C11 Whitney extension} to $E = B^{n-k}_1(\mathbf{0})\times \partial\mathcal{C}\cup U$ to obtain a $C^{1,1}$ map $W$ on $E$, which equals the identity on $B_1^{n-k}(\mathbf{0})\times\partial\mathcal{C}$ and equals $F_0$ on $U\times \mathcal{C}$, such that $DW = A$ on $B_1^{n-k}(\mathbf{0})\times \partial\mathcal{C}$ . Then, we replace $W$ by a smooth model map $V_i$ in a neighborhood of each facet of the cone. The smooth model map $V_i$ is now constructed as in Lemma \ref{lem: an elementary construction for the smooth model near the boundary}. The rest of the proof runs in the same way as that of Lemma \ref{lem: vector valued C11 Whitney extension}. 
\end{proof}

Finally, we record some lemmas concerning extension of maps defined on the boundary neighborhood to the interior. The following lemma is elementary.
\begin{lemma}\label{lem: extension 2}
    Let $(M^n,g)$ be a compact Riemannian manifold with corners. Let $U$ be a neighborhood of $\partial M$, and let
\begin{align*}
    \Phi:U\longrightarrow M
\end{align*}
be a $C^{1,1}$ embedding which is smooth away from $\partial M$. Assume that
\begin{align*}
    \Phi=\mathrm{Id}\quad\text{on }\partial M
\end{align*}
and that $\Phi$ is sufficiently $C^1$-close to the inclusion map on $U$. Then, after possibly shrinking $U$, $\Phi$ extends to a $C^{1,1}$ self-diffeomorphism
\begin{align*}
    \widetilde{\Phi}:M\longrightarrow M
\end{align*}
which is smooth away from $\partial M$ and agrees with $\Phi$ near $\partial M$. Moreover, $\widetilde\Phi$ can also be chosen to be sufficiently close to the identity map.
\end{lemma}

As a consequence, we have the following lemma.
\begin{lemma}\label{lem: extension}
Let $(M^n,g)$ be a compact Riemannian manifold with corners. Let $U$ be a neighborhood of $\partial_{\cor}M$, and let
\begin{align*}
    \Phi:U\longrightarrow M
\end{align*}
be a $C^{1,1}$ embedding which is smooth away from $\partial_{\cor}M$. Assume that
\begin{align*}
    \Phi=\mathrm{Id}\quad\text{on }\partial_{\cor}M
\end{align*}
and that $\Phi$ is sufficiently $C^1$-close to the inclusion map on $U$. Then, after possibly shrinking $U$, $\Phi$ extends to a $C^{1,1}$ self-diffeomorphism
\begin{align*}
    \widetilde{\Phi}:M\longrightarrow M
\end{align*}
which is smooth away from $\partial_{\cor}M$ and agrees with $\Phi$ near $\partial_{\cor}M$. Moreover, $\widetilde{\Phi}$ can also be chosen to be sufficiently close to the identity map.
\end{lemma}
\begin{proof}
    We first apply Lemma \ref{lem: extension 2} to each facets of $M$ to obtain a extension $\Phi_1$ of $\Phi$ to $U\cup\partial M$. Then, we extend $\Phi_1$ to a $C^{1,1}$ embedding $\Phi_2$ defined on a neighborhood $\mathcal{U} $ of $\partial M$. For any facet $F_\lambda$ of $M$, Lemma \ref{lem: extension 2} ensures $\Phi_1|_{F_\lambda}$ can be sufficiently close to $\Id|_{F_\lambda}$. Hence, we can extend $\Phi_2$ to the diffeomorphism of the whole $M$ in the same spirit of Lemma \ref{lem: extension 2}.
\end{proof}

\bibliographystyle{alpha}

\bibliography{Positive}

\end{document}